\documentclass[a4paper]{amsart}

\usepackage{etex}
\usepackage{bbm}
\usepackage{pbox}

\newcommand{\bk}{\mathbbm{k}}

\newcommand{\cA}{\mathcal{A}}\newcommand{\cB}{\mathcal{B}}
\newcommand{\cC}{\mathcal{C}}

\newcommand{\cG}{\mathcal{G}}\newcommand{\cH}{\mathcal{H}}

\newcommand{\cL}{\mathcal{L}}
\newcommand{\cM}{\mathcal{M}}
\newcommand{\cO}{\mathcal{O}}\newcommand{\cP}{\mathcal{P}}

\newcommand{\cV}{\mathcal{V}}

\newcommand{\bC}{\mathbb{C}}

\newcommand{\bQ}{\mathbb{Q}}\newcommand{\bR}{\mathbb{R}}

\newcommand{\bZ}{\mathbb{Z}}

\newcommand{\gS}{\mathbf{S}}

\usepackage{lmodern}
\usepackage{booktabs}
\usepackage[dvipsnames,svgnames,x11names,hyperref]{xcolor}
\usepackage{mathtools,url,graphicx,verbatim,amssymb,enumerate,stmaryrd,nicefrac}
\usepackage[pagebackref,colorlinks,citecolor=Mahogany,linkcolor=Mahogany,urlcolor=Mahogany,filecolor=Mahogany]{hyperref}
\usepackage{microtype}
\usepackage[all,cmtip]{xy}
\usepackage[margin=1.5in]{geometry}
\usepackage{subcaption}

\usepackage{tikz, tikz-cd}
\usetikzlibrary{matrix,calc,positioning,arrows,decorations.pathreplacing,decorations.markings,patterns}
\tikzset{->-/.style={decoration={
			markings,
			mark=at position #1 with {\arrow{>}}},postaction={decorate}}}
\tikzset{-<-/.style={decoration={
					markings,
					mark=at position #1 with {\arrow{<}}},postaction={decorate}}}

\newtheorem{theorem}{Theorem}[section]
\newtheorem{lemma}[theorem]{Lemma}
\newtheorem{proposition}[theorem]{Proposition}
\newtheorem{corollary}[theorem]{Corollary}
\newtheorem*{corollary*}{Corollary}

\newtheorem{atheorem}{Theorem}

\theoremstyle{definition}
\newtheorem{definition}[theorem]{Definition}

\newtheorem{assumption}[theorem]{Assumption}

\theoremstyle{remark}

\newtheorem{remark}[theorem]{Remark}
\newtheorem{claim}[theorem]{Claim}

\numberwithin{equation}{section}

\newcommand{\cat}[1]{\mathsf{#1}}
\newcommand{\mr}[1]{{\rm #1}}

\newcommand{\fS}{\mathfrak{S}}

\newcommand{\diff}{\mr{Diff}}

\newcommand{\lra}{\longrightarrow}

\newcommand{\Sp}{\mathbf{Sp}}
\newcommand{\OO}{\mathbf{O}}

\newcommand{\GG}{\mathbf{G}}
\newcommand{\Hom}{\mr{Hom}}

\DeclareMathOperator*{\colim}{colim}

\title{On the cohomology of Torelli groups}

 \author{Alexander Kupers}
 \email{kupers@math.harvard.edu}
 \address{Department of Mathematics \\
 	 One Oxford Street \\
 	 Cambridge MA, 02138 \\USA}

 \author{Oscar Randal-Williams}
 \email{o.randal-williams@dpmms.cam.ac.uk}
 \address{Centre for Mathematical Sciences\\
 Wilberforce Road\\
 Cambridge CB3 0WB\\
 UK}

\dedicatory{Dedicated to Shigeyuki Morita}

\subjclass[2010]{55R40, 11F75, 57S05, 18D10, 20G05}
\keywords{Cohomology of diffeomorphism groups, Torelli groups, cohomology of arithmetic groups, Miller-Morita-Mumford classes}

\begin{document}

\begin{abstract}
We completely describe the algebraic part of the rational cohomology of the Torelli groups of the manifolds $\#^g S^n \times S^n$ relative to a disc in a stable range, for $2n \geq 6$. Our calculation is also valid for $2n=2$ assuming that the rational cohomology groups of these Torelli groups are finite dimensional in a stable range.
\end{abstract}

\maketitle

\tableofcontents

\section{Introduction}
In the study of the cohomology of the mapping class group $\Gamma_g$ of the genus $g$ surface $\Sigma_g$, an important role is played by its normal subgroup $T_g$, the \emph{Torelli group}, consisting of those diffeomorphisms which act trivially on $H_1(\Sigma_g;\bZ)$. This is the kernel of the (surjective) homomorphism $\Gamma_g \to \mathrm{Sp}_{2g}(\bZ)$ which sends a diffeomorphism to the induced map on $H_1(\Sigma_g;\bZ)$, and so is equipped with an outer action of $\mathrm{Sp}_{2g}(\bZ)$. It is a fundamental problem to study the cohomology $H^*(T_g;\bQ)$ and its structure as a $\mathrm{Sp}_{2g}(\bZ)$-representation, cf.\ \cite{JohnsonAb, MoritaLift, HainTorelli, SakasaiThirdTorelli, BoldsenDollerup, ChurchFarbAJ, MPW}.

In this paper we will study the generalisation of this problem to all even dimensions $2n$, replacing the surface of genus $g$ by its $2n$-dimensional analogue $W_g \coloneqq \#^g S^n \times S^n$. Most of our results will be for $2n \geq 6$, though our results are also valid in the classical case $2n=2$ conditional on the conjecture that $H^*(T_g;\bQ)$ is finite-dimensional in a range of degrees for large enough $g$.

Let us explain the variant of the Torelli group we consider and the form of our main result. Let $\diff(W_g, D^{2n})$ denote the topological group of diffeomorphisms of $W_g$ which are equal to the identity near a specified disc $D^{2n} \subset W_g$, equipped with the $C^\infty$-topology. This acts on the middle homology group $H_n(W_g;\bZ)$, and the \emph{Torelli group}
\[\mr{Tor}(W_g, D^{2n})\leq \diff(W_g,  D^{2n})\]
is the normal subgroup of those diffeomorphisms which act trivially on $H_n(W_g;\bZ)$. In the case $2n=2$ this differs from the Torelli group $T_g$ described above, as we only consider those diffeomorphisms fixing a disc. However, the difference between the cohomology of these two groups is mild (and described in Section \ref{sec:variants}) and it is convenient to work with a fixed disc.

The automorphisms of the middle homology of $W_g$ which may be realised by diffeomorphisms are constrained: they must at least respect the intersection form, which is $(-1)^n$-symmetric and nondegenerate, giving a homomorphism
\[\alpha_g \colon \diff(W_g, D^{2n}) \lra G_g \coloneqq \begin{cases}
\mr{Sp}_{2g}(\bZ) & \text{ if $n$ is odd,}\\
\mr{O}_{g,g}(\bZ) & \text{ if $n$ is even}.
\end{cases}\]
The image of $\alpha_g$ is a certain finite index subgroup $G'_g \leq G_g$, which is an arithmetic subgroup associated to the algebraic group $\GG \in \{\Sp_{2g}, \OO_{g,g}\}$. This subgroup acts by outer automorphisms on $\mr{Tor}(W_g, D^{2n})$, and so $H^*(B\mr{Tor}(W_g, D^{2n});\bQ)$ has the structure both of a $\bQ$-algebra and of a $G'_g$-representation. Writing 
\[H^i(B\mr{Tor}(W_g, D^{2n});\bQ)^\mr{alg} \subseteq H^i(B\mr{Tor}(W_g, D^{2n});\bQ)\]
for the sum of all finite-dimensional $G'_g$-subrepresentations which extend to representations of $\GG$, the goal of this paper is to determine $H^*(B\mr{Tor}(W_g, D^{2n});\bQ)^\mr{alg}$ as a $\bQ$-algebra and a $G'_g$-representation in a range of degrees tending to infinity with $g$. 

\subsection{Some stable cohomology}\label{sec:IntroStabCoh}
Before describing $H^*(B\mr{Tor}(W_g, D^{2n});\bQ)^\mr{alg}$, let us recall the description of the stable cohomology of $BG'_g$ and $B\diff(W_g, D^{2n})$ for $2n \neq 4$. 

The rational cohomology of $G'_g$ has been determined by Borel \cite{borelstable} in a range of degrees tending to infinity with $g$. In this range it is given by
\[H^*(BG'_g;\bQ) = \begin{cases}
\bQ[\sigma_2, \sigma_6, \sigma_{10}, \ldots] & \text{ if $n$ is odd,}\\
\bQ[\sigma_4, \sigma_8, \sigma_{12}, \ldots] & \text{ if $n$ is even,}
\end{cases}\]
for certain classes $\sigma_{2i}$ of degree $|\sigma_{2i}| = 2i$. 

The rational cohomology of $B\diff(W_g, D^{2n})$ in a stable range has been determined by a combination of work by Harer and Madsen--Weiss \cite{Ha,MW} for $2n=2$ and by Galatius--Randal-Williams \cite{grwcob, grwstab1} for $2n \geq 6$. To give a uniform description, let us write $\mathcal{V}$ for the polynomial algebra in the Euler class $e$ of degree $2n$, and the Pontrjagin classes $p_i$ of degree $4i$, for $i= \lceil \frac{n+1}{4}\rceil, \ldots, n-2, n-1$, and $\cB$ for the set of monomials in these generators. If $c \in \cB$,
\[W_g \lra E \overset{\pi}\lra B\diff(W_g, D^{2n})\]
denotes the universal $W_g$-bundle over $B\diff(W_g, D^{2n})$, and $T_\pi E \to E$ denotes its vertical tangent bundle, then we define the \emph{Miller--Morita--Mumford class}
\[\kappa_c \coloneqq \int_\pi c(T_\pi E) \in H^{|c|-2n}(B\diff(W_g, D^{2n});\bQ).\]
Then as long as $2n \neq 4$ the natural map
\[\bQ[\kappa_c \, \mid \, c \in \cB_{> 2n}] \lra H^*(B\diff(W_g, D^{2n});\bQ)\]
is an isomorphism in a range of degrees tending to infinity with $g$.

The interaction between these two calculations is easy to describe. The Hirzebruch $L$-classes $\cL_i$ are certain polynomials in the Pontrjagin classes $p_i$, and we may write $\kappa_{\cL_i}$ for the associated linear combination of $\kappa_c$'s, which is a class of degree $4i-2n$. We choose the classes $\sigma_i$ in Borel's theorem to satisfy $\kappa_{\cL_i} = (\alpha_g)^*(\sigma_{2(2i-n)})$, which is possible by a theorem of Atiyah \cite{AtiyahFib}.

From this discussion we see that the Miller--Morita--Mumford classes $\kappa_{\cL_i}$ vanish in the rational cohomology of $B\mr{Tor}(W_g, D^{2n})$, so there is an induced map
\[\frac{\bQ[\kappa_c \, \mid \, c \in \cB_{> 2n}]}{(\kappa_{\cL_i} \, | \, 4i-2n > 0)} \lra H^*(B\mr{Tor}(W_g, D^{2n});\bQ).\]
This will give the $G'_g$-invariant part of the cohomology of $B\mr{Tor}(W_g, D^{2n})$ in a stable range---as was already shown in the pseudoisotopy stable range by Ebert--Randal-Williams \cite{oscarjohannestorelli}---but the full cohomology will be much larger.

\subsection{Twisted Miller--Morita--Mumford classes}\label{sec:intro:MMM}

Our description of (the algebraic part of) the cohomology of $B\mr{Tor}(W_g, D^{2n})$ will be in terms of certain variants of the Miller--Morita--Mumford classes. To describe them, now let
\[W_g \lra E \overset{\pi}\lra B\mr{Tor}(W_g, D^{2n})\]
denote the universal $W_g$-bundle over $B\mr{Tor}(W_g, D^{2n})$, and $s \colon B\mr{Tor}(W_g, D^{2n}) \to E$ denote the section determined by the centre of the disc $D^{2n} \subset W_g$. The Serre spectral sequence for $\pi$ degenerates at $E_2$, and the section $s$ determines a splitting of the short exact sequence
\[0 \lra H^n(B\mr{Tor}(W_g, D^{2n});\bQ) \overset{\pi^*}\lra H^n(E;\bQ) \lra H^n(W_g;\bQ) \lra 0,\]
and hence a map $\iota \colon H^n(W_g;\bQ) \to H^n(E;\bQ)$. Then for $v_1, \ldots, v_r \in H^n(W_g;\bQ)$ and $c \in \cB$ we define
\[\kappa_c(v_1 \otimes \cdots \otimes v_r) \coloneqq \int_\pi c(T_\pi E) \cdot \iota(v_1) \cdots \iota(v_r) \in H^{|c|+n(r-2)}(B\mr{Tor}(W_g, D^{2n});\bQ).\]
These classes generalise the Miller--Morita--Mumford classes, in the sense that $\kappa_c(1)= \kappa_c$ for $1 \in \bQ = H^n(W_g;\bQ)^{\otimes 0}$. Under the action of $G_g'$ on $H^{*}(B\mr{Tor}(W_g, D^{2n});\bQ)$ these classes transform via the action of $G'_g$ on the $v_i \in H^n(W_g;\bQ)$, which is identified with the dual $H_n(W_g;\bQ)^\vee$ of the standard representation of $G'_g$.

\subsection{The ring presentation} 

The easiest formulation of our results is as a presentation of the ring $H^{*}(B\mr{Tor}(W_g, D^{2n});\bQ)^\mr{alg}$ in a stable range of degrees, generated by the classes $\kappa_c(v_1 \otimes \cdots \otimes v_r)$ and subject to an explicit collection of relations. To formulate this theorem we write $a_1, a_2, \ldots, a_{2g}$ for a basis of $H^n(W_g;\bQ)$, and $a_1^\#, a_2^\#, \ldots, a_{2g}^\#$ for the Poincar{\'e} dual basis characterised by $\langle a_i^\# \cdot a_j, [W_g]\rangle=\delta_{ij}$.

\begin{atheorem}\label{MainThm:Ring}
If $2n \geq 6$ then in a range of degrees tending to infinity with $g$ the graded-commutative ring $H^{*}(B\mr{Tor}(W_g, D^{2n});\bQ)^\mr{alg}$ is generated by the classes
\[\kappa_c(v_1 \otimes \cdots \otimes v_r) \text{ with $r \geq 0$,  $c \in \cB$, and $|c|+ n(r-2)>0$}.\]
A complete set of relations in this range is given by
\begin{enumerate}[\indent (i)]
\item linearity in each $v_i$,
\item $\kappa_c(v_{\sigma(1)} \otimes \cdots \otimes v_{\sigma(k)}) = \mr{sign}(\sigma)^n \cdot \kappa_c(v_1 \otimes \cdots \otimes v_k)$,

\item $ \sum_i \kappa_x(v_1 \otimes \cdots \otimes v_j \otimes a_i) \cdot  \kappa_y(a_i^\# \otimes v_{j+1} \otimes \cdots \otimes v_r) = \kappa_{x \cdot y}(v_1 \otimes \cdots \otimes v_r)$,

\item $\sum_i \kappa_x(v_1 \otimes \cdots \otimes v_r \otimes a_i \otimes a_i^\#) = \kappa_{e \cdot x}(v_1 \otimes \cdots \otimes v_r)$,

\item $\kappa_{\cL_i}(1)=0$.
\end{enumerate}
In the case $2n=2$, if $H^*(B\mr{Tor}(W_g, D^{2});\bQ)$ is finite-dimensional for $* < N$ and $g \gg 0$, then this description is valid in degrees $* \leq N$ for $g \gg 0$.
\end{atheorem}

\begin{remark}\mbox{}
\begin{enumerate}[\indent (i)]
\item The presentation in Theorem \ref{MainThm:Ring} is not supposed to be efficient. In Theorem \ref{thm:SmallRingStr} we give a smaller but somewhat more complicated presentation, in which the generators are just the classes $\kappa_c(1)$, $\kappa_c(v_1)$, and $\kappa_1(v_1 \otimes v_2 \otimes v_3)$.

\item We describe explicit stability ranges for all the results of this paper in Section \ref{sec:ranges}.

\item No assumption about the finiteness of the cohomology of $B\mr{Tor}(W_g, D^{2n})$ is required in the case $2n \geq 6$ because it is indeed finite-dimensional in each degree: this has been recently proved by the first author \cite{kupersdisk}. 

\item In a companion paper \cite{KR-WAlg} we prove that for $2n \geq 6$ the $G'_g$-representations $H^i(B\mr{Tor}(W_g, D^{2n});\bQ)$ are in fact algebraic. Thus in this case Theorem \ref{MainThm:Ring} in fact computes the whole cohomology ring in a stable range.

\item In dimension $2n=2$ the homology of $B\mr{Tor}(W_g, D^{2})$ cannot be finite-dimensional in every degree \cite{AkitaInfinite}. However it is a folk conjecture (see e.g.~\cite[p.\ 71]{HainFin}) that the cohomology of $B\mr{Tor}(W_g, D^{2})$ is finite-dimensional in a range of degrees tending to infinity with $g$; assuming this conjecture, Theorem \ref{MainThm:Ring} gives a complete description of the algebraic subrepresentation of the cohomology of $B\mr{Tor}(W_g, D^{2})$ in a stable range. We explain further consequences for the case $2n=2$ in Section \ref{sec:Dim2}.
\end{enumerate}
\end{remark}

\subsection{The categorical description}

While Theorem \ref{MainThm:Ring} is the most easily formulated of our results, it is often difficult to answer questions about an object described by a presentation. Our main result is a different description of $H^{*}(B\mr{Tor}(W_g, D^{2n});\bQ)^\mr{alg}$ in the stable range, of a categorical nature, which we shall explain in this section. 

Theorem \ref{MainThm:Ring} will be deduced from this categorical description, but using this description it is also mechanical to calculate the character of each $G'_g$-representation $H^{i}(B\mr{Tor}(W_g, D^{2n});\bQ)^\mr{alg}$ in the stable range (whereas it is not clear how to extract this from Theorem \ref{MainThm:Ring}). We will explain how to calculate such characters in Section \ref{sec:Char}, and give several examples.

Our categorical description will be in terms of Brauer categories, a notion which we learnt from Sam--Snowden \cite{SS}. The description we will give depends of course on the value of $n$, but its \emph{form} also depends on the parity of $n$. In this introduction for simplicity we describe the case $n$ even; the case $n$ odd is similar in spirit but requires a substantial discussion of signs, which we defer to the body of the text.

\begin{definition}
An \emph{unordered matching} of a finite set is a decomposition of that set into disjoint pairs. The \emph{downward Brauer category} $\cat{dBr}$ has objects finite sets. A morphism in $\cat{dBr}$ from $S$ to $T$ is a pair $(f, m_S)$ of an injection $f \colon T \hookrightarrow S$ along with an unordered matching $m_S$ of $S \setminus f(T)$. The composition of such a morphism with a morphism $(g, m_T) \colon T \to U$ is given by the injection $f \circ g \colon U \hookrightarrow S$ along with the unordered matching $m_S \sqcup f(m_T)$ of $S \setminus (f \circ g)(U)$. Disjoint union endows $\cat{dBr}$ with a symmetric monoidal structure.
\end{definition}

\begin{figure}[h]
	\begin{tikzpicture}
	\draw (0,3) to[out=-90,in=-90,looseness=1.5] (2,3);
	\draw (-1,3) to[out=-90,in=90] (0,0);
	\draw [line width=2mm,white] (1,3) to[out=-90,in=90] (3,0);
	\draw (1,3) to[out=-90,in=90] (3,0);
	\draw [line width=2mm,white] (3,3) to[out=-90,in=90] (1,0);
	\draw (3,3) to[out=-90,in=90] (1,0);
	\draw [line width=2mm,white] (4,3) to[out=-90,in=90] (2,0);
	\draw (4,3) to[out=-90,in=90] (2,0);
	\foreach \x in {-1,...,4}
	\node at (\x,3) {$\bullet$};
	\foreach \x in {-0,...,3}
	\node at (\x,0) {$\bullet$};
	\node at (-2,3) {$S$};
	\node at (-1,0) {$T$};
	\end{tikzpicture}
	\caption{A graphical representation of a morphism $(f,m_S)$ in $\mathsf{dBr}(S,T)$ from a $6$-element set $S$ to a $4$-element set $T$. The order of crossings is irrelevant.}
	\label{fig:brauerdownward}
\end{figure}
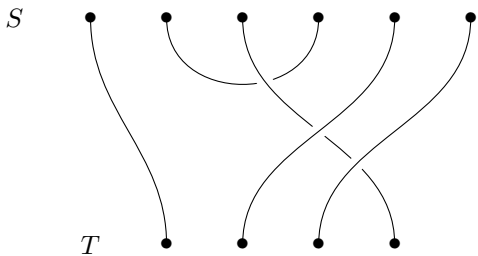

As we have supposed that $n$ is even for now, the fundamental representation $H(g)$ of $G_g'$ is equipped with a non-degenerate symmetric bilinear form $\lambda \colon H(g) \otimes H(g) \to \bQ$. Using it, we may define a functor
\[K \colon \cat{dBr} \lra \cat{Rep}(G_g')\]
to the category of $\bQ$-representations of $G'_g$, given on objects by $K(S) = H(g)^{\otimes S}$ and on a morphism $(f, m_S) \colon S \to T$ by
\[K(f, m_S) \colon H(g)^{\otimes S} \xrightarrow{m_S} H(g)^{\otimes f(T)} \xrightarrow{H(g)^{\otimes f^{-1}}} H(g)^{\otimes T},\]
where the first map applies the symmetric pairing $\lambda$ to the matched pairs of $S$. Taking $\bQ$-linear duals defines a functor $K^\vee \colon \cat{dBr}^\mr{op} \to \cat{Rep}(G_g')$.

Both $\cat{Rep}(G_g')$ and the category $\mathsf{Gr}(\bQ\text{-mod})$ of graded $\bQ$-modules may be considered as subcategories of the category $\mathsf{Gr}(\cat{Rep}(G_g'))$ of graded $\bQ$-representations of $G'_g$, as those graded representations which are concentrated in degree zero or are trivial respectively. We can thus use coends to define a functor
\[K^\vee \otimes^{\cat{dBr}} - \colon \mathsf{Gr}(\bQ\text{-mod})^{\cat{dBr}} \lra \mathsf{Gr}(\cat{Rep}(G_g')).\]
As $K$ is strong symmetric monoidal, $K^\vee \otimes^{\cat{dBr}} -$ is also strong symmetric monoidal when the functor category $\mathsf{Gr}(\bQ\text{-mod})^{\cat{dBr}}$ is equipped with the symmetric monoidal structure given by Day convolution. The categorical formulation of our main result for $n$ even is an identification of the commutative ring object
\[H^*(B\mr{Tor}(W_g, D^{2n});\bQ)^\mr{alg} \in \mathsf{Gr}(\cat{Rep}(G_g'))\]
with the value of the functor $K^\vee \otimes^{\cat{dBr}} -$ on a certain commutative ring object in $\mathsf{Gr}(\bQ\text{-mod})^{\cat{dBr}}$, which we now define. Recall that $\cB$ denotes the set of monomials in the Euler class $e$ and the Pontrjagin classes $p_i$ for $i = \lceil \frac{n+1}{4} \rceil,\cdots,n-2,n-1$, including the trivial monomial $1$.

\begin{definition}\label{def:PartAlg}
A \emph{partition} of a finite set $S$ is a finite collection of (possibly empty) subsets $\{S_\alpha\}_{\alpha \in I}$ of $S$ which are pairwise disjoint and whose union is $S$.

We write $\mathcal{P}(-; \mathcal{B})'_{\geq 0} \colon \cat{dBr} \to \mathsf{Gr}(\bQ\text{-mod})$ for the functor which assigns to a finite set $S$ the vector space with basis the set of partitions $\{S_\alpha\}_{\alpha \in I}$ of $S$ equipped with a labelling of each part $S_\alpha$ by an element $c_\alpha \in \cB$, such that
\begin{enumerate}[\indent (i)]
\item each part of size 0 has label of degree $>2n$,

\item each part of size 1 has label of degree $\geq n$,

\item each part of size 2 has label of degree $> 0$.
\end{enumerate}
We make this a graded vector space by declaring a part $S_\alpha$ labelled by $c_\alpha$ to have degree $|c_\alpha| + n(|S_\alpha|-2)$, and a labelled partition to have degree the sum of the degrees of its parts.

The linear map $\mathcal{P}(S; \mathcal{B})'_{\geq 0} \to \mathcal{P}(T; \mathcal{B})'_{\geq 0}$ induced by a bijection $(f, \varnothing) \colon S \to T$ in $\cat{dBr}$ is simply given by relabelling. The linear map induced by $(\mathrm{inc}, (x,y)) \colon  S \to S \setminus \{x,y\}$ sends a labelled partition $(\{S_\alpha\}, \{c_\alpha\})$ to the labelled partition given as follows:
\begin{enumerate}[\indent (i')]

\item if some $S_\alpha$ contains $\{x,y\}$ (and $|c_\alpha|>0$ if $S_\alpha=\{x,y\}$) then we change the part to $S_\alpha \setminus \{x,y\}$, and change the label to $e \cdot c_\alpha$, 

\item if $x$ and $y$ lie in different parts $S_\alpha$ and $S_\beta$, then we merge these into a new part $(S_\alpha \setminus \{x\}) \cup (S_\beta \setminus \{y\})$ labelled by $c_\alpha \cdot c_\beta$.
\end{enumerate}
On a more general morphism in $\cat{dBr}$ the effect of the functor $\mathcal{P}(-; \mathcal{B})'_{\geq 0}$ is determined by the above and functoriality.

The functor $\mathcal{P}(-; \mathcal{B})'_{\geq 0}$ has a lax symmetric monoidality given by disjoint union, making it into a commutative ring object in $\smash{\mathsf{Gr}(\bQ\text{-mod})^\cat{dBr}}$.
\end{definition}

When $n$ is odd we must instead consider a variant $\cat{dsBr}$, the downwards signed Brauer category, and the analogue of the functor of Definition \ref{def:PartAlg} must be twisted by a certain determinant line functor. Allowing for these differences, for all $n$ the categorical formulation of our result is as follows, where we identify an empty part labelled by $c \in \cB_{>2n}$ with the Miller--Morita--Mumford class $\kappa_c$.

\begin{atheorem}\label{thm:Main}
There is a morphism
\[\frac{K^\vee \otimes^{\cat{d(s)Br}} \left( \mathcal{P}(-; \mathcal{B})'_{\geq 0} \otimes {\det}^{\otimes n} \right)}{(\kappa_{\cL_i} \, | \, 4i-2n >0)} \lra H^*(B\mr{Tor}(W_g, D^{2n});\bQ)^\mr{alg}\]
of commutative ring objects in $\cat{Gr}(\cat{Rep}(G_g'))$, which if $2n \geq 6$ is an isomorphism in a range of degrees tending to infinity with $g$.

If $2n=2$ and $H^*(B\mr{Tor}(W_{g}, D^{2});\bQ)$ is finite dimensional for $* < N$ and $g \gg 0$, then this map is an isomorphism in degrees $* \leq N$, and is a monomorphism in degree $N+1$, for $g \gg 0$.
\end{atheorem}

\begin{remark}\mbox{}
\begin{enumerate}[\indent (i)]

\item Many of the remarks after the statement of Theorem \ref{MainThm:Ring} apply here too.

\item Irreducible representations of the symmetric groups and of the algebraic groups $\{\Sp_{2g}, \OO_{g,g}\}$ are both indexed by partitions. In the stable range we will show that the multiplicity in $H^*(B\mr{Tor}(W_g, D^{2n});\bQ)$ of the irreducible algebraic $G'_g$-representation corresponding to a partition $\lambda \vdash q$  is the same as the multiplicity in \[\bQ \otimes_{\bQ[\kappa_{\cL_i} \, | \, 4i-2n >0]}\mathcal{P}(\{1,2,\ldots,q\}; \mathcal{B})'_{\geq 0} \otimes {\det(\bQ^q)}^{\otimes n}\]
of the irreducible $\Sigma_q$-representation corresponding to the partition $\lambda$. We explain how to calculate these multiplicities in Section \ref{sec:Char}.

\item Letting $\cH(g)$ denote the local coefficient system on $B\mr{Diff}(W_g, D^{2n})$ given by the action of diffeomorphisms on $H_n(W_g;\bQ)$, a key step in the proof of this theorem is to completely describe the bigraded ring $H^*(B\mr{Diff}(W_g, D^{2n}); \cH(g)^{\otimes \bullet})$ in a stable range, together with its behaviour in the variable $\bullet$ as a functor on the (signed) Brauer category. We do this in Section \ref{sec:isomorphism-theorem}. This description is valid in all dimensions $2n \neq 4$.

\end{enumerate}
\end{remark}

\subsection*{Acknowledgements}
The authors would like to thank M.\ Krannich and D.\ Petersen for their comments on earlier versions of this paper. AK was supported by the Danish National Research Foundation through the Centre for Symmetry and Deformation (DNRF92) and by the European Research Council (ERC) under the European Union's Horizon 2020 research and innovation programme (grant agreement No.\ 682922). AK is supported by NSF grant DMS-1803766. ORW was partially supported by EPSRC grant EP/M027783/1, and partially supported by the ERC under the European Union's Horizon 2020 research and innovation programme (grant agreement No.\ 756444), and by a Philip Leverhulme Prize from the Leverhulme Trust.

\section{Some background on representation theory}\label{sec:rep-theory}

\subsection{Arithmetic groups and their representations} Let $\epsilon \in \{-1,1\}$ and let $H(g)$ be a $2g$-dimensional rational vector space equipped with a nonsingular $\epsilon$-symmetric pairing $\lambda \colon H(g) \otimes H(g) \to \bQ$, of signature 0 if $\epsilon=1$. We denote the group of automorphisms of $H(g)$ which preserve this pairing $\mr{O}_{\epsilon}(H(g))$; this is usually denoted by $\mr{O}_{g,g}(\bQ)$ if $\epsilon=1$, and by $\mr{Sp}_{2g}(\bQ)$ if $\epsilon=-1$. These are the $\bQ$-points of algebraic groups $\OO_{g,g}$ and $\Sp_{2g}$ respectively. As $\OO_{g,g}$ is not Zariski connected we shall have to occasionally work with its index two connected subgroup $\gS\OO_{g,g} \leq \OO_{g,g}$, and in this case we will write $\gS\GG$ for $\Sp_{2g}$ or $\gS\OO_{g,g}$.

We shall need to consider \emph{arithmetic subgroups} $G$ of the algebraic groups $\GG \in \{\Sp_{2g}, \OO_{g,g}\}$ defined over $\bQ$, which we shall take to mean: a subgroup $G \leq \GG(\bQ)$ which is commensurable to $\GG(\bZ)$ and which, in the case $\GG = \OO_{g,g}$, is not entirely contained in $\mr{SO}_{g,g}(\bQ)$. The latter condition is non-standard, but holds for us and ensures that $G$ is Zariski dense in $\GG(\bQ)$, as we now explain.

\subsubsection{Zariski density}
Given an arithmetic subgroup $G$ of $\GG$ as above, write
\[SG \coloneqq \begin{cases}
G & \text{ if } \GG= \Sp_{2g},\\
G \cap \mr{SO}_{g,g}(\bQ) & \text{ if } \GG= \OO_{g,g}.
\end{cases}\]
As $\Sp_{2g}$ and $\gS\OO_{g,g}$ are connected semisimple algebraic groups defined over $\bQ$, it follows from a theorem of Borel--Harish-Chandra \cite[Theorem 7.8]{BHC} that $SG$ is a lattice in $\gS\GG(\bR)$, and hence by the Borel Density Theorem \cite{BorelDensity} that $SG$ is Zariski dense in $\gS\GG(\bR)$, so also in $\gS\GG(\bQ)$. As we have assumed in the case $\GG = \OO_{g,g}$ that $G$ does not lie entirely inside $\mr{SO}_{g,g}(\bQ)$, it follows that $G$ is Zariski dense in $\GG(\bQ)$.

\subsubsection{Algebraic and almost algebraic representations}

We consider an arithmetic group $G$ associated to $\GG \in \{\Sp_{2g}, \OO_{g,g}\}$ as defined above.

\begin{definition}
A representation $\phi \colon G \to GL(V)$ on an $n$-dimensional $\bQ$-vector space $V$ is \emph{algebraic} if it is the restriction of a finite-dimensional representation of the algebraic group $\GG$, i.e.\ there is a morphism of algebraic groups $\varphi \colon \GG \to \mathbf{GL}(V)$ which on taking $\bQ$-points and restricting to $G$ yields $\phi$.

More generally the representation $(\phi, V)$ is \emph{almost algebraic} if there is a finite index subgroup $G' \leq G$ such that the restriction of $\phi$ to $G'$ is algebraic.
\end{definition}

We usually denote a representation $(\phi, V)$ by $V$, leaving the action of $G$ on $V$ implicit.

If $V$ is an algebraic representation of $G$ and $W \leq V$ is a $G$-subrepresentation, then, as $G$ is Zariski dense in $\GG(\bQ)$, the subspace $W$ is also $\GG(\bQ)$-invariant so $W$ is again an algebraic representation. Similarly, $V/W$ is again algebraic. If $V$ is a (not necessarily finite-dimensional) $G$-representation, we let $V^\mr{alg} \leq V$ be the union of its algebraic subrepresentations; this need not be itself algebraic, but it is if it is finite-dimensional: in any case we call it the \emph{maximal algebraic subrepresentation} of $V$.

The following appears in page 109 of \cite{serrearithmetic} and is a consequence of a theorem of Margulis \cite[Theorem (2)]{margulis}; see Raghunathan \cite{RaghunathanCoh} for a special case.

\begin{theorem}\label{thm.margulis}
If $\GG$ is a simple algebraic group of $\bQ$-rank $\geq 2$ defined over $\bQ$, $G$ is an arithmetic subgroup of $\GG$ and $V$ is a finite-dimensional representation of $G$, then $V$ is almost algebraic.
\end{theorem}

This for example applies to $\GG = \Sp_{2g}$ or $\gS\OO_{g,g}$ for $g \geq 2$, but the conclusion then easily follows for $\GG=\OO_{g,g}$ too, as this contains $\gS\OO_{g,g}$ with finite index.

For the algebraic groups under consideration Borel \cite{borelstable, borelstable2} proved a cohomological vanishing result, the following strong version of which we shall use:

\begin{theorem}\label{thm.borelvanishingweak}
Let $G$ be an arithmetic subgroup of $\GG \in \{ \Sp_{2g},\OO_{g,g}\}$, and set $e=0$ if $\GG = \Sp_{2g}$ and $e=1$ if $\GG = \OO_{g,g}$. Then for $g \geq 3+e$ and $V$ an almost algebraic representation of $G$, the natural maps
\[H^*(\GG_\infty;\bQ) \otimes V^G \lra H^*(G;\bQ) \otimes V^{G} \xrightarrow{-\cdot-} H^*(G ; V)\]
are both isomorphisms for $* < g-e$, where
\[H^*(\GG_\infty;\bQ) \coloneqq \begin{cases} \bQ[\sigma_2,\sigma_6,\ldots] & \text{if $\GG = \Sp_{2g}$,} \\
\bQ[\sigma_4,\sigma_8,\ldots] & \text{if $\GG = \OO_{g,g}$}.\end{cases}\]
\end{theorem}

Here $H^*(\GG_\infty;\bQ)$ is simply notation for the graded ring indicated in the statement, and the classes $\sigma_i \in H^i(G;\bQ)$ are to be interpreted as described in Section \ref{sec:IntroStabCoh}.

\begin{proof}
The groups $\Sp_{2g}$ and $\gS\OO_{g,g}$ are connected and simple, so the claim for arithmetic subgroups of these groups and algebraic $V$ follows in some range of degrees by combining \cite[Theorem 4.4 (i)]{borelstable2} and the main result of \cite{borelstable}, with $H^*((\gS\OO_{g,g})_\infty;\bQ) = \bQ[\sigma_4,\sigma_8,\ldots]$. The ranges we have stated are improvements of those given by Borel, and were stated in \cite{HainTorelli} without proofs, and proven in \cite{tshishikuBorel}, Theorem 17 for $\gS\OO_{g,g}$ and Theorem 29 for $\Sp_{2g}$.

To deal with the case that $V$ is almost algebraic, suppose that $G' \lhd G$ is a finite index normal subgroup such that the restriction of $V$ to $G'$ is algebraic. Then there is a commutative diagram
\[\begin{tikzcd} H^*(G;\bQ) \otimes V^G \rar \dar{\cong} & H^*(G;V) \dar{\cong} \\
(H^*(G';\bQ) \otimes V^{G'})^{G/G'} \rar{\cong} & H^*(G';V)^{G/G'},\end{tikzcd}\]
with bottom map an isomorphism by the previous case, and the vertical maps isomorphisms by a transfer argument. 

To deduce the result for $\OO_{g,g}$ from that for $\gS \OO_{g,g}$, we observe that if $G$ is an arithmetic subgroup of $\mr{O}_{g,g}(\bQ)$ then by our slightly non-standard definition $SG \coloneqq G \cap \mr{SO}_{g,g}(\bQ)$ is a proper subgroup and there is an extension
\begin{equation*}
1 \lra SG \lra G \lra C_2 \lra 1.
\end{equation*}
The spectral sequence for this extension collapses to $H^*(SG ; V)^{C_2} \cong H^*(G;V)$. Using the result for $SG$, we find that the maps
\[H^*((\gS\OO_{g,g})_\infty;\bQ) \otimes V^{SG} \lra H^*(SG;\bQ) \otimes V^{SG} \xrightarrow{-\cdot-} H^*(SG ; V)\]
are isomorphisms in the given range. But $C_2$ acts trivially on $H^*((\gS\OO_{g,g})_\infty;\bQ) = \bQ[\sigma_4,\sigma_8,\ldots]$, by considering Borel's proof of this identity, so taking $C_2$-invariants therefore gives the required conclusion.
\end{proof}

A consequence of this theorem is that as long as $g \geq 3+e$ taking $G$-invariants is exact on the category of almost algebraic representations of $G$. However, by \cite{raghunathan} this is in fact true for $g \geq 2$ already (see also \cite[Theorem (3)]{margulis}). More generally, if $V$ and $W$ are almost algebraic representations then so is $W^\vee \otimes V$, so
\[\mr{Ext}_G^1(W, V) \cong H^1(G ; W^\vee \otimes V)=0\]
for $g \geq 2$, and hence every short exact sequence of almost algebraic representations splits.

\subsubsection{Orthogonal and symplectic representation theory} \label{sec:ortho-sympl-rep-theory} 

The non-singular $\epsilon$-symmetric pairing $\lambda$ is dual to an $\epsilon$-symmetric form $\omega \colon \bQ \to H(g) \otimes H(g)$, which is characterised by $(\lambda \otimes \mr{id})(- \otimes \omega) = \mr{id}(-)$. If $\{a_i\}$ is a basis of $H(g)$ and $\{a_i^\#\}$ is the dual basis determined by $\lambda(a^\#_i \otimes a_j) = \delta_{ij}$, then
\[\omega = \sum_i a_i \otimes a_i^\#.\]
For each $i$ and $j$ in $\{1,2,\ldots, q\}$ there is a map
\[\lambda_{i,j} \colon H(g)^{\otimes q} \lra H(g)^{\otimes q-2}\]
given by applying the pairing to the $i$th and $j$th factors, and dually a map
\[\omega_{i,j} \colon H(g)^{\otimes q-2} \lra H(g)^{\otimes q}\]
which inserts the form $\omega$ at the $i$th and $j$th factors. 

Weyl constructed irreducible representations of $\mr{O}_\epsilon(H(g))$ as follows. Let us write
\begin{align*}
H(g)^{[q]} &\coloneqq \mr{Ker}\left(H(g)^{\otimes q} \overset{\lambda_{i,j}}\lra \bigoplus_{i, j} H(g)^{\otimes q-2} \right),\\
H(g)_{[q]} &\coloneqq \mr{Cok}\left(\bigoplus_{i, j} H(g)^{\otimes q-2} \overset{\omega_{i,j}}\to H(g)^{\otimes q}\right).
\end{align*}
These have an action of the symmetric group $\Sigma_q$ by permuting factors, and the composition $H(g)^{[q]} \to H(g)^{\otimes q} \to H(g)_{[q]}$ is an isomorphism. Furthermore, the self-duality $x \mapsto \lambda (x,-) \colon H(g) \overset{\sim}\to H(g)^\vee$ induces an isomorphism $(H(g)_{[q]})^\vee \cong H(g)^{[q]}$. 

The irreducible $\bQ$-representations of the symmetric group $\Sigma_q$ are in bijection with partitions of $\lambda$ of the number $q$; the construction sends each partition $\lambda$ to an irreducible module $S^\lambda$ given by the image of the Young symmetriser acting on $\bQ[\Sigma_q]$, see Section 9.2.4 of \cite{Procesi}. For each partition $\lambda$ of $q$ we then define a $\mr{O}_\epsilon(H(g))$-representation
\[V_\lambda(H(g)) \coloneqq [S^\lambda \otimes H(g)^{[q]}]^{\Sigma_q},\]
which we shall usually shorten to $V_\lambda$. In particular, we have a decomposition
\begin{equation}\label{eq:SchurWeyl}
H(g)^{[q]} \cong \bigoplus_{\lambda \vdash q} S^\lambda \otimes V_\lambda(H(g))
\end{equation}
as a $\Sigma_q \times \mr{O}_\epsilon(H(g))$-representation, cf.~\cite[Section 9.9.2]{Procesi}.

The following theorems are consequences of the representation theory of the Lie groups $\mr{Sp}_{2g}(\bC)$ and $\mr{O}_{g,g}(\bC)$ (note that $\mr{O}_{g,g}(\bC) \cong \mr{O}_{2g}(\bC)$), which may be extracted from Section 11.6.4 and 11.6.5 of \cite{Procesi}, and of the Zariski density of $\mr{Sp}_{2g}(\bQ)$ and $\mr{O}_{g,g}(\bQ)$ inside these groups.

\begin{theorem}\label{thm:RepsOfSpAndO}
The representation $V_\lambda(H(g))$ of $\mr{O}_\epsilon(H(g))$ is zero or irreducible. If $2|\lambda| \leq \dim(H(g)) = 2g$ then it is irreducible, and such irreducibles are all distinct. 
\end{theorem}

The $V_\lambda(H(g))$ are representations of the algebraic groups $\OO_{g,g}$ or $\Sp_{2g}$, so their restrictions to an arithmetic subgroup $G$ of $\mr{O}_{g,g}(\bQ)$ or $\mr{Sp}_{2g}(\bQ)$ are by definition algebraic representations. 

\begin{theorem}\label{thm:RepsOfSpAndO2}
Every algebraic representation of an arithmetic subgroup $G$ of $\mr{O}_{g,g}(\bQ)$ or $\mr{Sp}_{2g}(\bQ)$ is a sum of $V_\lambda(H(g))$'s.
\end{theorem}

\subsubsection{Invariant theory}

The map $\omega \colon \bQ \to H(g) \otimes H(g)$ gives an invariant $\omega \in (H(g) \otimes H(g))^{\mr{O}_\epsilon(H(g))}$, which is sent to $\epsilon \cdot \omega$ under swapping the two factors. More generally, to each perfect ordered matching $m=((a_1, b_1), \ldots, (a_p, b_p))$ of a set $S = \{a_1, b_1, a_2, \ldots, a_p, b_p\}$ there is an associated invariant
\[\omega_m \coloneqq \bigotimes_{i=1}^p \omega_{a_i, b_i} \in (H(g)^{\otimes S})^{\mr{O}_\epsilon(H(g))}\]
and if $m'$ differs from $m$ by changing the order of $k$ pairs, then $\omega_{m'} = \epsilon^k \cdot \omega_m$. This observation provides a linear map
\begin{equation}\label{eq:invariant}
\frac{\bQ\{\text{perfect ordered matchings on $S$}\}}{\langle m'- \epsilon^k \cdot m \rangle} \lra  (H(g)^{\otimes S})^{\mr{O}_\epsilon(H(g))}
\end{equation}
We may summarise the first and second fundamental theorems of invariant theory for $\mr{O}_\epsilon(H(g))$ as follows.

\begin{theorem}\label{thm:FTInvariantTheory}
The map \eqref{eq:invariant} is surjective, and is injective as long as $2g \geq \vert S \vert$.
\end{theorem}

For a proof see Section 11.6.3 of \cite{Procesi}, apply $- \otimes_\bQ \bC$, use Zariski density and again that $\mr{O}_{g,g}(\bC) \cong \mr{O}_{2g}(\bC)$. The range for injectivity we have given is coarser than what is known to hold, see Section \ref{sec:PropTransfOfTransf} for a discussion.

\subsection{Representations of categories}\label{sec:RepOfCat}

Our strategy for approaching the cohomology of Torelli groups as $G'_g$-representations will be via symplectic or orthogonal Schur--Weyl duality. However as we wish to recover the ring structure too it is not enough to simply obtain the characters of these representations, or what is the same, their isomorphism class: one must work in a more categorified way. In this section we describe the required background on categorical representation theory. We were influenced, as is this exposition, by the treatment of Sam--Snowden \cite{SS}, which we shall attempt to follow closely, adapting slightly to fit our needs. 

We shall often work in the category $\mathsf{Gr}(\bQ\text{-mod})$ of non-negatively graded $\bQ$-vector spaces, equipped with the monoidal structure given by graded tensor product, and with symmetry given by the Koszul sign rule.

We let $\cA$ be a $\bQ$-linear abelian symmetric monoidal category (in our applications it will usually be the category of finite dimensional representations of a fixed arithmetic group $G$). We shall assume $\cA$ has all finite enriched colimits. We often impose one of the following two finiteness conditions on objects of $\cA$:

\begin{definition}An object $X$ of abelian category has \emph{finite length} if it admits a finite filtration with simple filtration quotients, i.e.\ there exists a finite sequence of monomorphisms $0 \hookrightarrow X_1 \hookrightarrow X_2 \hookrightarrow \cdots \hookrightarrow X$ such that each cokernel $X_{i+1}/X_i$ only has $0$ and itself as quotients. We let $(-)^f \subset (-)$ denote the full subcategory of finite length objects.\end{definition}

\begin{definition}An object $X$ of a symmetric monoidal category is a \emph{dualisable object} if there exists an object $X^\vee$ with a map $\eta \colon 1 \to X \otimes X^\vee$ called coevaluation and a map $\epsilon \colon X^\vee \otimes X \to 1$ called evaluation, satisfying the triangle identities. If it exists, the \emph{dual} $X^\vee$ is unique up to isomorphism. We let $(-)^d \subset (-)$ denote the full subcategory of dualisable objects.\end{definition}

The category $\cA$ is tensored over $(\bQ\text{-mod})^f$, the category of finite-dimensional vector spaces: for $V \in (\bQ\text{-mod})^f$ and $A \in \cA$ there is an object $A \odot V \in \cA$ characterised by a natural isomorphism $\Hom_\cA(A \odot V, -) \cong \Hom_\bQ(V, \Hom_\cA(A, -))$. 
In particular we have a functor $V \mapsto 1_\cA \odot V \colon (\bQ\text{-mod})^f \to \cA^f$, which has a right adjoint $A \mapsto \Hom_\cA(1_\cA, A)\colon \cA^f \to (\bQ\text{-mod})^f$.

\begin{definition}\label{def:2point9}
Let $\Lambda$ denote a $\bQ$-linear category such that all vector spaces of morphisms are finite-dimensional, such that the relation $[x] \leq [y] \Leftrightarrow \mr{Hom}_\Lambda(x,y) \neq 0$ on the set of isomorphism classes of objects of $\Lambda$ is a well-defined partial order, and and for which each object only admits nonzero morphisms to finitely many other objects up to isomorphism.
\end{definition}

We shall consider the categories $\cA^\Lambda$ and $(\bQ\text{-mod})^\Lambda$ of $\bQ$-linear functors. Objectwise tensor product gives a pairing $- \otimes - \colon \cA^\Lambda \times \cA \to \cA^\Lambda$. In particular, we may fix a $K \in \cA^\Lambda$ to get a functor $K \otimes - \colon \cA \to \cA^\Lambda$. When $K$ has dualisable values, this has an enriched right adjoint. We call such an object $K \in (\cA^d)^\Lambda$ with dualisable values a \emph{kernel}; taking the objectwise duals defines a functor $K^\vee \colon \Lambda^\mr{op} \to \cA^d$, which we may also consider as a functor to $\cA$. For any $M \in (\cA^\Lambda)^f$ we may therefore form the coend
\[K^\vee \otimes^\Lambda M \coloneqq \int^{x \in \Lambda} K^\vee(x) \otimes M(x) \in \cA.\]
This coend is formed in the enriched sense, and exists because it may be expressed as the coequaliser of
\[
\begin{tikzcd} 
\bigoplus_{x , y \in \Lambda} (K(y)^\vee \otimes M(x)) \odot {\Hom_\Lambda(x,y)} \arrow[r, shift left]
\arrow[r, shift right]& \bigoplus_{x \in \Lambda} K(x)^\vee \otimes M(x),
\end{tikzcd}
\]
which is equivalent to a finite colimit by the assumption that $M$ has finite length (so in particular $M(x)=0$ for all but finitely many isomorphism classes of $x \in \Lambda$; this is a simple consequence of the second assumption of Definition \ref{def:2point9}) and that objects $x$ admit morphisms only to finitely-many isomorphism classes of objects.

\begin{proposition}
The functors $K^\vee \otimes^\Lambda - \colon (\cA^\Lambda)^f \to \cA$ and $K \otimes - \colon \cA \to \cA^\Lambda$ participate in a natural isomorphism
\[\Hom_{\cA^\Lambda}(K \otimes -, -) \cong \Hom_\cA(-, K^\vee \otimes^\Lambda -) \colon \cA \times (\cA^\Lambda)^f \lra \bQ\text{-mod}.\]
\end{proposition}

\begin{proof}
The collection of evaluation maps $K(x)^\vee \otimes K(x) \to 1_\cA$ coequalises the two maps expressing $K^\vee \otimes^\Lambda K$ as a coequaliser, so determine a map $K^\vee \otimes^\Lambda K \to 1_\cA$. For $V \in \cA$ we have $K^\vee \otimes^\Lambda (K \otimes V) \cong (K^\vee \otimes^\Lambda K) \otimes V$, and using the morphism $K^\vee \otimes^\Lambda K \to 1_\cA$ constructed above gives a morphism $\epsilon_V \colon (K^\vee \otimes^\Lambda K) \otimes V \to V$ natural in $V$. 

As each $K(x)$ is dualisable, there are coevaluation maps $\eta_x \colon 1_\cA \to K(x) \otimes K(x)^\vee$ expressing this duality. This gives morphisms
\[M(x) \xrightarrow{\eta_x \otimes M(x)} K(x) \otimes K(x)^\vee \otimes M(x) \lra K(x) \otimes (K^\vee \otimes^\Lambda M)\]
natural in $x$, and hence a natural transformation $\eta_M \colon M \to K \otimes (K^\vee \otimes^\Lambda M)$.

One can verify that the compositions
\[K^\vee \otimes^\Lambda M \xrightarrow{K^\vee \otimes^\Lambda \eta_M} K^\vee \otimes^\Lambda K \otimes (K^\vee \otimes^\Lambda M) \xrightarrow{\epsilon_{K^\vee \otimes^\Lambda M}} K^\vee \otimes^\Lambda M\]
and
\[K \otimes V \xrightarrow{\eta_{K \otimes V}} K \otimes (K^\vee \otimes^\Lambda (K \otimes V)) =K \otimes (K^\vee \otimes^\Lambda K) \otimes V \xrightarrow{K \otimes \epsilon_V} K \otimes V\]
are the identity, which gives the required natural isomorphism.
\end{proof}

\subsubsection{Multiplicativity}

We shall now suppose that $\Lambda$ is equipped with a symmetric monoidal structure $\oplus$, in which case $(\cA^\Lambda)^f$ and $((\bQ\text{-mod})^\Lambda)^f$ have symmetric monoidal structures $- \otimes_\Lambda -$ given by Day convolution. That is, we first form the external product $M \boxtimes N \colon \Lambda \times \Lambda \to \cA$, and then take its left Kan extension $M \otimes_\Lambda N = \oplus_*(M \boxtimes N)$ along $\oplus \colon \Lambda \times \Lambda \to \Lambda$. Concretely, we have
\[(M \otimes_\Lambda N)(x) = \colim_{f \colon a \oplus b \to x} M(a) \otimes N(b),\]
which again exists because it is equivalent to a finite colimit.

There are several equivalent conditions we can impose on a $K \in (\cA^d)^\Lambda$ so that the above defined transformations $\epsilon$ and $\eta$ have good multiplicativity properties. The condition which is simplest to state and which we shall usually verify, is that $K \colon \Lambda \to \cA$ is a strong symmetric monoidal functor. This is equivalent to asking for a natural isomorphism $\oplus^* K \overset{\sim}\to K \boxtimes K \colon \Lambda \times \Lambda \to \cA$ which is associative and commutative in the evident sense. We call a $K$ satisfying any of these equivalent conditions a \emph{tensor kernel}.

\begin{proposition}
If $K \in (\cA^d)^\Lambda$ has the structure of a tensor kernel, then the functor $K^\vee \otimes^\Lambda - \colon (\cA^\Lambda)^f \to \cA$ has a strong symmetric monoidality.
\end{proposition}
\begin{proof}
Note that 
\begin{align*}(K^\vee \otimes^\Lambda A) \otimes (K^\vee \otimes^\Lambda B) &\cong \int^{x \in \Lambda, y \in \Lambda} \!\!\!\!\!\!\!\! K(x)^\vee \otimes A(x) \otimes K(y)^\vee \otimes B(y) \\
&\cong (K^\vee \boxtimes K^\vee) \otimes^{\Lambda \times \Lambda} (A \boxtimes B).\end{align*}
By dualising we obtain an isomorphism $K^\vee \boxtimes K^\vee \overset{\sim}\to \oplus^*K^\vee$, so write the above as
\[(\oplus^*K^\vee) \otimes^{\Lambda \times \Lambda} (A \boxtimes B) \cong K^\vee \otimes^\Lambda \oplus_*(A \boxtimes B) = K^\vee \otimes^\Lambda (A \otimes_\Lambda B).\]
This gives a strong monoidality, and it is routine to check that it is symmetric.
\end{proof}

\subsubsection{Detecting isomorphisms}

For a kernel $K$ we shall be interested in using the composition
\[\Delta \colon \cA \xrightarrow{K \otimes -} \cA^\Lambda \xrightarrow{\Hom_\cA(1_\cA, -)} (\bQ\text{-mod})^\Lambda\]
to test whether morphisms in $\cA$ are isomorphisms. As each $K(x)$ is a dualisable object, the functor $K(x) \otimes - \colon \cA \to \cA$ has $K(x)^\vee \otimes -$ as both a left and a right adjoint, and so is exact; thus $K \otimes - \colon \cA \to \cA^\Lambda$ is an exact functor. The functor $\Hom_\cA(1_\cA, -)$ is left exact, but will not typically be right exact.

Let $\cA_K \subset \cA$ be the subcategory of those objects which occur as sums of summands of $K(x)^\vee$'s. Let $\cA_K^\circ \subset \cA$ be the subcategory of those objects $X$ such that $\Hom_\cA(1_\cA, X \otimes K(x))=0$ for all $x \in \Lambda$. Note that if $Y$ is a summand of some $K(x)^\vee$ then it is also dualisable, and its dual $Y^\vee$ is a summand of $K(x)$: then $\Hom_\cA(1_\cA, Y \otimes K(x))\neq 0$, as it contains the non-zero morphism $1_\cA \overset{\eta_Y}\to Y \otimes Y^\vee \to Y \otimes K(x)$, so $\cA_K \cap \cA_K^\circ=\{0\}$.

\begin{lemma}\label{lem:Detection}
Let $f \colon A \to B$ be a morphism in $\cA$.
\begin{enumerate}[(i)]
\item If $\Delta(f)$ is injective, then $\mr{Ker}(f) \in \cA_K^\circ$.

\item If $A \in \cA_K$, $\mr{Ext}^1_\cA(K(x)^\vee, K(y)^\vee)=0$ for all $x, y \in \Lambda$, and $\Delta(f)$ is bijective, then $\mr{Cok}(f) \in \cA_K^\circ$.
\end{enumerate}
\end{lemma}
\begin{proof}
Consider the left exact sequence $0 \to \mr{Ker}(f) \to A \overset{f}\to B$, which remains left exact after applying $\Delta$. As $\Delta(f)$ is injective it follows that $\Delta(\mr{Ker}(f))=0$, i.e.\ that $\Hom_\cA(1_\cA, K(x) \otimes \mr{Ker}(f))=0$ for all $x \in \Lambda$. This is the definition of being in $\cA_K^\circ$.

Consider the exact sequence $0 \to A \overset{f}\to B \to \mr{Cok}(f) \to 0$. This remains exact after applying $K \otimes -$, so gives a long exact sequence
\[0 \to \Delta(A)(x) \xrightarrow{\Delta(f)(x)} \Delta(B)(x) \to \Delta(\mr{Cok}(f))(x) \overset{\partial}\to \mr{Ext}^1_\cA(1_\cA, K(x) \otimes A) \to \cdots\]
where the morphism $\Delta(f)(x)$ is surjective, so the connecting map is injective. But 
\[\mr{Ext}^1_\cA(1_\cA, K(x) \otimes A) \cong \mr{Ext}^1_\cA(K(x)^\vee, A)\]
and as $A$ is a sum of summands of $K(y)^\vee$'s this group vanishes by assumption.
\end{proof}

\subsection{The representation theory of Brauer categories}\label{sec:RepOfBrauer}

\subsubsection{The orthogonal group}

Let $G \subset \OO_{g,g}(\bQ)$ be an arithmetic subgroup (and recall that we write $SG = G \cap \mr{SO}_{g,g}(\bQ)$, which by our definition of arithmetic group is an index 2 subgroup of $G$). Let $\cA = \cat{Rep}(G)$ denote the category of finite dimensional representations of $G$, which is easily seen to have all finite $\bQ$-enriched colimits. We shall assume that $g \geq 2$ so that the functor $[-]^G$ is exact on this category and all extensions split, as discussed after Theorem \ref{thm.borelvanishingweak}.

Let us write $H(g) \in \cat{Rep}(G)$ for the standard $2g$-dimensional representation, which is isomorphic to $V_1$ as defined in Section \ref{sec:ortho-sympl-rep-theory}. It is equipped with a symmetric pairing $\lambda \colon H(g) \otimes H(g) \to \bQ$ and, dual to this, a symmetric form $\omega \colon \bQ \to H(g) \otimes H(g)$.

\begin{definition}
A \emph{matching} of a finite set $S$ is a partition of $S$ into disjoint ordered pairs. If $(a,b)$ is such a pair, its \emph{reverse} is the pair $(b,a)$.
\end{definition}

\begin{definition}The \emph{Brauer category} $\mathsf{Br}_d$ of charge $d \in \bQ$ is the following $\bQ$-linear category:
	\begin{itemize}
		\item The objects of $\cat{Br}_d$ are the finite sets.
		\item The morphisms $\mathsf{Br}_d(S,T)$ are given by the following $\bQ$-vector space. First, let $\mathsf{Br}_d(S,T)'$ be the vector space with basis given by triples $(f, m_S, m_T)$ of a bijection $f$ from a subset of $S^\circ \subset S$ to a subset of $T^\circ \subset T$, a matching $m_S$ of $S \setminus S^\circ$, and a matching $m_T$ of $T \setminus T^\circ$. Let $\mathsf{Br}_d(S,T)$ be the quotient vector space by the subspace generated by $(f, m_S, m_T) - (f, m'_S, m'_T)$ where $m'_S$ and $m'_T$ differ from  $m_S$ and $m_T$ by reversing some pairs. We consider it as being spanned by pictures as in Figure \ref{fig:brauerorthogonal}.

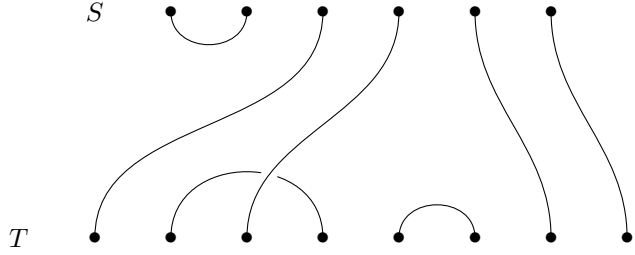
\begin{figure}[h]
	\begin{tikzpicture}
		\draw (-1,0) to[out=90,in=90,looseness=1.5] (1,0);
		\draw (3,0) to[out=90,in=90,looseness=1.5] (2,0);
		\draw (-1,3) to[out=-90,in=-90,looseness=1.5] (0,3);
		\draw (1,3) to[out=-90,in=90] (-2,0);
		\draw [line width=2mm,white] (2,3) to[out=-90,in=90] (0,0);
		\draw (2,3) to[out=-90,in=90] (0,0);
		\draw (3,3) to[out=-90,in=90] (4,0);
		\draw (4,3) to[out=-90,in=90] (5,0);
		\foreach \x in {-1,...,4}
		\node at (\x,3) {$\bullet$};
		\foreach \x in {-2,...,5}
		\node at (\x,0) {$\bullet$};
		\node at (-2,3) {$S$};
		\node at (-3,0) {$T$};
	\end{tikzpicture}
	\caption{A morphism $(f,m_S,m_T)$ in $\mathsf{Br}_d(S,T)$ from a $6$-element set $S$ to a $8$-element set $T$. Here $f$ is a bijection between $4$-element subsets of $S$ and $T$, $m_S$ is a matching on two elements of $S$ and $m_T$ is a matching on four elements of $T$. This is not the downward Brauer category, as the bottom pairing $m_T$ is not $\varnothing$.}	\label{fig:brauerorthogonal}
\end{figure}

\item Composition $\mathsf{Br}_d(S,T) \otimes \mathsf{Br}_d(T,U) \to \mathsf{Br}_d(S,U)$ is given in terms of such pictures by concatenating, then removing the closed components and multiplying by $d^{\text{number of closed components}}$.\end{itemize}\end{definition}

\begin{definition}
The \emph{downward Brauer category} $\mathsf{dBr} \subset \mathsf{Br}_d$ contains all objects but only those morphisms with $m_T = \varnothing$. We consider it as being spanned by pictures as in Figure \ref{fig:brauerdownward}. In this case concatenation can never form closed components, so this category is independent of the charge $d$. We write $i \colon \mathsf{dBr} \to \mathsf{Br}_d$ for the inclusion.
\end{definition}

Both of these categories are symmetric monoidal under disjoint union. It is $\cat{dBr}$ that will serve the role of $\Lambda$ in the general framework discussed in Section \ref{sec:RepOfCat}; it is easily seen to satisfy the assumptions of Definition \ref{def:2point9}.

Consider the functor $K \colon \mathsf{Br}_{2g} \to \cat{Rep}(G)$ given on objects by $K(S) = H(g)^{\otimes S}$. On a morphism $(f, m_S, m_T) \colon S \to T$, with bijection $f \colon S^\circ \to T^\circ$ between the complement of the matchings, it is given by
\[H(g)^{\otimes S} \xrightarrow{m_S} H(g)^{\otimes S^\circ} \xrightarrow{H(g)^{\otimes f}} H(g)^{\otimes T^\circ} \xrightarrow{m_T} H(g)^{\otimes T}\]
where the first map applies $\lambda$ to the pairs in $m_S$, and the last map applies $\omega$ to create the pairs in $m_T$. This functor has an evident symmetric monoidality. By taking linear duals of the values of $K$ on objects as well as its value of morphisms, we get a functor $K^\vee \colon (\mathsf{Br}_{2g})^\mr{op} \to \cat{Rep}(G)$. Restricting this functor along $i \colon \mathsf{dBr} \to \mathsf{Br}_d$ gives a functor $i^* K^\vee \colon \mathsf{dBr}^\mr{op} \to \cat{Rep}(G)$.

\begin{proposition}\label{prop:Recognition}
Let $B \in \cat{Rep}(G)$, $A \in  (\bQ\text{-mod})^\mathsf{dBr}$ have finite length, and there be given a map
\[\phi^\mathsf{Br_{2g}} \colon i_*(A) \lra [K \otimes B]^G \in (\bQ\text{-mod})^{\mathsf{Br}_{2g}}.\]
Then there is an induced map
\[\phi \colon i^*(K^\vee) \otimes^{\mathsf{dBr}} (1_{\cat{Rep}(G)} \odot A) \lra B  \in \cat{Rep}(G)\]
which is an isomorphism onto the maximal algebraic subrepresentation of $B$ if $\phi^\mathsf{Br_{2g}}$ an isomorphism, and is a monomorphism if $\phi^\mathsf{Br_{2g}}$ is a monomorphism.

If $\phi^\mathsf{Br_{2g}}$ is an isomorphism, then for a partition $\lambda$ of $q$ the multiplicity of the irreducible $G$-representation $V_\lambda(H(g))$ in $B$ is the same as the multiplicity of the irreducible $\Sigma_q$-representation $S^\lambda$ in $A(\{1,2,\ldots, q\})$.\footnote{Part of the claim is that if $V_\lambda(H(g))$ is not irreducible, so is zero by Theorem \ref{thm:RepsOfSpAndO}, then $S^\lambda$ does not occur in $A(\{1,2,\ldots, q\})$.}
\end{proposition}

\begin{proof}
The map $\phi^\mathsf{Br_{2g}}$ is adjoint to a map $\phi^\mathsf{dBr} \colon A \to i^*([K \otimes B]^G) = [i^*(K \otimes B)]^G$, and as $\Hom_{\cat{Rep}(G)}(1_{\cat{Rep}(G)}, -) = [-]^G$ this is adjoint to a map $1_{\cat{Rep}(G)} \odot A \to i^*(K \otimes B)$ in $\cat{Rep}(G)^\mathsf{dBr}$, whose adjoint is the map $\phi$ in the statement.

We apply the criterion of Lemma \ref{lem:Detection} to $\phi$. As discussed above, we will take $\cA = \cat{Rep}(G)$ the category of finite dimensional representations of $G$, and $\Lambda = \mathsf{dBr}$ the downward Brauer category. The functor $\Delta$ will be given by $[i^*(K) \otimes -]^G$, and hence we must verify that the morphism
\[[i^*(K) \otimes \phi]^G \colon [i^*(K) \otimes (i^*(K^\vee) \otimes^{\mathsf{dBr}} (1_{\cat{Rep}(G)} \odot A))]^G \lra [i^*(K) \otimes B]^G \in (\bQ\text{-mod})^\mathsf{dBr}\]
is an isomorphism or monomorphism. We will do this by relating it to $\phi^{\mathsf{Br}_{2g}}$, which is an isomorphism or monomorphism by assumption. Using the coend formula for $(i^*(K^\vee) \otimes^{\mathsf{dBr}} (1_{\cat{Rep}(G)} \odot A))$, we can write the source evaluated at $S \in \mathsf{dBr}$ as
\[\left[\int^{T \in \mathsf{dBr}} K(S) \otimes K(T)^\vee \odot A(T) \right]^G,\]
and as $[-]^G$ is an exact functor on $\cat{Rep}(G)$ we can evaluate this as
\[\int^{T \in \mathsf{dBr}} \left[K(S) \otimes K(T)^\vee\right]^G \otimes_\bQ A(T).\]
Now there is a natural transformation of two variables
\[\kappa \colon \mathsf{Br}_{2g}(T, S) \lra \left[K(S) \otimes K(T)^\vee\right]^G\]
given by the functoriality of $K$, which is surjective by Theorem \ref{thm:FTInvariantTheory}. This gives a surjection
\[\kappa: i_*(A)(S) = \int^{T \in \mathsf{dBr}} \mathsf{Br}_{2g}(T, S) \otimes_\bQ A(T) \lra \int^{T \in \mathsf{dBr}} \left[K(S) \otimes K(T)^\vee\right]^G \otimes_\bQ A(T).\]
As the composition
\[\phi^{\mathsf{Br}_{2g}}(S) \colon i_*(A)(S) \overset{\kappa}\lra \int^{T \in \mathsf{dBr}} \left[K(S) \otimes K(T)^\vee\right]^G \otimes_\bQ A(T) \xrightarrow{[K(S) \otimes \phi]^G} [K(S) \otimes B]^G\]
is a monomorphism by assumption, this shows that the first map is also injective and so in fact an isomorphism, from which it follows that $[K(S) \otimes \phi]^G$ an isomorphism or monomorphism whenever $\phi^{\mathsf{Br}_{2g}}(S)$ is.

It then follows from Lemma \ref{lem:Detection} that if $\phi^{\mathsf{Br}_{2g}}$ is a monomorphism then the kernel of $\phi$ lies in $\cat{Rep}(G)_K^\circ$, and if it is an isomorphism then the cokernel of $\phi$ does too. Unwrapping the definition, $\cat{Rep}(G)_K^\circ$ is precisely the category of finite dimensional $G$-representations $V$ which contain no algebraic subrepresentation (by Theorem \ref{thm:RepsOfSpAndO2}). The kernel of $\phi$ is a subrepresentation of $i^*(K^\vee) \otimes^{\mathsf{dBr}} (1_{\cat{Rep}(G)} \odot A)$, which is algebraic, so $\mr{Ker}(\phi)$ is also algebraic: if it lies in $\cat{Rep}(G)_K^\circ$ it is therefore zero, so $\phi$ is injective. If the cokernel of $\phi$ lies in $\cat{Rep}(G)_K^\circ$ then it contains no algebraic subrepresentations, so the image of $\phi$ is the maximal algebraic subrepresentation of $B$.

For the last part, we use the isomorphism
\[\phi^{\mathsf{Br}_{2g}} \colon i_*(A)(\{1,2,\ldots, q\}) \lra [H(g)^{\otimes q} \otimes B]^G\]
of $\Sigma_q$-representations. Taking the kernels of all the maps induced by $(\mr{inc} \colon S \to S', m_S, \varnothing)$ with $m_S$ nontrivial, we get an isomorphism of $\Sigma_q$-representations
\[A(\{1,2,\ldots, q\}) \lra [H(g)^{[q]} \otimes B]^G.\]
By \eqref{eq:SchurWeyl} we may write the right-hand side as $\bigoplus_{\lambda \vdash q} S^\lambda \otimes [V_\lambda(H(g)) \otimes B]^G$, so as the $S^\lambda$ are distinct irreducible $\Sigma_q$-representations we have
\[\dim_\bQ[S^\lambda \otimes A(\{1,2,\ldots, q\})]^{\Sigma_q} = \dim_\bQ [V_\lambda(H(g)) \otimes B]^G\]
as required.
\end{proof}

\begin{proposition}\label{prop:TransfOfTransf}
For $A \in  (\bQ\text{-mod})^\mathsf{dBr}$ there is a morphism
\[\psi^{\mathsf{Br_{2g}}} \colon i_*(A) \lra [K \otimes (i^*(K^\vee) \otimes^{\mathsf{dBr}} (1_{\cat{Rep}(G)} \odot A))]^G \in (\bQ\text{-mod})^{\mathsf{Br_{2g}}}\]
which is an epimorphism and, if $A$ satisfies $A(T)=0$ for all finite sets $T$ with $|T| \geq N$, is an isomorphism when evaluated on sets $S$ with $|S| \leq 2g-N+1$.
\end{proposition}
\begin{proof}
We define $\psi^{\mathsf{Br_{2g}}}$ by declaring its adjoint to be the map
\[\psi^{\mathsf{dBr}} \colon A \lra [i^*(K) \otimes (i^*(K^\vee) \otimes^{\mathsf{dBr}} (1_{\cat{Rep}(G)} \odot A))]^G \in (\bQ\text{-mod})^{\mathsf{dBr}}\]
which at the object $S \in \mathsf{dBr}$ is
\[A(S) \xrightarrow{\mr{coev} \otimes A(S)} [K(S) \otimes K(S)^\vee]^G \otimes A(S) 
\xrightarrow{\mr{inc}} \left[\int^{T \in \mathsf{dBr}} K(S) \otimes K(T)^\vee \otimes A(T) \right]^G.\]
One may verify that these form the components of a natural transformation of functors, i.e.\ a morphism in $(\bQ\text{-mod})^{\mathsf{dBr}}$.

As in the proof of Proposition \ref{prop:Recognition}, there is a natural transformation of two variables
\[\kappa \colon \mathsf{Br}_{2g}(T, S) \lra \left[K(S) \otimes K(T)^\vee\right]^G\]
given by the functoriality of $K$, which is an epimorphism by Theorem \ref{thm:FTInvariantTheory} and is an isomorphism if $2g \geq |S| + |T|$. Evaluating the map $\psi^{\mathsf{Br_{2g}}}$ at $S \in \mathsf{dBr}$, using the coend formula for left Kan extension, gives
\[i_*(A)(S) = \int^{T \in \mathsf{dBr}} \mathsf{Br}_{2g}(T, S) \otimes A(T) \lra \int^{T \in \mathsf{dBr}} \left[K(S) \otimes K(T)^\vee \right]^G \otimes A(T)\]
and this is identified with the map on coends induced by the bifunctor $\kappa$. As $\kappa$ is an epimorphism, so is $\psi^{\mathsf{Br_{2g}}}$. The map
\[\mathsf{Br}_{2g}(T, S) \otimes A(T) \lra  \left[K(S) \otimes K(T)^\vee \right]^G \otimes A(T)\]
is an isomorphism if $|T| \geq N$, as then both sides are zero because $A(T)=0$. It is also an isomorphism if $2g \geq |S| + |T|$. Thus it is an isomorphism for all sets $T$ as long as $|S| \leq 2g-N+1$, and so $\psi^{\mathsf{Br_{2g}}}(S)$ is also an isomorphism under this condition.
\end{proof}

\begin{corollary}\label{cor:TransfDetectsZero}
If $A \in  (\bQ\text{-mod})^\mathsf{dBr}$ is such that $A(T)=0$ for all finite sets $T$ with $|T| \geq g+1$, then $A=0$ if and only if $i^*(K^\vee) \otimes^{\mathsf{dBr}} (1_{\cat{Rep}(G)} \odot A)=0$. 

More generally, if $\phi$ is a map between such objects, then it is an epimorphism (resp.\ monomorphism) if and only if $(i^*(K^\vee) \otimes^{\mathsf{dBr}} (1_{\cat{Rep}(G)} \odot \phi))$ is.
\end{corollary}
\begin{proof} The implication $\Rightarrow$ is obvious, so we prove $\Leftarrow$ and suppose $i^*(K^\vee) \otimes^{\mathsf{dBr}} (1_{\cat{Rep}(G)} \odot A)=0$. Under the given condition, by Proposition \ref{prop:TransfOfTransf} the map
\[\psi^{\mathsf{Br_{2g}}} \colon i_*(A) \lra [K \otimes (i^*(K^\vee) \otimes^{\mathsf{dBr}} (1_{\cat{Rep}(G)} \odot A))]^G \in (\bQ\text{-mod})^{\mathsf{Br_{2g}}}\]
is an isomorphism when evaluated on sets $S$ with $|S|\leq g$, and so $i_*(A)(S)=0$ for such sets. But as every morphism in $\mathsf{Br}_{2g}$ factors uniquely as a morphism in the downward Brauer category followed by a morphism in the analogous upward Brauer category $\mathsf{uBr}$, up to isomorphisms of the intermediate object, we have
\[i_*(A)([n]) \cong \bigoplus_{2k \leq n} \mathsf{uBr}([n-2k], [n]) \otimes_{\fS_{n-2k}} A([n-2k])\]
and in particular $A(S)$ injects into $i_*(A)(S)$, so $A$ vanishes on sets of size at most $g$. But by assumption it also vanish on sets of size at least $g+1$, so $A=0$.

For the more general case, apply the above to the kernel and cokernel of $\phi$.
\end{proof}

\subsubsection{The symplectic group}

The discussion in the previous section goes through for symplectic groups rather than orthogonal groups with some minor changes, which we record here. Let $G \subset \Sp_{2g}(\bQ)$ be an arithmetic subgroup and $\cA=\cat{Rep}(G)$ its category of finite-dimensional representations. We shall suppose that $g \geq 2$ so that the functor $[-]^G$ is exact and all extensions split. The standard $2n$-dimensional representation $H(g) \in \cat{Rep}(G)$ is equipped with an antisymmetric pairing $\lambda$, and dually an alternating form $\omega$ (characterised by $(\lambda \otimes \mr{id})(- \otimes \omega) = \mr{id}(-)$).

\begin{definition}
The \emph{signed Brauer category} $\mathsf{sBr}_d$ of charge $d \in \bQ$ is the following $\bQ$-linear category.
\begin{itemize}
	\item The objects of $\cat{sBr}_d$ are the finite sets.
	\item The morphisms of $\mathsf{sBr}_d(S,T)$ are given by the following $\bQ$-vector space. First, let $\mathsf{sBr}_d(S,T)'$ be the vector space with basis given by triples $(f, m_S, m_T)$ of a bijection $f$ from a subset $S^\circ \subset S$ to a subset $T^\circ \subset T$, a matching $m_S$ of $S \setminus S^\circ$, and a matching $m_T$ of $T \setminus T^\circ$. Let $\mathsf{sBr}_d(S,T)$ be the quotient vector space by the subspace generated by $(f, m_S, m_T) - (-1)^r(f, m'_S, m'_T)$ where $m'_S$ and $m'_T$ differ from  $m_S$ and $m_T$ by reversing precisely $r$ pairs. We consider it as being spanned by pictures as in Figure \ref{fig:brauersymp}, where reversing a matched edge changes the picture by a sign.
	\begin{figure}[h]
		\begin{tikzpicture}
		\draw [->-=.5] (-1,0) to[out=90,in=90,looseness=1.5] (1,0);
		\draw [->-=.5] (3,0) to[out=90,in=90,looseness=1.5] (2,0);
		\draw [->-=.5] (-1,3) to[out=-90,in=-90,looseness=1.5] (0,3);
		\draw (1,3) to[out=-90,in=90] (-2,0);
		\draw [line width=2mm,white] (2,3) to[out=-90,in=90] (0,0);
		\draw (2,3) to[out=-90,in=90] (0,0);
		\draw (3,3) to[out=-90,in=90] (4,0);
		\draw (4,3) to[out=-90,in=90] (5,0);
		\foreach \x in {-1,...,4}
		\node at (\x,3) {$\bullet$};
		\foreach \x in {-2,...,5}
		\node at (\x,0) {$\bullet$};
		\node at (-2,3) {$S$};
		\node at (-3,0) {$T$};
		\end{tikzpicture}
		\caption{A morphism $(f,m_S,m_T)$ in $\mathsf{sBr}_d(S,T)$ from a $6$-element set $S$ to a $8$-element set $T$. Here $f$ is a bijection between $4$-element subsets of $S$ and $T$, $m_S$ is a matching on two elements of $S$ and $m_T$ is a matching on four elements of $T$.}\label{fig:brauersymp}
\end{figure}
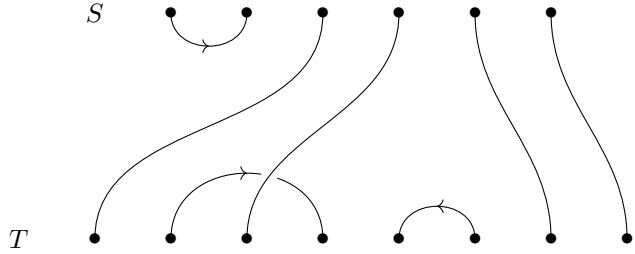

\item Composition $\mathsf{sBr}_d(S,T) \otimes \mathsf{sBr}_d(T,U) \to \mathsf{sBr}_d(S,U)$ is given in terms of such pictures by concatenating (arranging that any matched edges that are concatenated have compatible orientations), then removing the closed components and multiplying by $d^c$ with $c$ the number of closed components.
\end{itemize}  
\end{definition}

\begin{definition}

The \emph{downward signed Brauer category} $\mathsf{dsBr} \subset \mathsf{sBr}_d$ contains all objects but only those morphisms with $m_T = \varnothing$. In this case concatenation can never form closed components, so this category is independent of the charge $d$. We write $i \colon \mathsf{dsBr} \to \mathsf{sBr}_d$ for the inclusion.\end{definition}

Both of these categories are symmetric monoidal under disjoint union. Just as in the orthogonal case, there is a symmetric monoidal functor $K \colon \mathsf{sBr}_{2g} \to \cat{Rep}(G)$ given by the same formula. Using this object, the statements of Proposition \ref{prop:Recognition}, Proposition \ref{prop:TransfOfTransf}, and Corollary \ref{cor:TransfDetectsZero} hold verbatim, and are proved completely analogously.

\section{Twisted Miller--Morita--Mumford classes}\label{sec:TwistedMMM}

Recall that $W_{g}$ denotes the manifold $\#^g S^n \times S^n$.  Fix a fibration $\theta \colon B \to B\mr{SO}(2n)$. In this section we wish to attach characteristic classes in twisted cohomology to the following data: a smooth oriented $W_g$-bundle $\pi \colon E \to X$ with section $s \colon X \to E$, and a choice of lift $\ell \colon E \to B$ of the map $\tau_\pi \colon E \to B\mr{SO}(2n)$ classifying the oriented vertical tangent bundle $T_\pi E \to E$. We can summarise this data in the following diagram:
\begin{equation}\label{eq:BundleData}
\begin{tikzcd}
W_g \dar{i} & B \dar{\theta}\\
E \arrow[ru, "\ell"]\dar{\pi}\rar[swap]{\tau_\pi}& B\mr{SO}(2n)\\
X. \arrow[bend left=60]{u}{s}
\end{tikzcd}
\end{equation}

We will write $\cH(g)$ for the local coefficient system on $X$ with $\cH(g)_x = H_n(\pi^{-1}(x) ; \bZ)$, which is equipped with a $(-1)^n$-symmetric nondegenerate pairing $\lambda \colon \cH(g) \otimes \cH(g) \to \bZ$ given by the intersection form (with respect to the given orientations of the fibres). For a commutative ring $\bk$ we write $\cH(g)_\bk \coloneqq \cH(g) \otimes_\bZ \bk$ for the associated local system of $\bk$-modules. 

We shall explain how to construct certain characteristic classes with coefficients in tensor powers of $\cH(g)$, following Kawazumi \cite{KawazumiInvent, Kawazumi} (see also Kawazumi--Morita \cite{KM, KMunpub}) who considered this situation for $2n=2$. Our goal is to associate to the data above and to any partition $\{P_i\}_{i \in I}$ of a finite set $S$ and label $c_i \in H^{2d_i}(B;\bk)$ of each part $P_i$, an element
\[\kappa(\{P_i\}, \{c_i\}) \in H^{*}(X ; \cH(g)^{\otimes S}_\bk) \otimes (\det \bk^{S})^{\otimes n}\]
of degree $\sum_{i \in I} n(|P_i|-2)+2d_i$, which transforms under the symmetric group of $S$ in the expected way. Here and later for a finitely-generated free $\bk$-module $M$ we write $\det M$ for its top exterior power.

\subsection{Gysin homomorphism}

For any local coefficient system of $\bk$-modules $\cM$ on $X$, the fibration sequence
\[W_{g} \lra E \overset{\pi}\lra X\]
has an associated cohomological Serre spectral sequence
\begin{equation}\label{eqn:gysin-relative-serre-ss}
E^{p,q}_2 = H^p(X; \cH^q(W_{g};\bk) \otimes_{\bk} \cM) \Longrightarrow H^{p+q}( E ; \pi^* \cM)
\end{equation} 
with three non-zero rows, the $0$th, $n$th and $2n$th.

The map 
\[\pi^* \colon H^*(X;\cM) \lra H^*(E, \pi^*\cM)\]
is split injective, as $s^*$ gives a one-sided inverse for it. This splits off the $q=0$ row of the spectral sequence.

The local coefficient system $\cH^{2n}(W_{g};\bk)$ is trivial, because we have assumed that the bundle $\pi \colon E \to X$ is oriented. It follows that the Serre spectral sequence has $E_2^{p, 2n} = H^p(X ; \cH^{2n}(W_{g};\bk) \otimes_{\bk} \cM)$ canonically identified with $H^p(X ; \cM)$, and so as usual projection to the $2n$th row defines a Gysin homomorphism
\[\pi_! \colon H^*(E;\pi^* \cM) \lra H^{*-2n}(X;\cM),\]
which is a homomorphism of right $H^{*}(X;\bk)$-modules.

\begin{lemma}\label{lem:VolClass}
There is a class $v \in H^{2n}(E;\bk)$ which restricts to a generator of the top cohomology of each fibre.
\end{lemma}
\begin{proof}
The homotopy cofibre of the inclusion $E \setminus s(X) \to E$ is identified with the Thom space of the normal bundle of $s(X) \subset E$, which is the restriction of $T_\pi E$ to $s(X)$. This yields a map
\[ E \lra \mathrm{Th}(s^* T_\pi E \to X),\]
and the pullback of the Thom class---which exists because $T_\pi E$ is oriented---along this map defines a class $v \in H^{2n}(E;\bk)$. This restricts to the Poincar{\'e} dual of a point in any fibre, which is a generator of the top cohomology.
\end{proof}

By pulling back to each point $* \in X$, we see that this class satisfies
\[\pi_!(v) = 1 \in H^0(X;\bk),\]
so for any $x \in H^{*-2n}(X;\cM)$ we have
\[\pi_!(v \cdot \pi^*(x)) = \pi_!(v) \cdot x = x,\]
and hence $v \cdot \pi^*(-)$ shows that $\pi_!(-)$ is split surjective. 

Using the above two splittings we see that the Serre spectral sequence (\ref{eqn:gysin-relative-serre-ss}) collapses, and under the identification $\cH^n(W_g ; \bk) = \cH(g)^\vee$ we obtain a preferred decomposition
\begin{equation}\label{eq:KMDecomposition}
H^*(E;\pi^*\cM) = H^*(X;\cM) \oplus H^{*-n}(X; \cH(g)^\vee \otimes \cM) \oplus H^{*-2n}(X ; \cM).
\end{equation}

\subsection{A twisted cohomology class} \label{sec:semi-euler-class} 

The Serre spectral sequence (\ref{eqn:gysin-relative-serre-ss}) with coefficients in the local coefficient system $\cH(g)_\bk$ on $X$ has the form
\[E^{p,q}_2 = H^p(X; \cH^q(W_{g};\bk) \otimes \cH(g)_\bk) \Longrightarrow H^{p+q}( E ; \pi^* \cH(g)_\bk).\]
As $\cH^n(W_{g,1};\bk) = \cH(g)_\bk^\vee$ we have $E^{0,n}_2 = H^0(X ; \cH(g)_\bk^\vee \otimes \cH(g)_\bk)$, which contains a canonical element given by coevaluation; that is, the adjoint to the identity map of $\cH(g)_\bk$. Using the decomposition \eqref{eq:KMDecomposition} for this spectral sequence, $coev$ defines a unique class 
\[\epsilon \in H^n( E ; \pi^* \cH(g)_\bk).\]
(This extends to higher dimensions a class constructed by Morita \cite[Section 6]{MoritaJacobian1}.) By construction, $\epsilon$ is characterised by its restriction to any fibre and the properties $s^*(\epsilon)=0$ and $\pi_!(\epsilon)=0$.

\subsection{Defining twisted Miller--Morita--Mumford classes}

Given the data in \eqref{eq:BundleData} and a class $c \in H^{2d}(B;\bk)$, we can define
\[{\pi}_!({\epsilon}^k \cdot \ell^*c) \in H^{n(k-2)+2d}(X;\cH(g)_\bk^{\otimes k}).\]
This will be an example of a twisted Mumford--Morita--Miller class. More generally, to a partition $(p_1, p_2, \ldots, p_r)$ of the number $k$ with $p_i \geq p_{i+1}$, in which $p_i=0$ is allowed, we associate the standard partition $\cP_\mr{std}(p_1, p_2, \ldots, p_r)$ of the set $\{1,2,\ldots, k\}$ given by
\[\{\{1,2,\ldots, p_1\}, \{p_1+1, \ldots, p_1+p_2\}, \ldots, \{p_1 + \cdots + p_{r-1}+1, \ldots, p_1 + \cdots + p_r\}\},\]
where the $i$th subset is taken to be empty if $p_i=0$. Given classes $c_i \in H^{2d_i}(B;\bk)$ for $i=1,2,\ldots, r$, we assign the class $\kappa((p_i); (c_i))$ of degree $\sum_{i=1}^r n(p_i-2) + 2d_i$ defined as
\begin{align*}
&{\pi}_!({\epsilon}^{p_1} \cdot \ell^*c_1)  \cdots {\pi}_!({\epsilon}^{p_r} \cdot \ell^*c_r) \otimes (e_1 \wedge \cdots \wedge e_k)^{\otimes n} \in H^{*}(X ; \cH(g)^{\otimes k}_\bk) \otimes (\det \bk^{k})^{\otimes n}.\end{align*}

For a set $S$ of cardinality $k$ and a partition $\{P_1, \ldots, P_r\}$ of $S$ into parts of sizes $p_1, \ldots, p_r$ with $p_i \geq p_{i+1}$ and where empty parts are allowed, we may choose a bijection $\phi \colon [k] \overset{\sim}\to X$ sending each $P_i$ to $\{p_1 + \cdots + p_{i-1}+1, \ldots, p_1 + \cdots + p_i\}$, and hence sending the partition $\{P_1, \ldots, P_r\}$ of $X$ to the standard partition $\cP_{std}(p_1, p_2, \ldots, p_r)$ of $[k]$; there is an induced isomorphism
\[\phi_* \colon H^{*}(X ; \cH(g)_\bk^{\otimes k}) \otimes (\det \bk^{k})^{\otimes n} \overset{\sim}\lra H^{*}(X ; \cH(g)_\bk^{\otimes S})  \otimes (\det \bk^{S})^{\otimes n}.\]
We wish to define
\[\kappa(\{P_i\}, \{c_i\}) \coloneqq \phi_*(\kappa((p_i); (c_i))) \in H^{*}(X ; \cH(g)_\bk^{\otimes S})  \otimes (\det \bk^{S})^{\otimes n}.\]

\begin{lemma}\label{lem:twisted-mmm-welldefined}
This is well-defined.
\end{lemma}
\begin{proof}
If $\psi$ is another such bijection, then $\psi^{-1} \circ \phi \colon [k] \to [k]$ is a bijection which preserves the partition $\cP_\mr{std}(p_1, p_2, \ldots, p_r)$. If $s_j \coloneqq |\{1 \leq i \leq r \, \vert \, p_i = j\}|$ denotes the number of parts of size $j$, then the subgroup of $\Sigma_k$ of permutations which preserve the partition $\cP_\mr{std}(p_1, p_2, \ldots, p_r)$ may be identifed with
\[\prod_{j=1}^k \Sigma_{j} \wr \Sigma_{s_j} \leq \Sigma_k.\]
Thus it is generated by arbitrary permutations of the elements of the parts
\[Q_i \coloneqq \{p_1 + \cdots + p_{i-1}+1, \ldots, p_1 + \cdots + p_i\},\]
as well as permutations of non-empty parts $Q_i$ having the same cardinality. 

A permutation $\sigma$ of $Q_i$ acts on ${\epsilon}^{p_i}$ by permuting the factors, and as $\epsilon$ has degree $n$ it therefore acts by $\mr{sign}(\sigma)^n$. Hence it acts on ${\pi}_!({\epsilon}^{p_i} \cdot \ell^*c_i)$ by $\mr{sign}(\sigma)^n$ too, so acts on $\kappa((p_i); (c_i))$ trivially.

A permutation of the set $\{Q_i \, \vert \, \vert Q_i\vert=j\}$ acts on ${\pi}_!({\epsilon}^{p_1} \cdot \ell^*c_1) \cdot {\pi}_!({\epsilon}^{p_2} \cdot \ell^*c_2) \cdots {\pi}_!({\epsilon}^{p_r} \cdot \ell^*c_r)$ by permuting the terms, and the group of such permutations is generated by transpositions of adjacent parts. A transposition $\sigma$ of adjacent $Q_i$'s involves $j^2$ transpositions in $\Sigma_k$, so has $\mr{sign}(\sigma)=(-1)^j$. On the other hand $\vert {\pi}_!({\epsilon}^{j} \cdot \ell^*c)\vert = n(j-2) + |c|$, so transposing two copies incurs a sign of $(-1)^{n(j-2)+|c|} = ((-1)^{j})^n$, as $|c|=|c_i|$ is even by assumption. Hence the subgroup of $\Sigma_k$ which preserves the standard partition acts trivially on the class $\kappa((p_i);(c_i))$.
\end{proof}

We have thus defined for each bundle as in \eqref{eq:BundleData}, and each partition $\{P_i\}_{i \in I}$ of a finite set $S$ and labels $c_i \in H^{2d_i}(B;\bk)$ of each part $P_i$, a \emph{twisted Miller--Morita--Mumford class}
\[\kappa(\{P_i\}, \{c_i\}) \in H^{*}(X ; \cH(g)^{\otimes S}_\bk) \otimes (\det \bk^{S})^{\otimes n}\]
of degree $\sum_{i \in I} n(|P_i|-2)+2d_i$.

For the remainder of this section we will write
\[\cV \coloneqq H^{*}(B;\bk),\]
for the graded $\bk$-module of labels, and suppose that it is concentrated in even degrees.

\begin{definition}
For a finite set $S$, let $\cP(S, \cV)$ be the graded $\bk$-module generated by partitions of $S$ (recall from Definition \ref{def:PartAlg} that partitions may have empty parts) with a labelling of each part by a homogeneous element of $\cV$, modulo $\bk$-linearity with respect to the labels. This module is graded by declaring a labelled partition $(\{P_i\}, \{c_i\})$ to have degree $\sum_{i \in I} n(|P_i|-2)+|c_i|$.
\end{definition}

\begin{remark}\label{rem:Bases1}
It is sometimes useful (when $\bk$ is a field) to choose a homogeneous basis $\cB$ of $\cV$, which gives a homogeneous basis of $\cP(S, \cV)$ given by those partitions of $S$ where each part is labelled by an element of $\cB$. 
\end{remark}

The above construction defines a $\Sigma_S$-equivariant map
\[\Phi''_S \colon \mathcal{P}(S, \cV) \lra H^{*}(X ; \cH(g)_\bk^{\otimes S})  \otimes (\det \bk^{S})^{\otimes n},\]
and hence by adjunction a $\Sigma_S$-equivariant map
\[\Phi'_S \colon \mathcal{P}(S, \cV)  \otimes (\det \bk^{S})^{\otimes n} \lra H^{*}(X ; \cH(g)_\bk^{\otimes S}).\]
(The map $\Phi''_\varnothing$ sends the empty partition of the empty set to $1 \otimes 1 \in H^0(X ; \bk) \otimes \bk$.)

By definition $\kappa(\{P_i\}, \{c_i\})$ is a cup product of classes, one for each part $P_i$ which up to the symmetric group action can be taken to be ${\pi}_!({\epsilon}^{p_i} \cdot \ell^*c_i)$. This has degree $n(p_i-2) + |c_i|$, so if $p_i=0$ and $|c_i| < 2n$, or $p_i=1$ and $|c_i| < n$, then it gives a cohomology class of negative degree and so vanishes. Furthermore if $p_i=0$ and $|c_i| = 2n$ it gives a degree zero cohomology class with $\bk$-coefficients, i.e.\ a scalar.

\begin{definition}
Writing $\varphi \colon \cV_{2n} \overset{\ell^*}\to H^{2n}({E};\bk) \overset{i^*}\lra H^{2n}(W_g;\bk) = \bk$, we let $\mathcal{P}(S, \cV)_{\geq 0}$ be the quotient of $\mathcal{P}(S, \cV)$ by the submodule generated by those labelled partitions having some part of size 0 and label of degree $<2n$, or some part of size 1 and label of degree $<n$, as well as by the differences
\[(\{P_i\}_{i \in I}, \{c_i\}_{i \in I}) - (\{P_i\}_{i \in I \setminus j}, \{c_i\}_{i \in I \setminus j}) \cdot \varphi(c_j)\]
whenever $P_j = \{\varnothing\}$ and $c_j$ has degree $2n$.
\end{definition}

\begin{remark}\label{rem:Bases2}
As in Remark \ref{rem:Bases1}, if we choose a homogeneous basis $\cB$ for $\cV$ then we obtain a homogeneous basis for $\mathcal{P}(S, \cV)_{\geq 0}$ given by those partitions of $S$ where each part is labelled by elements of $\cB$, having no parts (i) of size 0 with label of degree $ \leq 2n$, or (ii) of size 1 with label of degree $<n$. This description presents $\mathcal{P}(S, \cV)_{\geq 0}$ as a \emph{subspace} of $\mathcal{P}(S, \cV)$.
\end{remark}

By the discussion above the map $\Phi_S'$ factors over a map
\[\Phi_S \colon \mathcal{P}(S, \cV)_{\geq 0}  \otimes (\det \bk^{S})^{\otimes n} \lra H^{*}(X ; \cH(g)_\bk^{\otimes S}).\]

\begin{remark}\label{rem:bundle-types}
The construction of the twisted Miller--Morita--Mumford classes can be done with weaker input than \eqref{eq:BundleData}. All that is required is a family $\pi \colon E \to X$ with general fibre $W_g$ and section, regular enough to have a Serre spectral sequence, and a source of cohomology classes on $E$.
	
For example, we may take PL or topological $W_g$-bundles with section instead of smooth $W_g$-bundles at the cost of replacing $B\mr{SO}(2n)$ with $B\mr{SPL}(2n)$ or $B\mr{STop}(2n)$ respectively and (vertical) tangent bundles with (vertical) tangent microbundles. More generally, we may take (smooth, PL, or topological) block $W_g$-bundles with section: in \cite[Proposition 2.8]{ebertrwmmm} it is shown that a block bundle is a weak quasifibration so has a Serre spectral sequence; in \cite[Proposition 3.2]{ebertrwmmm} it is shown that a smooth block bundle has a stable vertical tangent bundle, and in \cite[Section 2]{hllrw} this is extended to PL or topological block bundles; finally, in \cite[Section 3]{hllrw} it is shown that a block bundle (and even a fibration with Poincar{\'e} fibre) has a fibrewise Euler class. Then the construction of $\kappa_{\epsilon^k c}$ with $c$ a monomial in Euler and Pontrjagin classes can be made.
\end{remark}

\subsection{Functoriality with respect to bijections}\label{sec:functoriality-FB}

Let $\cat{FB}$ denote the category of finite sets and bijections. Define a functor
\[\mathcal{P}(-, \cV)_{\geq 0} \colon \mathsf{FB} \lra \mathsf{Gr}(\bk\text{-mod})\]
by sending a finite set $S$ to the $\bk$-module $\mathcal{P}(S, \cV)_{\geq 0}$, and sending a bijection $f \colon S \to T$ to the $\bk$-linear map induced by relabelling elements. Taking the objectwise tensor product with the $n$th power of the sign functor gives a functor
\[\mathcal{P}(-, \cV)_{\geq 0} \otimes {\det}^{\otimes n} \colon \mathsf{FB} \lra \mathsf{Gr}(\bk\text{-mod}).\]
It follows from Lemma \ref{lem:twisted-mmm-welldefined} that the $\Phi_S$ determine a natural transformation of functors
\[\Phi \colon \mathcal{P}(-, \cV)_{\geq 0} \otimes {\det}^{\otimes n} \Longrightarrow H^{*}(B ; \cH(g)_\bk^{\otimes -}) \colon \mathsf{FB} \lra \mathsf{Gr}(\bk\text{-mod}).\]

\subsection{Functoriality on the Brauer category}\label{sec:functoriality-brauer}

We now wish to determine how the maps $\Phi_S$, the pairing $\lambda \colon \cH(g)_\bk \otimes \cH(g)_\bk \to \bk$, and its dual, the form $\omega \colon \bk \to \cH(g)_\bk \otimes \cH(g)_\bk$, interact. More precisely, for an ordered pair of elements $x, y \in S$ there is a map
\[\lambda_{x,y} \colon \cH(g)_\bk^{\otimes S} \lra \cH(g)_\bk^{\otimes S \setminus \{x,y\}}\]
of local coefficient systems on $X$ given by applying $\lambda$ to the $x$th and $y$th factors, and a map
\[\omega_{x,y} \colon \cH(g)_\bk^{\otimes S \setminus \{x,y\}} \lra \cH(g)_\bk^{\otimes S}\]
given by inserting $\omega$ in these factors, and we wish to determine the induced maps
\begin{align*}
\lambda_{x,y} \colon H^{*}(X ; \cH(g)_\bk^{\otimes S}) \lra H^{*}(X ; \cH(g)_\bk^{\otimes S \setminus \{x,y\}})\\
\omega_{x,y} \colon H^{*}(X ; \cH(g)_\bk^{\otimes S \setminus \{x,y\}}) \lra H^{*}(X ; \cH(g)_\bk^{\otimes S}) 
\end{align*}
on the classes we have just defined. By the equivariance and multiplicativity results we have already established, it is enough to
\begin{enumerate}[\indent (i)]
\item consider only the case $S=\{1,2,\ldots, k\}$,

\item determine $\omega_{1,2}(1)$,

\item determine $\lambda_{1,2}({\pi}_!({\epsilon}^k \cdot \ell^*c))$,

\item determine $\lambda_{a, a+1}({\pi}_!({\epsilon}^a \cdot \ell^*c) \cdot {\pi}_!({\epsilon}^{k-a} \cdot \ell^*c'))$. 
\end{enumerate}

In the following we will make use of the cap product. For the avoidance of doubt we emphasise that we adopt sign conventions such that the cap product makes homology into a left module over the cohomology ring.

\begin{proposition}
We have $\omega_{1,2}(1) =  {\pi}_!({\epsilon}^2) \in H^0(X ; \cH(g)_\bk^{\otimes 2})$.
\end{proposition}
\begin{proof}
By naturality, we can test this identity by restricting to a point $* \in X$, i.e.\ considering the fibre bundle $\pi \colon W_{g} \to *$. In this case $\pi_!(\epsilon^2)$ is $\langle \epsilon \cdot \epsilon, [W_g] \rangle$. Writing $\{b_i\}$ for a basis of $H_n(W_g;\bk)$, and $\{b_i^*\}$ for the dual basis of $H^n(W_g;\bk)$, the class $\epsilon \in H^n(W_g ; \bk) \otimes H_n(W_g ; \bk)$ may be written as $\sum_i b_i^* \otimes b_i$. Let $a_i \in H^n(W_g;\bk)$ be Poincar{\'e} dual to $b_i$, so that $b_i = a_i \frown [W_g]$, and $\{a_i^*\}$ be the corresponding dual basis for $H_n(W_g;\bk)$. Then $\epsilon$ may also be written as $\sum_i a_i \otimes a_i^*$. Thus
\[\langle \epsilon \cdot \epsilon, [W_g] \rangle = \sum_{i,j} \langle b_i^* \cdot a_j, [W_g] \rangle \cdot b_i \otimes a_j^*.\]
Now $\langle b_i^* \cdot a_j, [W_g] \rangle = \langle b_i^* , a_j \frown [W_g] \rangle = \langle b_i^*, b_j \rangle = \delta_{ij}$, so
\[\langle \epsilon \cdot \epsilon, [W_g] \rangle = \sum_i b_i \otimes a_i^* \in H_n(W_{g};\bk) \otimes H_n(W_{g};\bk).\]
On the other hand $\omega_{1,2}(1) = \omega \in H_n(W_{g};\bk) \otimes H_n(W_{g};\bk)$. Let $\{b_i^\#\}$ be the $\lambda$-dual basis of $H_n(W_{g};\bk)$, characterised by $\lambda(b_i^\#, b_j)=\delta_{ij}$. Then
\[(\lambda \otimes \mr{id}) \left(b_j^\# \otimes \sum_i b_i \otimes a_i^* \right) = \sum_i \lambda(b_j^\#, b_i) a_i^* = a_j^*.\]
However as $b_i = a_i \frown [W_g]$ we have $\lambda(a_j^*, b_i) = \langle a_j^*, a_i \rangle = \delta_{ij}$, so $a_j^* = b_j^\#$. Hence $(\lambda \otimes \mr{id})(b_j^\# \otimes \sum_i b_i \otimes a_i^*) = b_j^\#$ and so $\sum_i b_i \otimes a_i^* = \omega$ by the characterisation $(\lambda \otimes \mr{id})(- \otimes \omega) = \mr{id}(-)$ of $\omega$.
\end{proof}

In order to state the following lemma, recall that ${v} \in H^{2n}({E} ; \bk)$ is the class constructed in Lemma \ref{lem:VolClass}, which is fibrewise Poincar{\'e} dual to the section $s \colon X \to E$. In particular, if $T_\pi E \to E$ denotes the vertical tangent bundle, then $s^*(v) = s^*(e(T_\pi E))$. Write $p \colon E \times_X E \to X$ for the projection map of the fibre product of $\pi \colon E \to X$ with itself.

\begin{lemma}\label{lem:ContractionExterior}
We have
\[\lambda_{1,2}({\epsilon} \times {\epsilon}) = \Delta_!(1) - 1 \times {v} - {v} \times 1 + p^*s^*(e(T_\pi E)) \in H^{2n}(E \times_X E ;\bk).\]
\end{lemma}
\begin{proof}
The class $\epsilon$ satisfies $s^*(\epsilon)=0$, by definition, so lifts to a class $\rho \in H^n(E, s(X) ; \cH(g)_\bk)$. Thus the class $\epsilon \times \epsilon$ lifts to the class
\[\rho \times \rho \in H^{2n}(E \times_X E, (E \times_X s(X)) \cup(s(X) \times_X E) ; \cH(g)_\bk^{\otimes 2}).\]

The Serre spectral sequence for the relative fibration
\[(W_g \times W_g, W_g \vee W_g) \lra (E \times_B E, (E \times_X s(X)) \cup(s(X) \times_X E)) \overset{p}\lra X\]
has lowest row the $2n$th, so there is an isomorphism
\[H^0(X; \cH^n(W_{g}, *;\bk) \otimes \cH^n(W_{g}, *;\bk)) \overset{\sim}\lra H^{2n}(E \times_X E, (E \times_X s(X)) \cup(s(X) \times_X E); \bk),\]
and hence the class $\lambda_{1,2}(\rho \times \rho)$ is characterised by its restriction to a single fibre. Thus the class $\lambda_{1,2}({\epsilon} \times {\epsilon})$ is characterised by its restriction to a single fibre and the fact that it lifts to a class in $H^{2n}(E \times_X E , (E \times_X s(X)) \cup(s(X) \times_X E) ; \bk)$.

By definition, the restriction of $\epsilon$ to a fibre of $\pi$ corresponds, under the universal coefficient isomorphism
\[H^n(W_g; H_n(W_{g};\bk)) \cong \mr{Hom}(H_n(W_{g};\bk),H_n(W_{g};\bk)),\]
to the identity map $\mr{id}$. Thus the restriction of $\epsilon \times \epsilon$ to a fibre of $p$ corresponds, under the universal coefficient isomorphism
\[H^n(W_g \times W_g; H_n(W_{g};\bk)^{\otimes 2}) \cong \mr{Hom}(H_n(W_{g};\bk)^{\otimes 2},H_n(W_{g};\bk)^{\otimes 2}),\]
to $\mr{id} \otimes \mr{id}$, and so the restriction of $\lambda_{1,2}(\epsilon \times \epsilon)$ to a fibre of $p$ corresponds, under the universal coefficient isomorphism
\[H^n(W_g \times W_g; \bk) \cong \mr{Hom}(H_n(W_{g};\bk)^{\otimes 2}, \bk),\]
to the map $\lambda$. Concretely for classes $x, y \in H_n(W_g;\bk)$ we evaluate this by writing $x = X \frown [W_g]$ and $y = Y \frown [W_g]$ and then
\begin{equation}\label{eq:IntForm}
\lambda(x, y) = \langle X \cdot Y, [W_g] \rangle.
\end{equation}

Our strategy will now be to show that $\Delta_!(1) - 1 \times {v} - {v} \times 1 + p^*s^*(v)$ also lifts to a class in $H^{2n}(E \times_X E , (E \times_X s(X)) \cup(s(X) \times_X E) ; \bk)$, and that its restriction to a single fibre is also, under the universal coefficient isomorphism, the map $\lambda$.

The map $\Delta \colon E \to E \times_X E$ has oriented normal bundle and so a normal Thom class which may be extended to a class $\Delta_!(1) \in H^{2n}({E} \times_X {E} ; \bk)$. Let $x_1 \coloneqq \Delta_!(1)$. Pulled back along 
\[f_1 \coloneqq s \circ \pi \times \mr{id} \colon E \overset{\cong}\lra s(X) \times_X {E} \subset E \times_X E\]
the class $x_1$ is $v$, the the fibrewise Poincar{\'e} dual to  $s(X) \subset E$, and so the class $x_2 \coloneqq x_1 - 1 \times v$ vanishes when pulled back along $f_1$. Pulled back along 
\[f_2 \coloneqq \mr{id} \times s \circ \pi \colon E \overset{\cong}\lra {E} \times_X s(X) \subset E \times_X E\]
the class $x_2$ is $v - \pi^*s^*(v)$, and so the class $x_3 \coloneqq x_2 - v \times 1 + p^*s^*(v)$ vanishes when pulled back along $f_2$. Pulling $x_3$ back along $f_1$ again the term $x_2$ vanishes, and $- v \times 1 + p^*s^*(v)$ becomes $- \pi^*s^*(v) + \pi^*s^*\pi^*s^*(v)=0$ (as $s^*\pi^*=\mr{id}$). Thus
\[\Delta_!(1) - v \times 1 - 1 \times v + p^*s^*(v)\]
vanishes on $(E \times_X s(X)) \cup(s(X) \times_X E)$.

If we restrict to a single fibre $W_g \times W_g$ we may use the usual formula for the decomposition of the diagonal. Write $\{a_i\}$ for a basis for $H^n(W_{g};\bk)$ and $\{a_i^\#\}$ for the dual basis, characterised by $\langle a_i \cdot a_j^\#, [W_g] \rangle = \delta_{ij}$. Then, by \cite[Theorem 11.11]{MilnorStasheff} the class $\Delta_!(1)$ restricts to
\[{v} \otimes 1 + 1 \otimes {v} + \sum_{i} (-1)^n a_i \otimes a_i^\#\]
on the fibre $W_g \times W_g$. Thus the class $\Delta_!(1) - 1 \times {v} - {v} \times 1 + p^*s^*(v)$ restricts to $\sum_{i} (-1)^n a_i \otimes a_i^\# \in H^n(W_g;\bk) \otimes H^n(W_g;\bk)$. Evaluating this on classes $x = X \frown [W_g]$ and $y = Y \frown [W_g]$ as above gives
\begin{align*}
\left\langle \sum_{i} (-1)^n a_i \otimes a_i^\#, x \otimes y \right\rangle &= \sum_i \langle a_i, x \rangle  \langle a_i^\#, y \rangle\\
&= \sum_i \langle a_i \cdot X, [W_g] \rangle  \langle a_i^\# \cdot Y, [W_g] \rangle. 
\end{align*}
Evaluated at $X=a_k^\#$ and $Y= a_l$ this gives
\[\left\langle \sum_{i} (-1)^n a_i \otimes a_i^\#, x \otimes y \right\rangle= \sum_i \delta_{ik} (-1)^n \delta_{il} = (-1)^n \delta_{kl}\]
which is the same as \eqref{eq:IntForm} evaluated on these elements. As $\{a_k^\# \otimes a_l\}$ form a basis of $H^n(W_{g};\bk)^{\otimes 2}$ it follows that the restriction of $\Delta_!(1) - 1 \times {v} - {v} \times 1 + p^*s^*(v)$ to a single fibre also corresponds, under the universal coefficient isomorphism, to the map $\lambda$.

By the characterisation of $\lambda_{1,2}({\epsilon} \times {\epsilon})$ above, we therefore have
\[\lambda_{1,2}({\epsilon} \times {\epsilon}) = \Delta_!(1) - 1 \times {v} - {v} \times 1 + p^*s^*(v).\]
Finally, we have $s^*(v) = s^*(e(T_\pi E))$.
\end{proof}

The following proposition generalises the Contraction Formula of Kawazumi and Morita \cite[Theorem 6.23]{KMunpub} to higher dimensions.\footnote{There is an overall difference of sign from \cite{KMunpub}. This seems to be due to identification $\mu' \colon H_n(W_g;\bk) \overset{\sim}\to H_n(W_g;\bk)^\vee$ used by Kawazumi and Morita (pp.\ 16-17), which in our notation is given by the formula $\mu'(v)(u) = \lambda(u, v)$. Under the universal coefficient isomorphism $H_n(W_g;\bk)^\vee \cong H^n(W_g;\bk)$ this is not the inverse Poincar{\'e} duality isomorphism, but rather is $(-1)^n$ times it. We instead use the more natural identification given by Poincar{\'e} duality.}

\begin{proposition}\label{prop:ContractEpsilon}
For $k \geq 2$ and $|c| + n(k-2) \geq 0$ we have
\begin{align*}
\lambda_{1,2}({\pi}_!({\epsilon}^k \cdot \ell^*c)) &= {\pi}_!({\epsilon}^{k-2} \cdot e(T_{{\pi}} {E}) \cdot \ell^*c) + s^* e(T_\pi E) \cdot {\pi}_!({\epsilon}^{k-2} \cdot \ell^*c)\\
& \qquad- \begin{cases}
2s^*(\ell^*c) & \text{ if } k=2 ,\\
0 & \text{ else.}
\end{cases}
\end{align*}

For $a \geq 1$ and $|c| + n(a-2) \geq 0$, and $b \geq 1$ and $|c'| + n(b-2) \geq 0$, we have
\begin{align*}
\lambda_{a, a+1}({\pi}_!({\epsilon}^a \cdot \ell^*c) \cdot {\pi}_!({\epsilon}^{b} \cdot \ell^*c'))& = {\pi}_!({\epsilon}^{\{1,2,\ldots, a-1\}} \cdot {\epsilon}^{\{a+2, \ldots, a+b\}}\cdot \ell^*(c \cdot c'))\\
&\qquad + s^*e(T_\pi E) \cdot \pi_!(\epsilon^{a-1} \cdot \ell^*c) \cdot \pi_!(\epsilon^{b-1} \cdot \ell^*c') \\
&\qquad -
\begin{cases}
s^*(\ell^*c) \cdot {\pi}_!({\epsilon}^{b-1} \cdot \ell^*c') & \text{ if } a=1,\\
0 & \text{ else,}
\end{cases}\\
&\qquad -
\begin{cases}
{\pi}_!({\epsilon}^{a-1} \cdot \ell^*c) \cdot s^*(\ell^*c') & \text{ if } b=1,\\
0 & \text{ else.}
\end{cases}
\end{align*}
\end{proposition}

\begin{proof}
The class $\lambda_{1,2}(\epsilon \cdot \epsilon)$ is obtained from $\lambda_{1,2}(\epsilon \times \epsilon)$ by pulling back along $\Delta \colon {E} \to {E}\times_X {E}$. As $\Delta^*\Delta_!(1) = e(T_\pi E)$, the Euler class of the vertical tangent bundle of $\pi$, by Lemma \ref{lem:ContractionExterior} it is given by
\[\lambda_{1,2}({\epsilon}^2) = e(T_{{\pi}} {E}) + \pi^* s^* e(T_\pi E) - 2{v} \in H^{2n}({E} ; \bk).\]
Thus we have
\[\lambda_{1,2}({\pi}_!({\epsilon}^k \cdot \ell^*c)) = {\pi}_!((e(T_{{\pi}} {E})+ \pi^* s^*e(T_\pi E) - 2{v}) \cdot {\epsilon}^{k-2} \cdot \ell^*c).\]
Expanding this out, the first two terms give (using the projection formula), the first two claimed terms, and we also obtain a term $-{\pi}_!(2{v} \cdot {\epsilon}^{k-2} \cdot \ell^*c)$. As $v$ is the fibrewise Poincar{\'e} dual of $s \colon X \to E$, we have 
\[-\pi_!(2v \cdot \epsilon^{k-2} \cdot \ell^*c) = -s^*(2 \cdot \epsilon^{k-2} \cdot \ell^*x).\]
As $s^*(\epsilon)=0$, if $k > 2$ this vanishes, and if $k=2$ it is $-2s^*(\ell^*c)$. This leads to the formula stated above.

For the second part, there are projection maps $\pi_1, \pi_2 \colon {E} \times_X {E} \to {E}$ and hence classes
\[{\epsilon}_1 = \pi_1^*({\epsilon}) \in H^n({E} \times_X {E} ; \cH(g)_\bk) \quad\text{ and }\quad {\epsilon}_2 = \pi_2^*({\epsilon}) \in H^n({E} \times_X {E} ; \cH(g)_\bk),\]
and if we write $p \colon {E} \times_X {E} \to X$ then we have
\[{\pi}_!({\epsilon}^a \cdot \ell^*c) \cdot {\pi}_!({\epsilon}^b \cdot \ell^*c') = p_!({\epsilon}_1^a \cdot {\epsilon}_2^{b} \cdot \pi_1^*(\ell^*c) \cdot \pi_2^*(\ell^*c'))\]
so we must calculate
\[\lambda_{1,2}({\epsilon}_1 \cdot {\epsilon}_2) \in H^{2n}({E} \times_X {E}; \bQ).\]
But this is precisely what was called $\lambda_{1,2}({\epsilon} \times {\epsilon})$ in Lemma \ref{lem:ContractionExterior}, and was shown there to be $\Delta_!(1) - {v} \times 1 - 1 \times {v} + p^*s^*(v)$, so we get
\begin{align*}
&\lambda_{a,a+1}(\pi_!(\epsilon^a \cdot \ell^*c) \cdot \pi_!(\epsilon^{b} \cdot \ell^*c')) =\\ &\quad\quad p_!({\epsilon}_1^{a-1} \cdot (\Delta_!(1) - {v} \times 1 - 1 \times {v}  + p^*s^*(v)) \cdot {\epsilon}_2^{b-1} \cdot \pi_1^*(\ell^*c) \cdot \pi_2^*(\ell^*c')).
\end{align*}
When we expand this, the first term simplifies to ${\pi}_!({\epsilon}^{a-1} \cdot {\epsilon}^{b-1} \cdot \ell^*(c \cdot c')) = \pi_!(\epsilon^{\{1,2,\ldots, a-1\}} \cdot \epsilon^{\{a+2, \ldots, a+b\}}\cdot \ell^*(c \cdot c'))$, and the last term, using the identity $s^*(v) = s^* e(T_\pi E)$, simplifies to $s^* e(T_\pi E) \cdot \pi_!(\epsilon^{a-1} \cdot \ell^*c) \cdot \pi_!(\epsilon^{b-1} \cdot \ell^*c')$, so it remains to analyse the other two terms.

We can write the second term as
\[p_!(({\epsilon}^{a-1} \cdot \ell^*c \cdot {v}) \times ({\epsilon}^{b-1} \cdot \ell^*c')) = {\pi}_!({\epsilon}^{a-1} \cdot \ell^*c \cdot {v}) \cdot {\pi}_!({\epsilon}^{b-1} \cdot \ell^*c')\]
and we can evaluate the first factor, as $v$ is the fibrewise Poincar{\'e} dual of $s \colon X \to E$ so
\[{\pi}_!({\epsilon}^{a-1} \cdot \ell^*c \cdot {v}) = s^*({\epsilon}^{a-1} \cdot \ell^*c)\]
As $s^*(\epsilon)=0$ this vanishes for $a>1$, and is $s^*(\ell^*c)$ for $a=1$, in which case the second term is
\[p_!((\ell^*c \cdot {v}) \times ({\epsilon}^{b-1} \cdot \ell^*c')) = s^*(\ell^*c) \cdot {\pi}_!({\epsilon}^{b-1} \cdot \ell^*c').\]
 The third term can be analysed analogously.
\end{proof}

At this point we add a further assumption to our bundle \eqref{eq:BundleData}, namely that the composition $\ell \circ s \colon X \to B$ is nullhomotopic. This means that the terms in Proposition \ref{prop:ContractEpsilon} involving $s^*e(T_\pi E)$ vanish, and the terms involving $s^*(\ell^*c)$ vanish (if $|c| + n(1-2) \geq 0$ then $|c| \geq n$ and so $s^*(\ell^*c)=0$). Under this assumption, we define an extension
\[\mathcal{P}(-, \cV)_{\geq 0}^{2g} \otimes {\det}^{\otimes n} \colon \mathsf{(s)Br}_{2g} \lra \mathsf{Gr}(\bk\text{-mod})\]
of the functor $\mathcal{P}(-, \cV)_{\geq 0} \otimes {\det}^{\otimes n}$ defined on $\mathsf{FB}$, in the following way. We first extend $\det$ by
\begin{enumerate}[\indent (i)]

\item sending $(\mr{id}_{S \setminus \{x,y\}} , \varnothing, (x,y)) \colon S\setminus \{x,y\} \to S$ to the map  $x \wedge y \wedge -$,

\item sending $(\mr{id}_{S \setminus \{x,y\}}, (x,y), \varnothing) \colon S \to S \setminus \{x,y\}$ to the inverse of the map in (i),

\item extending to general morphisms in $\mathsf{(s)Br}_{2g}$ by writing them as the composition of bijections and morphisms of the above two types.
\end{enumerate}
We next extend $\mathcal{P}(-, \cV)_{\geq 0}$ to $\mathcal{P}(-, \cV)_{\geq 0}^{2g}$ by
\begin{enumerate}[\indent (i)]
\item sending $(\mr{id}_{S \setminus \{x,y\}} , \varnothing, (x,y)) \colon S\setminus \{x,y\} \to S$ to the map which adds the labelled part $(\{x,y\}, 1)$,

\item sending $(\mr{id}_{S \setminus \{x,y\}}, (x,y), \varnothing) \colon S \to S \setminus \{x,y\}$ to the map which sends $(\{P_i\}, \{c_i\})$ to
\begin{enumerate}[(a)]
\item if some $P_i$ is $\{x,y\}$ and $|c_i|=0$, so $c_i=\lambda \cdot 1$, then we remove this part and multiply by the scalar $\lambda\cdot(-1)^n \cdot 2g$,

\item if some $P_i$ contains $\{x,y\}$ (and $|c_i|>0$ if $P_i=\{x,y\}$) then we change the part to $P_i \setminus \{x,y\}$ and change the label to $e \cdot c_i$, 

\item if $x$ and $y$ lie in different parts $P_i$ and $P_j$, then we merge these into a new part $(P_i \setminus \{x\}) \cup (P_j \setminus \{y\})$ labelled by $c_i \cdot c_j$.
\end{enumerate}
\item extending to general morphisms in $\mathsf{(s)Br}_{2g}$ by writing them as the composition of bijections and morphisms of the above two types.
\end{enumerate}

\begin{proposition}
The $\Phi_S$ determine a natural transformation of functors
\[\Phi \colon \mathcal{P}(-, \cV)_{\geq 0}^{2g} \otimes {\det}^{\otimes n} \Longrightarrow H^{*}(X ; \cH(g)_\bk^{\otimes -}) \colon \mathsf{(s)Br}_{2g} \lra \mathsf{Gr}(\bk\text{-mod}).\]
\end{proposition}
\begin{proof}
This follows almost tautologically from Proposition \ref{prop:ContractEpsilon}, because we have defined the functor $\mathcal{P}(-, \cV)_{\geq 0}^{2g}$ to transform in the way the twisted Miller--Morita--Mumford classes do. The only subtle point is the scalar in (ii) (a) above, but that this is correct comes from the following calculation, when $c_i = \lambda \cdot 1$:
\begin{equation*}
\lambda_{x,y}(\pi_!(\epsilon^{\{x,y\}} \cdot \ell^*c_i)) = \pi_!(e(T_\pi E)\cdot\ell^*c_i)-2s^*\ell^*c_i = \lambda(\chi(W_g)-2) = \lambda \cdot (-1)^n \cdot 2g.\qedhere
\end{equation*}
\end{proof}

Finally, we recognise that $\mathcal{P}(-, \cV)_{\geq 0}^{2g} \otimes {\det}^{\otimes n}$ is the left Kan extension to $\mathsf{(s)Br}_{2g}$ of the completely analogous functor
\[\mathcal{P}(-, \cV)_{\geq 0}' \otimes {\det}^{\otimes n}  \colon \mathsf{d(s)Br} \lra \mathsf{Gr}(\bk\text{-mod}),\]
where $\mathcal{P}(S, \cV)_{\geq 0}'$ is the submodule of $\mathcal{P}(S, \cV)_{\geq 0}$ generated by those labelled partitions of $S$ having no part of size 2 labelled by the multiplicative unit $1 \in \cV$. Note that by this condition the scalar $2g$ no longer arises when applying structure map, so is neglected from the notation.

\begin{remark}\label{rem:Bases3}
As in Remarks \ref{rem:Bases1} and \ref{rem:Bases2}, if we choose a homogeneous basis $\cB$ of $\cV$ containing the multiplicative unit $1 \in \cV$ as an element, then $\mathcal{P}(S, \cV)_{\geq 0}'$ has a homogeneous basis given by partitions of $S$ labelled by elements of $\cB$, having no parts (i) of size 0 with label of degree $ \leq 2n$, (ii) of size 1 with label of degree $<n$, or (iii) of size 2 labelled by $1 \in \cB$. These remarks show that the $\Sigma_S$-action on $\mathcal{P}(S, \cV)$, $\mathcal{P}(S, \cV)_{\geq 0}$, and $\mathcal{P}(S, \cV)_{\geq 0}'$ makes them all into permutation modules.
\end{remark}

\subsection{Multiplication}\label{sec:functoriality-mult}
The functor
\[H^{*}(X ; \cH(g)_\bk^{\otimes -}) \colon \mathsf{(s)Br}_{2g} \lra \mathsf{Gr}(\bk\text{-mod})\]
has the structure of a commutative ring object in this category of functors, under the Day convolution product. This is equivalent to saying that it may be equipped with a lax symmetric monoidality. To do so, for $S, T \in \mathsf{(s)Br}_{2g}$ we let
\[H^{*}(X ; \cH(g)_\bk^{\otimes S}) \otimes H^{*}(X ; \cH(g)_\bk^{\otimes T}) \lra H^{*}(X ; \cH(g)_\bk^{\otimes S \sqcup T})\]
be given by the cup product. It is an elementary verification that this defines a symmetric lax monoidality, recalling that the symmetry for the monoidal structure on $\mathsf{Gr}(\bk\text{-mod})$ includes the Koszul sign rule.

The functor $\mathcal{P}(-, \cV)_{\geq 0}^{2g} \otimes {\det}^{\otimes n}$ may also be equipped with the structure of a commutative ring object, making $\Phi$ a morphism of commutative rings. It is easiest to describe commutative ring structures on $\mathcal{P}(-, \cV)_{\geq 0}^{2g}$ and $\det$ separately, and then take their product. For $S, T \in \mathsf{(s)Br}_{2g}$ we let
\[\mathcal{P}(S, \cV)_{\geq 0}^{2g} \otimes \mathcal{P}(T, \cV)_{\geq 0}^{2g} \lra \mathcal{P}(S \sqcup T, \cV)_{\geq 0}^{2g}\]
be given by disjoint union of partitions, and we let
\[(\det \bk^S) \otimes (\det \bk^T) \lra \det \bk^{S \sqcup T}\]
be given by exterior product of volume forms. It is again elementary to verify that these define symmetric lax monoidalities. By our description of $\Phi$ it commutes with these symmetric lax monoidalities, and hence is a morphism of commutative ring objects.

Finally, the analogous discussion provides
\[\mathcal{P}(-, \cV)_{\geq 0}' \otimes {\det}^{\otimes n}  \colon \mathsf{d(s)Br} \lra \mathsf{Gr}(\bk\text{-mod})\]
with a commutative ring structure.

\subsection{Stabilisation}\label{sec:Stabilisation}
If we have a smooth $W_g$-bundle $\pi \colon E \to X$ with section $s \colon X \to E$, and this section has an extension to a fibrewise embedding $d \colon X \times D^{2n} \to E$, then we can form the fibrewise connected sum of $E$ and $X \times W_1$ to obtain a smooth $W_{g+1}$-bundle $\pi' \colon E' \to X$, which is again equipped with a fibrewise embedding $d' \colon X \times D^{2n} \subset X \times W_{1,1} \to E'$. In this situation we may ask if the twisted Miller--Morita--Mumford classes of $\pi$ and $\pi'$ can be compared, and we will now show how. 

There is an identification of coefficient systems $\cH'(g+1)_\bk = \cH(g)_\bk \oplus \bk^2$ and so a projection map $r \colon \cH'(g+1)_\bk \to \cH(g)_\bk$ and an inclusion map $i \colon \cH(g)_\bk \to \cH'(g+1)_\bk$. Recall that $s^*(\epsilon)=0$, so $\epsilon$ lifts to a class $\rho \in H^n(E, X \times D^{2n};\cH(g)_\bk)$, which is in fact unique. Now under the maps
\begin{equation*}
\begin{tikzcd}
\rho \in H^n(E, X \times D^{2n};\cH(g)_\bk)  & H^n(E', X \times W_{1,1} ; \cH(g)_\bk) \dar{(\mr{id}, d')^*} \lar{\sim}[swap]{\mr{exc}}\\
\rho' \in H^n(E', X \times D^{2n}; \cH'(g+1)_\bk) \rar{r_*}& H^n(E', X \times D^{2n} ; \cH(g)_\bk)
\end{tikzcd}
\end{equation*}
the classes $\rho$ and $\rho'$ correspond, because just as in the proof of Lemma \ref{lem:ContractionExterior} these classes are determined by their restriction to a fibre and $(\mr{id}_{\cH'} \otimes r)(\omega') = (i \otimes \mr{id}_{\cH})(\omega) \in H^0(B;\cH'(g+1)_\bk \otimes \cH(g)_\bk)$.

Now if there are lifts $\ell' \colon E' \to B$ and $\ell \colon E \to B$ of the maps classifying the respective vertical tangent bundles of these two fibre bundles, which agree when restricted to the common subspace
\[E' \supset E \setminus X \times D^{2n} \subset E,\]
then for $k>0$ and $c \in H^{2*}(B;\bk)$ we may calculate
\begin{align*}
r_*\left(\pi'_!((\epsilon')^k \cdot (\ell')^*c)\right) &= \pi'_!(r_*(\rho')^k \cdot (\ell')^*c)\\
&= \pi'_!\left(((\mr{id}, d')^* \circ \mr{exc}^{-1}(\rho))^k \cdot (\ell')^*c\right)\\
&= \pi'_!\left((\mr{exc}^{-1}(\rho))^k \cdot  (\ell')^*c\right)\\
&= \pi_!\left(\rho^k \cdot \ell^*c\right)\\
&= \pi_!\left(\epsilon^k \cdot \ell^*c\right).
\end{align*}

For $k=0$ these are the standard Miller--Morita--Mumford classes, and their behaviour under fibrewise stabilisation is well understood.

\subsection{The isomorphism theorem}\label{sec:isomorphism-theorem}

We will now apply the constructions of the previous sections to certain universal bundles. See \cite[Definition 1.5]{grwcob} for more details on the following construction. To define these bundles, note that the fibration $\theta \colon B \to B\mr{SO}(2n)$ classifies an oriented vector bundle $\theta^*\gamma_{2n} \to B$, and a \emph{$\theta$-structure} on a $2n$-dimensional vector bundle is a bundle map to $\theta^*\gamma_{2n}$ (i.e.\ a continuous map which is a linear isomorphism on each fibre). Fix a $\theta$-structure $\hat{\ell}_{D^{2n}} \colon TD^{2n} \to \theta^*\gamma_{2n}$, and let
\[\mr{Bun}^\theta(TW_{g}, D^{2n}; \hat{\ell}_{D^{2n}})\]
denote the space of all $\theta$-structures $\hat{\ell} \colon TW_g \to \theta^*\gamma$ which are equal to $\hat{\ell}_{D^{2n}}$ when restricted to $D^{2n} \subset W_g$. This space has an action of the group $\diff(W_g, D^{2n})$ of diffeomorphisms which are the identity on $D^{2n} \subset W_g$, and we define
\[B\mr{Diff}^\theta(W_{g}, D^{2n}; \hat{\ell}_{D^{2n}}) \coloneqq \mr{Bun}^\theta(TW_{g}, D^{2n};  \hat{\ell}_{D^{2n}}) \sslash \mr{Diff}(W_{g}, D^{2n}).\]
This space carries a smooth $W_g$-bundle given by
\[E^\theta \coloneqq (\mr{Bun}^\theta(TW_{g}, D^{2n}; \hat{\ell}_{D^{2n}}) \times W_g) \sslash \mr{Diff}(W_{g}, D^{2n})\]
with $\pi \colon E^\theta \to B\mr{Diff}^\theta(W_{g}, D^{2n}; \hat{\ell}_{D^{2n}})$ given by projection to the first factor. This has a section $s$ given by the $\diff(W_g, D^{2n})$-equivariant map $\{*\} \subset D^{2n} \subset W_g$. The bundle $\pi$ has a vertical tangent bundle, which may be described as
\[T_\pi E^\theta \coloneqq (\mr{Bun}^\theta(TW_{g}, D^{2n}; \hat{\ell}_{D^{2n}}) \times TW_g) \sslash \mr{Diff}(W_{g}, D^{2n})\]
using the action of $\diff(W_g, D^{2n})$ on $TW_g$ via the derivative. Evaluation defines a bundle map $\hat{\ell} \colon T_\pi E^\theta \to \theta^*\gamma_{2n}$, which has an underlying map $\ell \colon E^\theta \to B$. The composition $\ell \circ s \colon B\mr{Diff}^\theta(W_{g}, D^{2n}; \hat{\ell}_{D^{2n}}) \to B$ is constant, as it underlies the $\theta$-structure on the bundle $s^*T_\pi E^\theta$, but this is trivial by our definition of the space of bundle maps. 

This discussion shows that we are in the position to apply the constructions of the previous sections, giving maps
\begin{equation}\label{eq:PhiS}
\Phi_S \colon \mathcal{P}(S, \cV)_{\geq 0}^{2g} \otimes (\det \bk^S)^{\otimes n} \lra H^{*}(B\mr{Diff}^\theta(W_{g}, D^{2n}; \hat{\ell}_{D^{2n}}) ; \cH(g)_\bk^{\otimes S})
\end{equation}
for each finite set $S$. The goal of this section is to show that these maps are isomorphisms in a range of degrees when $\bk=\bQ$, we restrict to a certain path-component of $B\mr{Diff}^\theta(W_{g}, D^{2n}; \hat{\ell}_{D^{2n}})$, and the following technical assumptions on $\theta \colon B \to B\mr{SO}(2n)$ are made:

\begin{assumption}\label{ass:1}
$B$ is $n$-connected, $H^*(B;\bQ)$ is concentrated in even degrees and is finite-dimensional in each degree, and any $\theta$-structure on $D^{2n}$ extends to one on $S^{2n}$.
\end{assumption}

\begin{remark}
One can reduce to the case that $B$ is $n$-connected without loss of generality, as in \cite[Lemma 7.16]{grwcob}; allowing $B$ to have cohomology in odd degrees is surely possible, but will require a more careful discussion of signs when defining twisted Miller--Morita--Mumford classes; the last condition is called being spherical in \cite{grwcob,grwstab1,grwstab2}, and is standard.
\end{remark}

We let $\hat{\ell}_g \in \mr{Bun}^\theta(W_{g}, D^{2n}; \hat{\ell}_{D^{2n}})$ be a $\theta$-structure which is standard (in the sense of \cite[Definition 7.2]{grwstab1}) when restricted to $W_{g,1} = W_g \setminus \mathrm{int}(D^{2n}) \subset W_g$. Under the assumption that $\theta$ is spherical such a $\hat{\ell}_g$ exists, by the evident generalisation of \cite[Lemma 7.9]{grwstab1} to arbitrary genus. We let
\[B\mr{Diff}^\theta(W_{g}, D^{2n}; \hat{\ell}_{D^{2n}})_{\hat{\ell}_g} \subset B\mr{Diff}^\theta(W_{g}, D^{2n}; \hat{\ell}_{D^{2n}})\]
denote the path component of $\hat{\ell}_g$.

\begin{theorem}\label{thm:Iso}
Let $2n >0$ and $2n \neq 4$, and suppose that $\theta \colon B \to B\mr{SO}(2n)$ is a tangential structure satisfying Assumption \ref{ass:1}. For any finite set $S$ the map
\[\mathcal{P}(S, \cV)_{\geq 0}  \otimes (\det \bQ^{S})^{\otimes n} \lra H^*(B\mr{Diff}^\theta(W_{g}, D^{2n}; \hat{\ell}_{D^{2n}})_{\hat{\ell}_g} ; \cH(g)_\bQ^{\otimes S}),\]
induced by $\Phi_S$, is an isomorphism in a range of cohomological degrees tending to infinity with $g$.\footnote{If $2n=4$ then the argument we will give shows that map is an isomorphism after taking the limit as $g \to \infty$. That one can sensibly form an induced map between limits depends on the discussion in Section \ref{sec:Stabilisation}.}
\end{theorem}

In the rest of this paper we will be interested in the case where $\theta \colon B\mr{SO}(2n)\langle n \rangle \to B\mr{SO}(2n)$ is the $n$-connected cover, in which case it follows from obstruction theory that there is an equivalence $B\mr{Diff}^\theta(W_{g}, D^{2n}; \hat{\ell}_{D^{2n}})_{\hat{\ell}_g} \simeq B\mr{Diff}(W_{g}, D^{2n})$, as we shall explain in Section \ref{sec:cohomologytorelli}. In this case Theorem \ref{thm:Iso} may be considered as the analogue of the Madsen--Weiss theorem (for $2n=2$) or Theorem 1.1 of \cite{grwcob} for twisted coefficients. In the case $2n=2$ this result is due to Kawazumi \cite{Kawazumi}, though phrased a little differently. As the left-hand term is independent of $g$, it in particular recovers homological stability for the right-hand term: this was already known to hold by \cite{boldsen, Ivanov} for $2n=2$, and by \cite{Krannich} for $2n \geq 6$.

\begin{proof}[Proof of Theorem \ref{thm:Iso}]
We apply the method introduced in \cite{R-Wtwist}. Suppose for concreteness that $n$ is odd. Let $W$ be a finite-dimensional rational vector space and $Y = K(W^\vee, n+1)$ be a functorial model for the associated Eilenberg--MacLane space. Then $\theta \times Y \coloneqq \theta \circ \mathrm{pr}_B \colon B \times Y \to B\mr{SO}(2n)$ is a new tangential structure, and we may consider the moduli space $B\mr{Diff}^{\theta \times Y}(W_{g}, D^{2n}; \hat{\ell}_{D^{2n}}^Y)$ of manifolds equipped with a $\theta$-structure satisfying the boundary condition $\hat{\ell}_{D_{2n}}$ and a map to $Y$ which sends $D^{2n}$ to the basepoint $y_0 \in Y$. Forgetting the map to $Y$ gives a fibration sequence
\begin{equation}\label{eq:FibSeq}
\mr{map}_*((W_{g}, D^{2n}), (Y, y_0)) \overset{i} \lra B\mr{Diff}^{\theta \times Y}(W_{g}, D^{2n}; \hat{\ell}_{D^{2n}}^Y) \lra B\mr{Diff}^{\theta}(W_{g}, D^{2n}; \hat{\ell}_{D^{2n}})
\end{equation}
and the fibre is path-connected, so $B\mr{Diff}^{\theta \times Y}(W_{g}, D^{2n}; \hat{\ell}_{D^{2n}}^Y)$ has the same path-components as $B\mr{Diff}^{\theta}(W_{g}, D^{2n}; \hat{\ell}_{D^{2n}})$. In fact there is a canonical isomorphism
\[\pi_1(\mr{map}_*((W_{g}, D^{2n}), (Y, y_0)), *) \cong \tilde{H}^{n+1}(S^1 \wedge W_g/D^{2n}; W^\vee) \cong H^n(W_{g}, D^{2n};\bQ) \otimes W^\vee\]
which induced a canonical map
\[\Lambda^*(H(g) \otimes W[1]) \lra H^*(\mr{map}_*((W_{g}, D^{2n}), (Y, y_0));\bQ).\]
This map is easily checked to be an isomorphism, as this mapping space is a $K(\pi, 1)$. The identification is one of $\pi_1(B\mr{Diff}^{\theta}(W_{g}, D^{2n}; \hat{\ell}_{D^{2n}}), \hat{\ell}_g)$-modules, and we have used Poincar{\'e} duality to identify $H(g)^\vee$ with $H(g)$.

When $2n \neq 4$, by the main theorems of \cite{boldsen,oscarresolutions,grwcob,grwstab1} there is a map
\[\alpha \colon B\mr{Diff}^{\theta \times Y}(W_{g}, D^{2n}; \hat{\ell}^Y_{D^{2n}})_{\hat{\ell}_g} \lra \Omega^{\infty}_0(\mathrm{MT}\theta \wedge Y_+)\]
which is an isomorphism on cohomology in a range of degrees tending to infinity with $g$. (For $2n=4$ it is an isomorphism on homology upon taking the colimit as $g \to \infty$, by \cite[Theorem 1.5]{grwstab2}.) By our finite type assumption on $H^*(B;\bQ)$, and the fact that $W$ is finite-dimensional, we have that
\[H^*(\mathrm{MT}\theta \wedge Y_+;\bQ) = H^*(\mathrm{MT}\theta;\bQ) \otimes \mr{Sym}^{*}(W[n+1])\]
has finite type, and so the natural map
\[\mr{Sym}^*\left([H^*(\mathrm{MT}\theta;\bQ) \otimes \mr{Sym}^{*}(W[n+1])]_{>0}\right) \lra H^*(\Omega^\infty_0(\mathrm{MT}\theta \wedge Y_+);\bQ )\]
is an isomorphism.

Thus the Serre spectral sequence for \eqref{eq:FibSeq} is a spectral sequence of $GL(W)$-modules, of the form
\begin{equation}\label{eq:Collapse} \begin{tikzcd}
E_2^{p,q} = H^p\left(B\mr{Diff}^\theta(W_{g}, D^{2n}; \hat{\ell}_{D^{2n}})_{\hat{\ell}_g}; \Lambda^q(\cH(g)_\bQ \otimes W [1])\right) \dar[Rightarrow] \\ 
\mr{Sym}^*\left([H^*(\mathrm{MT}\theta;\bQ) \otimes \mr{Sym}^{*}(W[n+1])]_{>0}\right),
\end{tikzcd}
\end{equation}
where the target is as indicated only in a range of degrees. Different rows of this spectral sequence are $GL(W)$-representations of different weights, so it  collapses. Giving $W$ on the target $q$-grading 1, this is an isomorphism of bigraded rings in a range of degrees.

We now wish to apply Schur--Weyl duality for the general linear group. For further details on the following we refer to Sam--Snowden \cite{SS}, particularly Section 2.2. In the setting described in our Section \ref{sec:RepOfCat} this may be done as follows. We let $\Lambda \coloneqq \cat{FB}$ be the ($\bQ$-linearisation of the) category of finite sets and bijections. We let $\mr{GL} \coloneqq \colim_{n \to \infty} \mr{GL}_n(\bQ)$, let $\cat{Rep}(\mr{GL})$ be the $\bQ$-linear abelian symmetric monoidal category of all representations of the group $\mr{GL}$, and let $\cA \coloneqq \cat{Rep^{pol}}(\mr{GL})$ be the $\bQ$-linear abelian symmetric monoidal category of polynomial representations of the group $\mr{GL}$, i.e.\ those representations arising as finite direct sums of summands of tensor powers of the standard representation $W \coloneqq \colim_{n \to \infty} \bQ^n$. Similarly, we let $\hat{W}_* \coloneqq \lim_{n \to \infty} \mr{Hom}(\bQ^n, \bQ)$, a pro-algebraic representation of $\mr{GL}$, and let $\widehat{\cat{Rep^{pol}}}(GL)$ denote the category of polynomial pro-algebraic representations of $\mr{GL}$, equipped with the completed tensor product. Continuous dual gives an identification $\widehat{\cat{Rep^{pol}}}(\mr{GL})^\mr{op} \cong {\cat{Rep^{pol}}}(\mr{GL})$. 

We let $K \colon \cat{FB} \to \widehat{\cat{Rep^{pol}}}(\mr{GL})$ be defined as $K(S) = \hat{W}_*^{\otimes S}$, with its evident symmetric monoidality, which has the structure of a tensor kernel. Taking the continuous dual gives the functor $K^\vee \colon \cat{FB}^\mr{op} \to {\cat{Rep^{pol}}}(GL)$ with $K^\vee(S) = {W}^{\otimes S}$. It follows from \cite[Section 2.2.9]{SS} that the functor
\[K^\vee \otimes^{\cat{FB}} - \colon (\bQ\text{-mod}^\cat{FB})^f \lra \cat{Rep^{pol}}(\mr{GL})\]
is a symmetric monoidal equivalence of categories, with inverse given by
\begin{equation}\label{eq:trans}
\begin{aligned}
\cat{Rep^{pol}}(\mr{GL}) &\lra (\bQ\text{-mod}^\cat{FB})^f\\
U &\longmapsto (S \mapsto [\hat{W}_*^{\otimes S} \otimes U]^{\mr{\mr{GL}}}).
\end{aligned}
\end{equation}

We apply this discussion as follows. Taking the limit as $\mathrm{dim}(W) \to \infty$, the collapsing spectral sequence \eqref{eq:Collapse} gives an identification of (bi)graded objects in $\cat{Rep^{pol}}(GL)$, and hence of (bi)graded objects in $\bQ\text{-mod}^\cat{FB}$, which we now identify. Recall that we have
\[H^*(\mathrm{MT}\theta;\bQ) \cong H^{*}(B ;\bQ)[-2n] = \cV[-2n].\]
 We may write the abutment of \eqref{eq:Collapse} as
\[\mr{Sym}^*([H^*(\mathrm{MT}\theta;\bQ)]_{>0}) \otimes \mr{Sym}^*\left([H^*(\mathrm{MT}\theta;\bQ) \otimes \mr{Sym}^{*>0}(W[n+1])]_{>0}\right).\]
The transformation \eqref{eq:trans} of the second term is $\mathcal{P}_{>0}(-, \cV)_{\geq 0} \colon \mathsf{FB} \to \mathsf{Gr}(\bQ\text{-mod})$, by \cite[Proposition 5.1]{R-Wtwist}, and the first term is $\mathcal{P}(\varnothing, \cV)_{\geq 0}$. Thus the transformation \eqref{eq:trans} of the right-hand side is the functor
\[\mathcal{P}(-, \cV)_{\geq 0}  \colon \mathsf{FB} \lra \mathsf{Gr}(\bQ\text{-mod})\]
in a range of degrees. We recognise the $E_2$-page of \eqref{eq:Collapse} as 
\[E_2^{p,q} = \left[H^p(B\mr{Diff}^\theta(W_{g}, D^{2n}; \hat{\ell}_{D^{2n}})_{ \hat{\ell}_g}; \cH(g)_\bQ^{\otimes q}) \otimes \det(\bQ^q) \otimes W^{\otimes q} \right]^{\Sigma_q},\]
so its transformation under \eqref{eq:trans} is $H^*(B\mr{Diff}^\theta(W_{g}, D^{2n};\ell_{D^{2n}})_{\hat{\ell}_g}; \cH(g)_\bQ^{\otimes -}) \otimes {\det} \colon \mathsf{FB} \to \mathsf{Gr}(\bQ\text{-mod})$. Carrying the factor $\det$ to the other side, this shows that there is a natural isomorphism
\[\mathcal{P}(-, \cV)_{\geq 0} \otimes {\det} \cong H^*(B\mr{Diff}^\theta(W_{g}, D^{2n}; \hat{\ell}_{D^{2n}})_{\hat{\ell}_g} ; \cH(g)_\bQ^{\otimes -})\]
in a range of degrees.

Unfortunately it is not yet clear that it is the map we have constructed which yields this isomorphism. To see that it is, we go into the construction in more detail. Write $E^\theta \to B\mr{Diff}^\theta(W_{g}, D^{2n}; \hat{\ell}_{D^{2n}})_{\hat{\ell}_g}$, and $E^{\theta \times Y} \to B\mr{Diff}^{\theta \times Y}(W_{g}, D^{2n}; \hat{\ell}_{D^{2n}}^Y)_{\hat{\ell}_g}$ for the $W_g$-bundles over these spaces. There is a fibration sequence
\[\mr{map}_*((W_{g}, D^{2n}), (Y, y_0)) \overset{i} \lra E^{\theta \times Y} \lra E^\theta\]
and hence a spectral sequence of $GL(W)$-representations
\[E^{p,q}_2 = H^p(E^\theta; \Lambda^q(\cH(g)_\bQ \otimes W [1])) \Longrightarrow H^{p+q}(E^{\theta \times Y} ; \bQ),\]
which again collapses and splits as different rows are $GL(W)$-representations of different weights. Taking weight 1 pieces gives a canonical identification
\begin{equation}\label{eq:Wt1}
H^p(E^\theta; \cH(g)_\bQ) \otimes W \cong H^p(E^\theta; \cH(g)_\bQ \otimes W) \cong H^{p+1}(E^{\theta \times Y} ; \bQ)^{(1)}.
\end{equation}

Evaluation defines a map 
\[ev \colon E^{\theta \times Y} \lra Y\]
and so determines a $GL(W)$-equivariant map 
\[\rho \colon W = {H}^{n+1}(Y;\bQ) \lra H^{n+1}(E^{\theta \times Y};\bQ)^{(1)}\]
landing in the weight 1 piece. In terms of the identification \eqref{eq:Wt1} above the map $\rho$ must be given by $\rho(w)= \chi \otimes w \in H^n(E^\theta ; \cH(g)_\bQ) \otimes W$ for some class $\chi \in H^n(E^\theta; \cH(g)_\bQ)$.

\begin{claim} 
The class $\chi$ is the class $\epsilon$ defined in Section \ref{sec:semi-euler-class}.
\end{claim}

\begin{proof}[Proof of claim]
By naturality we may restrict to the trivial tangential structure $\theta = \mathrm{id} \colon B\mr{SO}(2n) \to B\mr{SO}(2n)$ to prove this. The decomposition \eqref{eq:KMDecomposition} in this case is
\[H^n(E ; \cH(g)_\bQ) = H^n(B\mr{Diff}(W_{g}, D^{2n}); \cH(g)_\bQ) \oplus H^0(B\mr{Diff}(W_{g}, D^{2n}) ; \cH(g)_\bQ^{\otimes 2}).\]
The component of $\chi$ in the first factor is given by pulling back along the section $s$, but the composition $B\mr{Diff}(W_{g}, D^{2n}) \overset{s}\to E^Y \overset{ev}\to Y$ is constant, and the section lands in the disc $D^{2n} \subset W_g$ on which the maps to $Y$ are constantly $y_0$. It remains to determine the component of $\chi$ in the second factor.

The restriction map to a single fibre
\[H^0(B\mr{Diff}(W_{g}, D^{2n}) ; \cH(g)_\bQ \otimes \cH(g)_\bQ) \lra \cH^n(W_g;\bQ) \otimes \cH^n(W_g;\bQ)\]
is injective and has image $\bQ\{\omega\}$, and under the identification above $\epsilon$ maps to the class $\omega$ by definition.

Restricting the previous discussion to a single fibre, we are considering the spectral sequence of the (trivial) fibration
\[\mr{map}_*(W_{g}/D^{2n}, Y) \lra W_{g}\times \mr{map}_*(W_g/D^{2n}, Y) \lra W_{g}.\]
The evaluation map gives
\[\rho \colon W \lra H^{n+1}(W_{g} \times \mr{map}_*(W_{g}/D^{2n}, Y);\bQ) \cong H^n(W_{g};\bQ) \otimes H^1(\mr{map}_*(W_{g}/D^{2n}, Y);\bQ)\]
which is the map
\[W \lra H(g) \otimes (H(g) \otimes W)\]
given by $\omega \otimes -$. This shows that $\chi$ restricts to $\omega$, so $\chi=\epsilon$.
\end{proof}

Consider the commutative diagram
\[\begin{tikzcd} 
H^{2d}(E^{\theta \times Y};\bQ) \otimes W^{\otimes k} \dar{\rho^{\otimes k}} &[-12pt] H^{2d}(B;\bQ) \otimes W^{\otimes k} \lar[swap]{\ell^* \otimes W^{\otimes k}}\\
H^{k(n+1)+2d}(E^{\theta \times Y};\bQ)^{(k)} \rar{\pi^Y_{!}} & H^{k(n+1)+2d-2n}(B\mr{Diff}^{\theta \times Y}(W_{g}, D^{2n}; \hat{\ell}^Y_{D^{2n}})_{\hat{\ell}_g} ;\bQ)^{(k)}\\
H^{kn+2d}(E^\theta ;\Lambda^k(\cH(g)_\bQ \otimes W)) \rar{\pi_{!}} \uar[equals] & H^{kn+2d-2n}(B\mr{Diff}^\theta(W_{g}, D^{2n};\hat{\ell}_{D^{2n}})_{\hat{\ell}_g}; \Lambda^k(\cH(g)_\bQ \otimes W)) \uar[equals]
\end{tikzcd}\]

If $x \in H^{2d}(B;\bQ)$, going along the top sends $x \otimes w_1 \otimes \cdots \otimes w_k$ to
\[\pi_{!}^Y(\rho(w_1) \cdots \rho(w_k) \cdot \ell^*x) \in H^{k(n+1)+2d-2n}(B\mr{Diff}^{\theta \times Y}(W_{g}, D^{2n}; \hat{\ell}_{D^{2n}}^Y)_{\hat{\ell}_g};\bQ)^{(k)}.\]
On the other hand, as $\rho(w_i) = \epsilon \otimes w_i$ this corresponds to
\[\pi_{!}((\epsilon \otimes w_1) \wedge \cdots \wedge (\epsilon \otimes w_k) \cdot \ell^*x) \in H^{kn+2d-2n}(B\mr{Diff}^\theta(W_{g}, D^{2n}; \hat{\ell}_{D^{2n}})_{\hat{\ell}_g}; \Lambda^k(\cH(g)_\bQ \otimes W)),\]
the result of antisymetrising the class
\[\pi_{!}(\epsilon^k \cdot \ell^*x) \otimes w_1 \otimes \cdots \otimes w_k \in H^*(B\mr{Diff}^\theta(W_{g}, D^{2n}; \hat{\ell}_{D^{2n}})_{\hat{\ell}_g}; \cH(g)_\bQ^{\otimes k}) \otimes W^{\otimes k}.\]
That is, the result may be expressed in terms of our twisted Miller--Morita--Mumford classes $\pi_!(\epsilon^k \cdot \ell^*x)$, showing that the map we have constructed is surjective in the stable range: as the source and target are graded vector spaces having the same finite dimension in each degree, it follows that the map in the statement of the theorem is an isomorphism.

\vspace{1ex}

Finally, if $n$ is even then we take instead $Y = K(W^\vee, n+2)$. Then there is an equivalence $\mr{map}((W_{g}, D^{2n}), (Y, y_0)) \simeq K(H(g) \otimes W^\vee, 2)$ and so the relevant spectral sequence of $GL(W)$-modules has the form
\begin{equation*}
\begin{tikzcd}
E_2^{p,q} = H^p(B\mr{Diff}^\theta(W_{g}, D^{2n}; \hat{\ell}_{D^{2n}})_{\hat{\ell}_g}; \mr{Sym}^q(\cH(g)_\bQ \otimes W [2])) \dar[Rightarrow] \\ 
\mr{Sym}^*([H^*(\mathrm{MT}\theta;\bQ) \otimes \mr{Sym}^{*}(W[n+2])]_{>0}),
\end{tikzcd}
\end{equation*}
The discussion then goes through as above, except that we now recognise the $E_2$-page as 
\[E_2^{p,q} = \left[H^p(B\mr{Diff}^\theta_\partial(W_{g}, D^{2n}; \hat{\ell}_{D^{2n}})_{\hat{\ell}_g} ; \cH(g)_\bQ^{\otimes q}) \otimes W^{\otimes q} \right]^{\Sigma_q},\]
so that its transform is now given by $H^*(B\mr{Diff}^\theta_\partial(W_{g}, D^{2n}; \hat{\ell}_{D^{2n}})_{\hat{\ell}_g} ; \cH(g)_\bQ^{\otimes -})$, without the sign representation. We then proceed as above.
\end{proof}

\section{The cohomology of the Torelli space}\label{sec:cohomologytorelli}

In this section we work with the tangential structure $\theta \colon B\mr{SO}(2n)\langle n \rangle \to B\mr{SO}(2n)$, in which case the forgetful map
\[B\mr{Diff}^\theta(W_{g}, D^{2n}; \hat{\ell}_{D^{2n}}) \lra B\mr{Diff}(W_{g}, D^{2n})\]
is a weak equivalence, because the space $\mr{Bun}^\theta(TW_{g}, D^{2n}; \hat{\ell}_{D^{2n}})$ is equivalent to the space of relative lifts
\begin{equation*}
\begin{tikzcd} 
D^{2n} \dar \rar{\ell_{D^{2n}}}& B\mr{SO}(2n)\langle n \rangle \dar{\sigma \circ \theta}\\
W_g \rar[swap]{TW_g} \arrow[ru, dashed] & B\mr{O}(2n),
\end{tikzcd} 
\end{equation*}
where $\sigma \colon B\mr{SO}(2n) \to B\mr{O}(2n)$ is the double cover, and this space of lifts is easily seen to be contractible by obstruction theory. We will therefore write $B\mr{Diff}(W_{g}, D^{2n})$ instead of $B\mr{Diff}^\theta(W_{g}, D^{2n}; \hat{\ell}_{D^{2n}})$, for simplicity.

As stated in the introduction, the action of diffeomorphisms on the middle-dimensional homology gives a homomorphism
\[\alpha_g \colon \mr{Diff}(W_{g}, D^{2n}) \lra G_g \coloneqq \begin{cases} \mr{Sp}_{2g}(\bZ) & \text{if $n$ is odd,} \\
\mr{O}_{g,g}(\bZ) & \text{if $n$ is even.}\end{cases}\]
We denote its image $G'_g$. It is often surjective, but further restrictions can arise from a quadratic refinement of the intersection form. A result of Kreck \cite{kreckisotopy} tells us that
\[G'_g = \begin{cases}G_g & \text{if $n=1,3,7$ or $n$ is even,} \\
\mr{Sp}^q_{2g}(\bZ) & \text{otherwise,}\end{cases}\]
where $\mr{Sp}^q_{2g}(\bZ) \leq \mr{Sp}_{2g}(\bZ)$ is the proper subgroup of those symplectic matrices which preserve the quadratic refinement $q \colon \bZ^{2g} \to \bZ/2$ of the bilinear form determined in terms of the standard symplectic basis by $q(e_i) = q(f_i) = 0$. In particular $G_g'$ always has finite index in $G_g$, so is an arithmetic group.

The classifying space of the Torelli group $\mr{Tor}(W_{g}, D^{2n}) \coloneqq \mathrm{Ker}(\alpha_g)$ therefore fits into a fibration sequence
\begin{equation}\label{eq:TorelliFib}
B\mr{Tor}(W_{g}, D^{2n}) \overset{i}\lra B\mr{Diff}(W_{g}, D^{2n}) \lra BG_g',
\end{equation}
so there is an action (up to homotopy) of $G_g'$ on $B\mr{Tor}(W_{g}, D^{2n})$. Hence the cohomology groups $H^*(B\mr{Tor}(W_{g}, D^{2n});\bQ)$ form a commutative ring object in the category of graded $G'_g$-representations (with the Koszul sign rule).

The local coefficient system $\cH(g)_\bQ$ on $B\mr{Diff}(W_{g}, D^{2n})$ is equipped with a canonical trivialisation $i^*\cH(g)_\bQ \overset{\sim}\to H(g)$ when pulled back to $B\mr{Tor}(W_{g}, D^{2n})$, where we recall that $H(g)$ denotes the standard $2g$-dimensional representation of $G'_g$. For any finite set $S$ the edge homomorphism for the spectral sequence of the fibration \eqref{eq:TorelliFib} with $\cH(g)_\bQ^{\otimes S}$-coefficients is then
\begin{equation*}
H^*(B\mr{Diff}(W_{g}, D^{2n}); \cH(g)_\bQ^{\otimes S}) \lra \left[H^*(B\mr{Tor}(W_{g}, D^{2n});\bQ) \otimes H(g)^{\otimes S}\right]^{G_g'}.
\end{equation*}
Composing this with the maps $\Phi_S$ given in \eqref{eq:PhiS}, and writing as usual 
\[\cV = H^*(B\mr{SO}(2n)\langle n \rangle;\bQ) = \bQ[e, p_{\lceil \frac{n+1}{4}\rceil}, \ldots, p_{n-1}]\]
with homogeneous basis of monomials $\cB$, we obtain maps
\[\Phi_S^t \colon \mathcal{P}(S, \cV)_{\geq 0}  \otimes (\det \bQ^{S})^{\otimes n} \lra [H^*(B\mr{Tor}(W_{g}, D^{2n});\bQ) \otimes H(g)^{\otimes S}]^{G_g'}\]
and hence, by adjunction, $G_g'$-equivariant maps
\[\Psi_S^t \colon (H(g)^{\otimes S})^\vee \otimes \mathcal{P}(S, \cV)_{\geq 0}  \otimes (\det \bQ^{S})^{\otimes n} \lra H^*(B\mr{Tor}(W_{g}, D^{2n});\bQ).\]
We now adopt the functorial perspective of Sections \ref{sec:RepOfCat} and \ref{sec:RepOfBrauer}. As the $\Phi_S^t$ are the components of a natural transformation of functors $\mathsf{(s)Br}_{2g} \to \mathsf{Gr}(\bQ\text{-mod})$, the $\Psi_S^t$ extend to a map
\[\Psi^t \colon K^\vee \otimes^{\mathsf{(s)Br}_{2g}} (\mathcal{P}(-, \cV)_{\geq 0}^{2g}  \otimes {\det}^{\otimes n}) \lra H^*(B\mr{Tor}(W_{g}, D^{2n});\bQ).\]
As in Section \ref{sec:functoriality-brauer} we recognise the term $\mathcal{P}(-, \cV)_{\geq 0}^{2g}  \otimes {\det}^{\otimes n}$ as being left Kan extended along $i \colon \cat{d(s)Br} \to \cat{(s)Br}_{2g}$, and we can rewrite the domain to get
\[\Psi^t \colon i^*(K^\vee) \otimes^{\mathsf{d(s)Br}} (\mathcal{P}(-, \cV)'_{\geq 0}  \otimes {\det}^{\otimes n}) \lra H^*(B\mr{Tor}(W_{g}, D^{2n});\bQ).\]
Recall that $\cP(-,\cV)'$ is distinguished from $\cP(-,\cV)$ by not allowing parts of size $2$ labelled by $1 \in \cV$. 

In particular, restricting to $S=\varnothing$ gives a ring homomorphism
\[\begin{tikzcd}\mathcal{P}(\varnothing, \cV)'_{\geq 0} = (H(g)^{\otimes \varnothing})^\vee \otimes \mathcal{P}(\varnothing, \cV)'_{\geq 0}  \otimes (\det \bQ^{\varnothing})^{\otimes n} \dar \\ i^*(K^\vee) \otimes^{\mathsf{d(s)Br}} (\mathcal{P}(-, \cV)'_{\geq 0}  \otimes {\det}^{\otimes n})\end{tikzcd}\]
which sends the labelled partition $(\varnothing, c)$ of $\varnothing$ to a class we shall call $\kappa_c$, as it maps to the Miller--Morita--Mumford class of this name under $\Psi^t$. In particular, taking the labels to be the Hirzebruch $L$-classes $\cL_i$ defines classes
\[\kappa_{\cL_i} \in i^*(K^\vee) \otimes^{\mathsf{d(s)Br}} (\mathcal{P}(-, \cV)'_{\geq 0}  \otimes {\det}^{\otimes n})\]
of degree $4i-2n$. These lie in the kernel of $\Psi^t$, as they are defined on $B\diff(W_g, D^{2n})$ and are pulled back from $BG'_g$ by a theorem of Atiyah \cite{AtiyahFib}, so vanish on $B\mr{Tor}(W_{g}, D^{2n})$ by the fibration sequence \eqref{eq:TorelliFib}. Thus the ideal generated by these classes also lies in the kernel of $\Psi^t$.

\begin{theorem}\label{thm:MainCalc}
If $2n \geq 6$ the ring homomorphism
\[\frac{i^*(K^\vee) \otimes^{\mathsf{d(s)Br}} \left(\mathcal{P}(-, \cV)_{\geq 0}'  \otimes {\det}^{\otimes n}\right)}{(\kappa_{\cL_i} \, | \, 4i-2n > 0)} \lra H^*(B\mr{Tor}(W_{g}, D^{2n});\bQ)\]
induced by $\Psi^t$ is an isomorphism onto the maximal algebraic $G_g'$-subrepresentation of $H^*(B\mr{Tor}(W_{g}, D^{2n});\bQ)$ in a range of degrees tending to infinity with $g$.

If $2n=2$ and $H^*(B\mr{Tor}(W_{g}, D^{2});\bQ)$ is finite dimensional in degrees $* < N$ for all large enough $g$, then this homomorphism is an isomorphism onto the maximal algebraic $G_g'$-subrepresentation in degrees $* \leq N$, and is a monomorphism in degree $N+1$, for all large enough $g$.
\end{theorem}

\begin{remark}
In \cite{KR-WAlg} we shall prove that $H^*(B\mr{Tor}(W_{g}, D^{2n});\bQ)$ is an algebraic $G_g'$-representation when $2n \geq 6$, so this theorem identifies the target completely in a stable range.
\end{remark}

As part of the proof of this theorem, we will need the following condition guaranteeing collapse of a Serre spectral sequence in a range of degrees.

\begin{lemma}\label{lem:SScollapse}
Let $F \to E \to X$ be a Serre fibration with $X$ path-connected, $\cM$ a local system of $\bQ$-module coefficients on $E$, and suppose that
\begin{enumerate}[\indent (i)]
\item $H^*(E;\cM)$ is a free $H^*(X;\bQ)$-module in degrees $* \leq N+1$,

\item the Serre spectral sequence has a product structure in a range, in the sense that the cup product map
\[H^p(X;\bQ) \otimes H^0(X ; \cH^q(F ; \cM)) \lra H^p(X; \cH^q(F;\cM)) = E^{p,q}_2\]
is an isomorphism when $q < N$ and $p+q \leq N+1$.
\end{enumerate}
Then there are no differentials out of $E_r^{p,q}$ for $p+q \leq N$ and any $r \geq 2$.
\end{lemma}
\begin{proof}
Suppose that that $d_r \colon E_r^{p,q} \to E_r^{p+r, q-r+1}$ is non-zero, with $p+q \leq N$. Then, by the product structure, the differential $d_r \colon E_r^{0,q} \to E_r^{r, q-r+1}$ is also non-zero. Without loss of generality we may suppose that $q$ is minimal with this property. Let $\{\bar{b}_i\}$ be free $H^*(X;\bQ)$-module generators for $H^*(E;\cM)$ in degrees $\leq N+1$. As $E_2^{0,*}=H^0(X ; \cH^*(F ; \cM))$ consists of permanent cycles for $* < q$, the map
\[\bQ \otimes_{H^*(X;\bQ)} H^*(E;\cM) \lra H^0(X ; \cH^*(F;\cM))\]
is surjective in degrees $* < q$, and so the restrictions $b_i$ of the $\bar{b}_i$ to $H^0(X;\cH^*(F;\cM))$ generate it in degrees $* < q$. A non-zero differential $d_r \colon E_r^{0,q} \to E_r^{r, q-r+1}$ would hit some 
\[\sum x_i \otimes b_i \in E_r^{r, q-r+1}=H^r(X;\bQ) \otimes H^0(X ; \cH^{q-r+1}(F ; \cM))\]
in total degree $q+1 \leq N+1$, which would say that $\sum x_i \cdot \bar{b}_i \in H^{q+1}(E;\cM)$ is zero modulo elements of Serre filtration $> r$. But in total degree $(q+1)$ all such elements are contained in the submodule $H^*(X;\bQ)\{\bar{b}_i \, | \, |\bar{b}_i| \leq q-r\} \leq H^*(E;\cM)$, which would say that there was a non-trivial linear dependence $\sum y_i \cdot \bar{b}_i=0 \in H^*(X;\bQ)\{\bar{b}_i\}$, a contradiction.
\end{proof}

\begin{proof}[Proof of Theorem \ref{thm:MainCalc}]
To give a unified treatment of the cases $2n=2$ and $2n \geq 6$, we proceed under the assumption that $H^*(B\mr{Tor}(W_{g}, D^{2n});\bQ)$ is finite dimensional in degrees $* < N$ for all $g$ large enough, and we shall establish the conclusion in degrees $* \leq N$. The first author has shown \cite[Corollary 5.5]{kupersdisk} that $H^*(B\mr{Tor}(W_{g}, D^{2n});\bQ)$ is finite dimensional in all degrees for $2n \geq 6$, giving the claimed conclusion in this case.

Consider the Serre spectral sequence with $\cH(g)^{\otimes S}_\bQ$-coefficients for the fibration \eqref{eq:TorelliFib}, which takes the form
\[E_2^{p,q} = H^p(BG'_g ; \cH^q(B\mr{Tor}(W_{g}, D^{2n}); \bQ) \otimes \cH(g)_\bQ^{\otimes S}) \Longrightarrow H^{p+q}(B\mr{Diff}(W_{g}, D^{2n}); \cH(g)_\bQ^{\otimes S}).\]
We wish to apply Lemma \ref{lem:SScollapse} to this spectral sequence, so must verify its hypotheses.

As $H^q(B\mr{Tor}(W_{g}, D^{2n}); \bQ) \otimes H(g)^{\otimes S}$ is finite dimensional for $q < N$ and $g$ large enough, by assumption, Theorem \ref{thm.margulis} implies that it is an almost algebraic representation of $G'_g$. Hence by Theorem \ref{thm.borelvanishingweak} the cup-product map
\begin{equation}\label{eq:EdgeComp}
\begin{tikzcd}H^p(BG_g';\bQ) \otimes [H^q(B\mr{Tor}(W_{g}, D^{2n}); \bQ) \otimes H(g)^{\otimes S}]^{G_g'} \dar \\ H^p(BG'_g ; \cH^q(B\mr{Tor}(W_{g}, D^{2n}); \bQ) \otimes \cH(g)_\bQ^{\otimes S})\end{tikzcd}
\end{equation}
is an isomorphism if both $q < N$ and $p+q \leq N+1$, for $g$ sufficiently large. This shows that the Serre spectral sequence has the required product structure. (This map is also clearly an isomorphism for $(p,q)=(0,N)$, so it is an isomorphism in total degrees $p+q \leq N$. Furthermore, Theorem \ref{thm.borelvanishingweak} also says $H^1(G'_g ; \bQ)=0$, so it is also a monomorphism in total degrees $p+q \leq N+1$.)

On the other hand we have computed $H^{*}(B\mr{Diff}(W_{g}, D^{2n}); \cH(g)^{\otimes S}_\bQ)$ for $2n \neq 4$ in a range of degrees in Theorem \ref{thm:Iso}. We saw there that it is a free $H^{*}(B\mr{Diff}(W_{g}, D^{2n});\bQ)$-module in a range of degrees tending to $\infty$ with $g$. The first hypothesis of Lemma \ref{lem:SScollapse} will therefore be fulfilled as long as $H^{*}(B\mr{Diff}(W_{g}, D^{2n});\bQ)$ is a free $H^*(BG_g';\bQ)$-module in a range of degrees tending to $\infty$ with $g$.

Stably we have
\[\lim_{g \to \infty}H^{*}(B\mr{Diff}(W_{g}, D^{2n});\bQ) \cong H^*(\Omega^\infty_0 \mathrm{MT}\theta_n;\bQ) = \bQ[\kappa_c \, \mid \, c \in \cB_{>2n}]\]
and by Theorem \ref{thm.borelvanishingweak} we have
\[\lim_{g \to \infty}H^{*}(BG'_g;\bQ) \cong H^*(BG'_\infty;\bQ) \cong \begin{cases} \bQ[\sigma_2,\sigma_6,\ldots] & \text{if $n$ is odd,} \\
\bQ[\sigma_4,\sigma_8,\ldots] & \text{if $n$ is even}.\end{cases}\]
In both cases these are $\bQ$-cohomologies of infinite loop spaces, so have the structure of primitively generated Hopf algebras. As we described in Section \ref{sec:IntroStabCoh}, the class $\sigma_{4i-2n}$ is chosen so that it pulls back under $\alpha_g$ to $\kappa_{\cL_i}$, the Miller--Morita--Mumford class associated to the $i$th Hirzebruch $\cL$-class (this choice is possible by a theorem of Atiyah \cite{AtiyahFib}). The pullback defines a map of commutative and cocommutative connected Hopf algebras of finite type, so by Borel's structure theorem \cite[Theorem 7.11]{milnormoore} these are free graded-commutative algebras freely generated their sets of primitive elements \cite[Corollary 4.18 (2)]{milnormoore}. Thus $\lim_{g \to \infty}H^{*}(B\mr{Diff}(W_{g}, D^{2n});\bQ)$ is a free $\lim_{g \to \infty}H^{*}(BG'_g;\bQ)$-module if each $\kappa_{\cL_i} \in \bQ\{\kappa_c \, \mid \, c \in \cB\}$ is non-zero, or in other words if $\cL_i \in H^{4i}(B\mr{SO}(2n)\langle n \rangle;\bQ)$ is non-zero for each $i > n/2$.

When $n \leq 3$ this is easy, as then $B\mr{SO}(2n)\langle n \rangle = B\mr{SO}(2n)$ and $\cL_i$ contains $p_1^i$ with non-zero coefficient and $p_1 \in H^4(B\mr{SO}(2n);\bQ)$ is a non-zero polynomial generator. For the general case we rely on the recent theorem of Berglund--Bergstr{\"o}m \cite{BB} that $\cL_i$ has every possible coefficient non-zero. As there is a monomial in $\smash{p_{\lceil \frac{n+1}{4}\rceil}}, \ldots, p_{n}$ having degree $4i$ for every $4i > 2n$, it follows that $\cL_i \neq 0 \in H^{4i}(B\mr{SO}(2n)\langle n \rangle;\bQ)$ for each $i > n/2$.

We have verified the hypotheses of Lemma \ref{lem:SScollapse}, so for large enough $g$ the spectral sequence has no differentials starting in total degree $p+q \leq N$. The spectral sequence is one of $H^*(BG'_g;\bQ)$-modules, and tensoring down gives a map
\[\bQ \otimes_{\bQ[\kappa_{\cL_i} \, | \, 4i-2n>0]} H^*(B\mr{Diff}(W_{g}, D^{2n});\cH(g)_\bQ^{\otimes S}) \lra [H^*(B\mr{Tor}(W_{g}, D^{2n}); \bQ) \otimes H(g)^{\otimes S}]^{G_g'}\]
which is an isomorphism in degrees $* \leq N$ and a monomorphism in degree $N+1$. Using Theorem \ref{thm:Iso} this shows that the natural map
\[\bQ \otimes_{\bQ[\kappa_{\cL_i} \, | \, 4i-2n>0]} \mathcal{P}(S, \cV)_{\geq 0}  \otimes (\det \bQ^{S})^{\otimes n} \lra [H^*(B\mr{Tor}(W_{g}, D^{2n}); \bQ) \otimes H(g)^{\otimes S}]^{G_g'}\]
is an isomorphism in degrees $* \leq N$ and a monomorphism in degree $N+1$.

Tracing through the maps involved shows that this map is induced by $\Phi^t_S$. In particular it shows that the natural transformation
\[ \bQ \otimes_{\bQ[\kappa_{\cL_i} \, \mid \, 4i-2n >0]} \left(\mathcal{P}(-, \cV)_{\geq 0}^{2g}  \otimes {\det}^{\otimes n}\right) \Longrightarrow [H^*(B\mr{Tor}(W_{g}, D^{2n}); \bQ) \otimes H(g)^{\otimes -}]^{G_g'}\]
of functors $\mathsf{(s)Br}_{2g} \to \mathsf{Gr}(\bQ\text{-mod})$ is an isomorphism in degrees $* \leq N$ and a monomorphism in degree $N+1$. The left-hand side is the Kan extension from $\cat{d(s)Br}$ to $\mathsf{(s)Br}_{2g}$ of the functor
\[\bQ \otimes_{\bQ[\kappa_{\cL_i} \, \mid \, 4i-2n >0]}\left(\mathcal{P}(-, \cV)_{\geq 0}'  \otimes {\det}^{\otimes n} \right) \colon \mathsf{d(s)Br} \lra \mathsf{Gr}(\bQ\text{-mod}).\]

To finish the argument we apply Proposition \ref{prop:Recognition} with $B = H^i(B\mr{Tor}(W_{g}, D^{2n});\bQ)$ for any $i \leq N$, $A$ the degree $i$ part of $\bQ \otimes_{\bQ[\kappa_{\cL_i} \, \mid \, 4i-2n >0]}\left(\mathcal{P}(-, \cV)_{\geq 0}' \otimes {\det}^{\otimes n} \right)$, and $\phi^{\mathsf{Br}_{2g}}$ given by the natural isomorphism above. 
\end{proof}

There is a final consequence of the proof of this theorem which it is useful to record.

\begin{proposition}\label{prop:LsAreRegular}
The sequence $\{\kappa_{\cL_i}\}_{i > \frac{2n}{4}}$ acts regularly on the ring $i^*(K^\vee) \otimes^{\mathsf{d(s)Br}} \left(\mathcal{P}(-, \cV)_{\geq 0}'  \otimes {\det}^{\otimes n}\right)$ in a range of degrees tending to infinity with $g$.
\end{proposition}
\begin{proof}
As discussed in the proof of Theorem \ref{thm:MainCalc}, $\bQ[\kappa_c \, | \, \cB_{>2n}]$ is a free module over $\bQ[\kappa_{\cL_i} \, | \, 4i-2n > 0]$. In addition, each $\mathcal{P}(S, \cV)_{\geq 0}'$ is a free module over the subring $\cP(\varnothing, \cV)_{\geq 0} = \bQ[\kappa_c \, | \, \cB_{>2n}]$, so the sequence $\{\kappa_{\cL_i}\}_{i > \frac{2n}{4}}$ acts regularly on each $\mathcal{P}(S, \cV)_{\geq 0}'$, so also on each $\mathcal{P}(S, \cV)_{\geq 0}'  \otimes (\det \bQ^S)^{\otimes n}$.

By Corollary \ref{cor:TransfDetectsZero}, the functor
\[i^*(K^\vee) \otimes^{\mathsf{d(s)Br}} - \colon ((\bQ\text{-mod})^{\mathsf{d(s)Br}})^f \lra \mathsf{Rep}(G'_g)\]
detects whether a morphism between objects which are supported on finite sets of cardinality $\leq g$ is a monomorphism. In a range of homological degrees tending to infinity with $g$ the object $\mathcal{P}(S, \cV)_{\geq 0}'  \otimes (\det \bQ^S)^{\otimes n}$ has such support (see Section \ref{sec:SupportEstimate} for a quantitative discussion of this), so the claim follows.
\end{proof}

\section{Ring structure}\label{sec:GraphicalCalc}

We may abstract some of the constructions made so far as follows. Let $\cV$ be a graded $\bQ$-algebra of finite type and concentrated in even degrees, and let $e \in \cV_{2n}$. Using this data we may construct a lax symmetric monoidal functor $\mathcal{P}(-, \cV)_{\geq 0}' \colon \mathsf{d(s)Br} \to \mathsf{Gr}(\bQ\text{-mod})$ by analogy with Sections \ref{sec:functoriality-FB}, \ref{sec:functoriality-brauer}, and \ref{sec:functoriality-mult}, and hence form the ring
\[R^\cV \coloneqq K^\vee \otimes^{\mathsf{d(s)Br}}\left(\mathcal{P}(-, \cV)_{\geq 0}' \otimes  {\det}^{\otimes n}\right).\]
One may rephrase Theorem \ref{thm:MainCalc} as saying that for $\cV = H^*(B\mr{SO}(2n) \langle n \rangle;\bQ)$ with $e \in H^{2n}(B\mr{SO}(2n) \langle n \rangle;\bQ)$ the Euler class there is a ring homomorphism
\[\frac{R^\cV}{(\kappa_{\cL_i} \, | \, 4i-2n > 0)} \overset{\cong}{\lra} H^*(B\mr{Tor}(W_{g}, D^{2n});\bQ)\]
which is an isomorphism in a range of degrees tending to infinity with $g$. Here the element $\kappa_{\cL_i}$ corresponds to the part of size 0 labelled by $\cL_i$; these form a regular sequence in a stable range by Proposition \ref{prop:LsAreRegular}. In order to make computational use of Theorem \ref{thm:MainCalc} it is useful to identify the ring $R^\cV$ with something more palatable. 

This is a purely algebraic question which can be asked for any $\cV$: in this section we will provide a generators and relations description of the ring $R^\cV$.

\subsection{Generators}
In this section we will freely identify $H(g)^\vee$ with $H(g)$ using Poincar{\'e} duality. We have been considering $H(g)$ as $H_n(W_g;\bQ)$, so the identification
\[H(g)^\vee = H_n(W_g;\bQ)^\vee = H^n(W_g;\bQ) \xrightarrow{- \frown [W_g]} H_n(W_g;\bQ) = H(g)\]
is inverse to $v \mapsto \lambda (v,-) \colon H(g) \overset{\sim}\lra H(g)^\vee$.

By the universal property of coends, for any finite set $S$ there are $\Sigma_S \times G'_g$-equivariant maps
\[H(g)^{\otimes S} \otimes \mathcal{P}(S, \cV)_{\geq 0}' \otimes (\det \bQ^S)^{\otimes n} \lra R^\cV,\]
where the target has the trivial $\Sigma_S$-action. If $c \in \cV$ is an allowed label for parts of size $k$, the labelled partition $\{(\{1,2, \ldots, k\}; c)\}$ of $\{1,2, \ldots, k\}$ gives a $\Sigma_k \times G'_g$-equivariant map $H(g)^{\otimes k} \otimes  (\det \bQ^k)^{\otimes n} \to R^\cV$ and so, forgetting the $\Sigma_k$-action, a $G'_g$-equivariant map
\[\kappa_c \colon  H(g)^{\otimes k} \lra R^\cV.\]
This construction is linear in $c$. We may record the $\Sigma_k$-equivariance of the original map by the identity
\begin{equation}\label{eq:rel0}
\kappa_c(v_{\sigma(1)} \otimes \cdots \otimes v_{\sigma(k)}) = \mr{sign}(\sigma)^n \cdot \kappa_c(v_1 \otimes \cdots \otimes v_k)
\end{equation}
for any $\sigma \in \Sigma_k$. Recall that the labelled partition $\{(\{1,2, \ldots, k\}; c)\}$ is given degree $|c| + n(k-2)$, so $\kappa_c(v_1 \otimes \cdots \otimes v_k)$ lies in this degree.

\subsection{Relations}We find relations between the $\kappa_c(v_1 \otimes \cdots \otimes v_k)$ by giving pairs of classes which map to the same element in $R^\cV$.

Let $a_i$ for $1 \leq i \leq 2g$ be a basis of $H(g)$ and $a_i^\#$ for $1 \leq i \leq 2g$ be the dual basis characterised by $\lambda(a_i^\#, a_j)=\delta_{ij}$, then the form $\omega$ dual to the pairing $\lambda$, determined by $(\lambda \otimes \mr{id})(- \otimes \omega) = \mr{id}(-)$, is given by $\omega = \sum_{i=1}^{2g} a_i \otimes a_i^\# \in H(g)^{\otimes 2}$.

Let $S = \{s_1, \ldots, s_p\}$ and $T = \{t_1, \ldots, t_q\}$ be finite sets, and consider enlarged sets $S' = S \sqcup \{s\}$ and $T' = \{t\} \sqcup T$. Let $v \in H(g)^{\otimes S}$ and $w \in H(g)^{\otimes T}$. In the coend defining $R^\cV$, the class
\begin{align*}
\sum_{i=1}^{2g} (v \otimes a_i \otimes a_i^\# \otimes w) \otimes \{(S', x), (T', y)\} \otimes (s_1 \wedge \cdots \wedge s_p \wedge s \wedge t \wedge t_1 \wedge \cdots \wedge t_q)^{\otimes n}\\
\in H(g)^{\otimes S' \sqcup T'} \otimes \mathcal{P}(S' \sqcup T', \cV)_{\geq 0}' \otimes (\det \bQ^{S' \sqcup T'})^{\otimes n}
\end{align*}
is identified with the class
\begin{align*}(v \otimes w) \otimes \{(S \sqcup T, x \cdot y)\} \otimes (s_1 \wedge \cdots \wedge s_p \wedge  t_1 \wedge \cdots \wedge t_q)^{\otimes n}\\
\in H(g)^{\otimes S \sqcup T} \otimes \mathcal{P}(S \sqcup T, \cV)_{\geq 0}' \otimes (\det \bQ^{S \sqcup T})^{\otimes n},
\end{align*}
which gives the identity
\begin{equation}\label{eq:relEdge}
 \sum_i \kappa_x(v \otimes a_i) \cdot  \kappa_y(a_i^\# \otimes w) = \kappa_{x \cdot y}(v \otimes w).
\end{equation}

Similarly, the class
\begin{align*}
\sum_i  (v \otimes a_i \otimes a_i^\#) \otimes \{(S \sqcup \{s,t\}, x)\} \otimes (s_1 \wedge \cdots \wedge s_p \wedge s \wedge t)^{\otimes n} \\
\in H(g)^{\otimes S \sqcup \{s, t\}} \otimes \mathcal{P}(S \sqcup \{s,t\}, \cV)_{\geq 0}' \otimes (\det \bQ^{S \sqcup \{s,t\}})^{\otimes n}
\end{align*}
is identified with the class
\[v \otimes \{(S, e\cdot x)\} \otimes (s_1 \wedge \cdots \wedge s_p)^{\otimes n}
\in H(g)^{\otimes S } \otimes \mathcal{P}(S , \cV)_{\geq 0}' \otimes (\det \bQ^{S})^{\otimes n},\]
which gives the identity
\begin{equation}\label{eq:relLoop}
\sum_i \kappa_x( v \otimes a_i \otimes a_i^\#) = \kappa_{e \cdot x}(v).
\end{equation}

\subsection{The ring presentation}

Our main result describing the ring $R^\cV$ is that the above gives a complete set of generators and relations for it in a stable range, as follows.

\begin{theorem}\label{thm:BigRingStr}
In a range of degrees tending to infinity with $g$, the graded-commutative ring $R^\cV$ is generated by the classes $\kappa_c(v_1 \otimes \cdots \otimes v_r)$ with $c$ a homogeneous element of $\cV$, $r \geq 0$, and $|c| + n(r-2) >0$, subject to
\begin{enumerate}[\indent (i)]
\item linearity in $c$ and in each $v_i$,

\item the symmetry relation \eqref{eq:rel0},

\item the contraction relations \eqref{eq:relEdge} and \eqref{eq:relLoop}.
\end{enumerate}
\end{theorem}

The details of the proof of this theorem are somewhat technical, but the underlying idea is quite simple: here is a synopsis. Letting $R^\cV_\mr{pres}$ be the commutative ring given by the presentation in the statement of the theorem, the fact that these relations indeed hold in $R^\cV$ gives a morphism $\phi \colon R^\cV_\mr{pres} \to R^\cV$. Both source and target are graded algebraic representations of finite type, so in any finite range of degrees only finitely-many isomorphism types of irreducible representations appear, which may be described independently of $g$. As each irreducible is detected by applying $[- \otimes H(g)^{ \otimes S}]^{G'_g}$ for some $S$, it is enough to show that for each $S$ the map $[\phi \otimes H(g)^{ \otimes S}]^{G'_g}$ is an isomorphism in a range of degrees tending to infinity with $g$.

Using an idea introduced by Morita \cite{Morita}, and developed by Kawazumi--Morita \cite{KM}, Garoufalidis--Nakamura \cite{GN}, and Akazawa \cite{Akazawa}, we will describe a certain space $\cG(S, \cV)$ of graphs with legs $S$ and with internal vertices labelled by elements of $\cV$, up to a certain rule for contracting internal edges and contracting loops, and we will construct a map
\[\cG(S, \cV) \lra [R^\cV_\mr{pres} \otimes H(g)^{ \otimes S}]^{G'_g}\]
which will be shown to be an epimorphism using Theorem \ref{thm:FTInvariantTheory}. This is to be interpreted as $\kappa_c(v_1 \otimes \cdots \otimes v_r)$ representing an $r$-valent corolla labelled by $c$ with an ordering of the incident half-edges, \eqref{eq:rel0} describes the effect of reordering these half-edges, \eqref{eq:relEdge} says that an edge between two labelled corollas may be contracted to form a new corolla labelled by the product of the previous labels, and \eqref{eq:relLoop} says that a loop at a labelled corolla may be contracted to give a new corolla with its label multiplied by $e$.

\begin{figure}[h]
	\begin{tikzpicture}
	\begin{scope}[scale=1]
	\node at (-3,0) [left] {$\kappa_c(v_1,\ldots,v_7)$};
	\draw [<->] (-2.7,0) -- (-1.7,0);
	\foreach \i in {1,...,7}
	{
		\draw [thick,Mahogany] (0,0) -- ({360/7*\i}:.9);
		\draw [thick,Mahogany,dotted] (0,0) -- ({360/7*\i}:1.1);
		\node [Mahogany] at ({360/7*\i}:1.25) {\footnotesize \i};
	}
	\node at (0,0) [Mahogany] {$\bullet$};
	\node [Mahogany] at (.3,-.15) {$c$};
	\end{scope}
	\end{tikzpicture}
	\caption{A corolla with $7$ ordered incident edges, and vertex labelled by $c$.}
\end{figure}
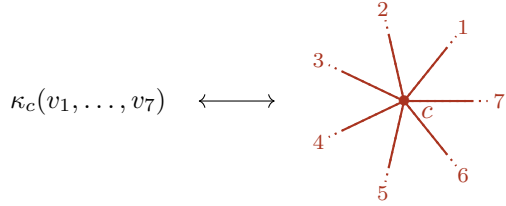

On the other hand by Proposition \ref{prop:TransfOfTransf} there is a  map
\[\psi^{\mathsf{(s)Br}_{2g}} \colon \cP(S, \cV)_{\geq 0} \otimes (\det \bQ^S)^{\otimes n}= i_*(\cP(S, \cV)'_{\geq 0} \otimes (\det \bQ^S)^{\otimes n}) \lra [R^\cV \otimes H(g)^{ \otimes S}]^{G'_g}\]
which is an isomorphism in a range of homological degrees tending to infinity with $g$. Contracting all internal edges will show that $\cG(S, \cV)$ is isomorphic to $\cP(S, \cV)_{\geq 0} \otimes (\det \bQ^S)^{\otimes n}$ and the following diagram commutes
\[\begin{tikzcd} \cG(S,\cV) \rar[two heads] \dar[swap]{\cong} &[20pt] {[R^\cV_\mr{pres} \otimes H(g)^{ \otimes S}]^{G'_g}} \dar{[\phi \otimes H(g)^{ \otimes S}]^{G'_g}} \\ 
\cP(S, \cV)_{\geq 0} \otimes (\det \bQ^S)^{\otimes n} \rar{\psi^{\mathsf{(s)Br}_{2g}}} & {[R^\cV \otimes H(g)^{\otimes S}]^{G'_g}}.\end{tikzcd}\]
Hence $[\phi \otimes H(g)^{ \otimes S}]^{G'_g}$ is an isomorphism in a range of homological degrees tending to infinity with $g$.

\begin{proof}[Proof of Theorem \ref{thm:BigRingStr}]
Let us suppose that the graded vector space $\cV$ has finite type (i.e.\ is finite dimensional in each degree); the general case follows from this by taking colimits over all finite type subspaces.

A \emph{marked oriented graph} with legs $S$ and vertices labelled by $\cV$ consists of the following data:
\begin{enumerate}[\indent (i)]
\item a totally ordered finite set $V$ (of vertices), a totally ordered finite set $H$ (of half-edges), and a monotone function $a \colon H \to V$ (encoding that a half-edge $h$ is incident to the vertex $a(h)$),\footnote{Given the monotonicity of $a$, the total orders of $V$ and $H$ are equivalent to ordering first the vertices and then the half-edges incident to each vertex.}
\item an ordered matching $m = \{(a_i, b_i)\}_{i \in I}$ of the set $H \sqcup S$ (encoding the oriented edges of the graph),
\item a function $c \colon V \to \cV$ with homogenous values, such that $|c(v)| + n(|a^{-1}(v)|-2) > 0$.
\end{enumerate}
We assign to a graph $\Gamma = (V, H, a, m, c)$ the degree
\[\mr{deg}(\Gamma) \coloneqq \prod_{v \in V} \left(|c(v)| + n(|a^{-1}(v)|-2)\right).\]
Two graphs $\Gamma = (V, H, a, m, c)$ and $\Gamma' = (V', H', a', m', c')$ are isomorphic if there are order-preserving bijections $V \cong V'$ and $H \cong H'$ which intertwine the functions $a$ and $a'$ and $c$ and $c'$, and send the matching $m$ to $m'$. An \emph{oriented graph} (with legs $S$ and vertices labelled by $\cV$) is an isomorphism class of marked oriented graphs. We let $\cC^\mr{or, pre}(S, \cV)$ denote the vector space with basis the oriented graphs with legs $S$ and vertices labelled by $\cV$, and $\cC^\mr{or}(S, \cV)$ denote the quotient vector space given by imposing linearity in the label $c(v)$ at each vertex $v \in v$. We consider these as graded vector spaces, with $[\Gamma]$ placed in degree $\mr{deg}(\Gamma)$.

If $[\Gamma]$ and $[\Gamma']$ are oriented graphs as above, and there are not-necessarily order-preserving bijections $f \colon H \to H'$ and $g \colon V \to V'$ such that $a' \circ f = g \circ a$ and $c' \circ g = c'$, and such that the matching $m'$ of $H' \sqcup S$ differs from $f(m)$ by reversing $k$ pairs, then we wish to declare such graphs equivalent up to a sign. Specifically we want to enforce
\[[\Gamma] = (-1)^{nk} \mr{sign}(f)\cdot \mr{sign}(g) [\Gamma']\]
where $\mr{sign}(g)$ and $\mr{sign}(f)$ are as follows:
\begin{enumerate}[\indent (i)]
\item Let the degree of a vertex $v \in V$ be $|c(v)| + n(|a^{-1}(v)|-2)$, and let $V_{\delta} \subset V$ be the subset of vertices of degree $\delta$. The bijection $g \colon V \to V'$ preserves degree, so induces bijections
 $g_{\delta} \colon V_{\delta} \to V'_{\delta}$. These sets are totally ordered, by restricting the total order from $V$ and $V'$, and so there is an associated sign $\mr{sign}(g_{\delta})$ of this permutation. Then
\[\mr{sign}(g) \coloneqq \prod_{\delta \text{ odd}}\mr{sign}(g_{\delta}).\]

\item For each $v \in V$ the function $f$ gives a bijection $a^{-1}(v) \to (a')^{-1}(g(v))$. These sets are totally ordered, by restricting the total order from $H$ and $H'$, and so there is an associated sign $\mr{sign}(f; v)$ of this permutation. Then
\[\mr{sign}(f) \coloneqq \prod_{v \in V} \mr{sign}(f; v)^n.\]
\end{enumerate}
We let the graded vector space $\cC(S, \cV)$ be the quotient of the graded vector space $\cC^\mr{or}(S, \cV)$ by the subspace generated by the homogeneous differences $[\Gamma] - (-1)^{nk} \mr{sign}(f)\cdot\mr{sign}(g) [\Gamma']$ for all such $[\Gamma]$'s and $[\Gamma']$'s. We further let $\cG(S, \cV)$ be the quotient of the graded vector space $\cC(S, \cV)$ by the space spanned by the differences $[\Gamma]- [\Gamma'']$ when $\Gamma = (V, H, a, m, c)$ and $\Gamma'' = (V'', H'', a'', m'', c'')$ are related by the following moves:
\begin{enumerate}[\indent (i)]
\item an edge contraction; that is, there are $x, y \in H$ which are adjacent with respect to the total order on $H$ and have $a(x) \neq a(y)$, such that $H'' = H \setminus \{x,y\}$ with the induced order, $V'' = V/(a(x) \sim a(y))$ with the induced order (as $a(x)$ and $a(y)$ must be adjacent with respect to the total order on $V$),
\[a'' \colon H'' \overset{\mr{inc}}\lra H \overset{a}\lra V \xrightarrow{\mr{quot}} V'',\]
which is again monotone with respect to these orders, $m = \{x,y\} \sqcup m''$, and $c''([a(x)]) = c(x) \cdot c(y)$.

\item a loop contraction; that is, there are $x, y \in H$ which are adjacent with respect to the total order on $H$ and have $a(x) = a(y)$, such that $H'' = H \setminus \{x,y\}$ with the induced order, $V'' = V$ with the same order, 
\[a'' \colon H'' \xrightarrow{\mr{inc}} H \overset{a}\lra V \overset{=}\lra V'',\]
which is again monotone with respect to these orders, $m = \{x,y\} \sqcup m''$, and $c''(a(x)) = c(x) \cdot e$.
\end{enumerate}

We now construct a map $\alpha \colon \cG(S, \cV) \lra [R^\cV_\mr{pres} \otimes H(g)^{ \otimes S}]^{G'_g}$. We do so by first associating to a graph $\Gamma = (V, H, a, m, c)$ the map
\[\bigotimes_{v \in V} \kappa_c \otimes H(g)^{ \otimes S} \colon \bigotimes_{v \in V} H(g)^{\otimes a^{-1}(v)} \otimes H(g)^{ \otimes S} \lra R^\cV_\mr{pres} \otimes H(g)^{ \otimes S}\]
and then applying this map to the $G'_g$-invariant vector given by
\[\bigotimes_{(a,b) \in m} \omega_{a,b} \in \bigotimes_{v \in V} H(g)^{\otimes a^{-1}(v)} \otimes H(g)^{ \otimes S},\]
to obtain $\alpha([\Gamma]) \in [R^\cV_\mr{pres} \otimes H(g)^{ \otimes S}]^{G'_g}$. This descends to a map from $\cG(S, \cV)$ by construction, because the relations in $R^\cV_\mr{pres}$ allow for the the symmetry, and edge and loop contraction relations we imposed on graphs.

Write $R^\cV_\mr{gen}$ for the graded commutative ring generated by the $\kappa_c(v_1 \otimes \cdots \otimes v_r)$ for $|c| + n(r-2)>0$, modulo linearity in each $v_i$. We may write this as
\[R^\cV_\mr{gen} = S^*\left(\bigoplus_{\stackrel{r \geq 0,}{\delta > -n(r-2)}} \cV_\delta \otimes H(g)^{\otimes r}[\delta + n(r-2)]\right),\]
and the construction above gives a map
\[\cC^\mr{or}(S, \cV) \lra [R^\cV_\mr{gen} \otimes H(g)^{\otimes S}]^{G'_g}.\]
By Theorem \ref{thm:FTInvariantTheory} this is an epimorphism. Imposing the symmetry relation \eqref{eq:rel0} and the contraction relations \eqref{eq:relEdge} and \eqref{eq:relLoop} corresponds to allowing local moves on graphs which correspond to the successive quotients $\cC(S, \cV)$ and $\cG(S, \cV)$, and so the map $\alpha \colon \cG(S, \cV) \to [R^\cV_\mr{pres} \otimes H(g)^{\otimes S}]^{G'_g}$ is obtained by taking the quotient of the above, and so is also an epimorphism.

In each homological degree the functor $\cP(-, \cV)'_{\geq 0} \otimes \det^{\otimes n}$ is non-zero only on sets of bounded cardinality, as each allowed part in the definition of $\cP(-, \cV)'_{\geq 0}$ has strictly positive homological degree (we discuss this more quantitatively in Section \ref{sec:SupportEstimate}). Thus by Proposition \ref{prop:TransfOfTransf} there is a map
\[\psi^{\mathsf{Br}_{2g}} \colon \cP(S, \cV)_{\geq 0} \otimes (\det \bQ^S)^{\otimes n}= i_*(\cP(S, \cV)'_{\geq 0} \otimes (\det \bQ^S)^{\otimes n}) \lra [R^\cV \otimes H(g)^{ \otimes S}]^{G'_g}\]
which is an epimorphism, and is an isomorphism in a range of homological degrees tending to infinity with $g$. The composition
\[\cG(S, \cV) \overset{\alpha}\lra [R^\cV_\mr{pres} \otimes H(g)^{\otimes S}]^{G'_g} \overset{\phi_*}\lra [R^\cV \otimes H(g)^{\otimes S}]^{G'_g}\]
is easily described in terms of the map $\psi^{\mathsf{Br}_{2g}}$. Using the contraction formulas, any graph is equivalent in $\cG(S, \cV)$ to a graph having no internal edges: such a graph has the form $(V, H, a, m, c)$ with $m$ a matching of $H \sqcup S$ having no pairs in $H$. In other words, $m$ is the data of an injection $\iota \colon H \hookrightarrow S$ and an ordered matching $m' = \{(a_i, b_i)\}$ of the complement $S \setminus \iota(H)$. Such a graph determines a labelled partition of $S$, with parts $\{a_i, b_i\}$ labelled by 1, and parts $\iota(a^{-1}(v))$ labelled by $c(v)$. It also determines an orientation of $\bQ^S$ as
\[\iota(h_1) \wedge \cdots \wedge \iota(h_r) \wedge a_1 \wedge b_1 \wedge \cdots \wedge a_s \wedge b_s \in \det(\bQ^S).\]
This describes the composition $\phi_* \circ \alpha$. It clearly shows that $\phi_* \circ \alpha$ is an epimorphism, as $\psi^{\mathsf{Br}_{2g}}$ is and every labelled partition is realised by a graph, namely a disjoint union of corollas. 

To show that $\phi_* \circ \alpha$ is a monomorphism, we now use our assumption that $\cV$ has finite type: then the vector spaces $\cG(S, \cV)$ and $\cP(S, \cV)_{\geq 0} \otimes (\det \bQ^S)^{\otimes n}$ do too, and so to see that $\phi_* \circ \alpha$ is a monomorphism it is enough to show that the dimension of $\cG(S, \cV)$ it at most that of $\cP(S, \cV)_{\geq 0}$ in each homological degree. To see this, contract all internal edges of each graph in $\cG(S, \cV)$: the result is a disjoint union of corollas with vertices with (certain) labels in $\cV$, and the dimension of this space in each degree is precisely the dimension of $\cP(S, \cV)_{\geq 0}$ in that degree. Thus even if certain disjoint unions of labelled corollas are equivalent in $\cG(S, \cV)$, its dimension is most that of $\cP(S, \cV)_{\geq 0} \otimes (\det \bQ^S)^{\otimes n}$. Thus $\phi_* \circ \alpha$ is an isomorphism in a range of degrees tending to infinity with $g$.

Finally, as $\alpha$ is an epimorphism it then follows that both $\alpha$ and $\phi_*$ are isomorphisms in a range of degrees tending to infinity with $g$. That is, for each finite set $S$ the map
\[{\phi_*} \colon [R^\cV_\mr{pres} \otimes H(g)^{\otimes S}]^{G'_g} \lra [R^\cV \otimes H(g)^{\otimes S}]^{G'_g}\]
is an epimorphism, and is an isomorphism in a range of degrees tending to infinity with $g$. The algebraic representation $R^\cV_\mr{pres}$ is generated by the classes $\kappa_c (H(g)^{\otimes r})$ of degree $|c| + n(r-2) > 0$, which can be detected by applying $[- \otimes H(g)^{\otimes r}]^{G_g'}$. Thus in degrees $* \leq d$ there is a finite list, independent of $g$, of irreducible representations $V_\lambda$ appearing in $R^\cV_\mr{pres}$, and hence which could appear in $\mr{Ker}(\phi)$. Thus if $\phi$ were not an isomorphism in degrees $* \leq d$ then this would be detected by applying $[- \otimes H(g)^{\otimes S}]^{G'_g}$ for a fixed finite collection of sets $S$, but by taking $g$ large enough this does not happen.
\end{proof}

\subsection{A smaller ring presentation}\label{sec:SmallPres}

Having understood the proof of Theorem \ref{thm:BigRingStr}, one can hope to simplify the presentation of the ring $R^\cV$ given there by manipulating labelled graphs. At the level of generators a simplification is quite obvious: graphically we may first replace an $r$-valent corolla labelled by $x$ by an $(r+1)$-valent corolla labelled by 1 joined to a univalent corolla labelled by $x$, and then by iterated expansions replace the $(r+1)$-valent vertex labelled by 1 by a trivalent tree with each vertex labelled by 1, see Figure \ref{fig:summary-small}. 

\begin{figure}[h]
	\begin{tikzpicture}
	\begin{scope}[scale=1]
		\foreach \i in {1,...,6}
	{
		\draw [thick,Mahogany] (0,0) -- ({360/6*\i}:.9);
		\draw [thick,Mahogany,dotted] (0,0) -- ({360/6*\i}:1.1);
	}
	\node at (0,0) [Mahogany] {$\bullet$};
	\node [Mahogany] at (.3,-.15) {$x$};
	\draw [<-] (2.7,0) -- (1.7,0);
	\begin{scope}[xshift=4.4cm]
		\foreach \i in {1,...,6}
	{
		\draw [thick,Mahogany] (0,0) -- ({360/7*\i}:.9);
		\draw [thick,Mahogany,dotted] (0,0) -- ({360/7*\i}:1.1);
	}
	\draw [thick,Mahogany] (0,0) -- (1,0);
	\node at (0,0) [Mahogany] {$\bullet$};
	\node  at (1,0) [Mahogany] {$\bullet$};
	\node at (1,0) [Mahogany,right] {$x$};
	\end{scope}
	\draw [<-] (7.1,0) -- (6.1,0);
	\begin{scope}[xshift=9.6cm]
	\begin{scope}[yshift=-.5cm,xshift=-1cm]
	\node at (0,0) [Mahogany] {$\bullet$};
	\foreach \i in {3,4}
	{
		\draw [thick,Mahogany] (0,0) -- ({360/7*\i}:.9);
		\draw [thick,Mahogany,dotted] (0,0) -- ({360/7*\i}:1.1);
	}
	\end{scope}
	\begin{scope}[yshift=-1cm,xshift=-.5cm]
	\node at (0,0) [Mahogany] {$\bullet$};
	\foreach \i in {5,6}
	{
		\draw [thick,Mahogany] (0,0) -- ({360/7*\i}:.9);
		\draw [thick,Mahogany,dotted] (0,0) -- ({360/7*\i}:1.1);
	}
	\end{scope}
	\begin{scope}[yshift=.5cm,xshift=-.3cm]
	\node at (0,0) [Mahogany] {$\bullet$};
	\foreach \i in {1,2}
	{
		\draw [thick,Mahogany] (0,0) -- ({360/7*\i}:.9);
		\draw [thick,Mahogany,dotted] (0,0) -- ({360/7*\i}:1.1);
	}
	\end{scope}
	\draw [thick,Mahogany] (-.3,.5) -- (0,0) -- (-.5,-.5) -- (-.5,-1) -- (-.5,-.5) -- (-1.,-.5);
	\node  at (-.5,-.5) [Mahogany] {$\bullet$};
	\draw [thick,Mahogany] (0,0) -- (1,0);
	\node at (0,0) [Mahogany] {$\bullet$};
	\node  at (1,0) [Mahogany] {$\bullet$};
	\node at (1,0) [Mahogany,right] {$x$};
	\end{scope}
	\end{scope}
	\end{tikzpicture}
	\caption{Replacing an $r$-valent corolla labelled by $x$ by an $(r+1)$-valent corolla labelled by $1$ connected to a univalent corolla labelled by $x$, and then expanding the $(r+1)$-valent vertex to a trivalent tree. We have suppressed the labels $1$ for clarity.}
	\label{fig:summary-small}
\end{figure}
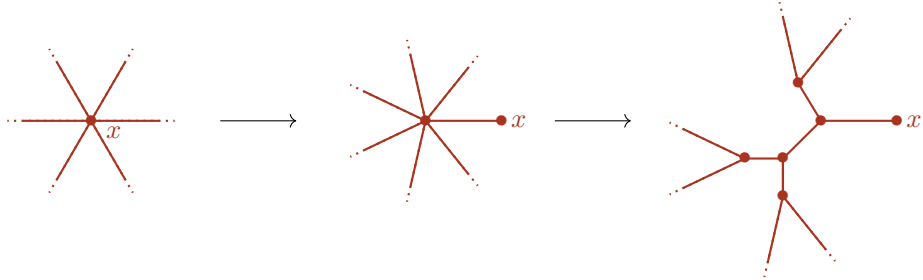

For this to be possible we need to know that if $x$ is a label of some corolla, and $|x|>0$, then the univalent vertex labelled by $x$ exists, i.e.\ $|x| + n(1-2) > 0$, or in other words $|x| > n$. In this section we will therefore suppose that $\cV = \bQ\{1\} \oplus \cV_{>n}$. In this case we see that the classes
\begin{itemize}
\item [(a)] $\kappa_c = \kappa_c(1)$  for $c \in \cV_{>2n}$, of degree $|c|-2n$,
\item [(b)] $\kappa_c(v_1)$  for $c \in \cV_{>n}$, of degree $|c|-n$, and
\item [(c)] $\kappa_1(v_1 \otimes v_2 \otimes v_3)$ of degree $n$,
\end{itemize}
are sufficient to generate $R^\cV$. The price to be paid for this smaller generating set is, as is to be expected, a somewhat more complicated set of relations. The reader will easily deduce from \eqref{eq:rel0}, \eqref{eq:relEdge} and \eqref{eq:relLoop} that along with linearity in $c$ and each $v_i$ the following relations hold among the generators listed above:

\begin{itemize}
\item [($\alpha$)] $\kappa_1(v_{\sigma(1)} \otimes v_{\sigma(2)} \otimes v_{\sigma(3)}) = \mr{sign}(\sigma)^n \cdot \kappa_1(v_1 \otimes v_2 \otimes v_3)$

\item [($\beta$)] $\kappa_{x \cdot y} = \sum_i \kappa_x(a_i) \cdot  \kappa_y(a_i^\#)$

\item [($\gamma$)] $\kappa_{x \cdot y}(v_1) = \sum_{i,j}  \kappa_1(v_1 \otimes a_j  \otimes a_i) \cdot \kappa_x(a_i^\#)\cdot \kappa_y(a_j^\#)$

\item [($\delta$)] $\sum_i \kappa_1( v_1 \otimes a_i \otimes a_i^\#) = \kappa_e(v_1)$

\item [($\epsilon$)] $\sum_i \kappa_1(v_1 \otimes v_2 \otimes a_i) \cdot \kappa_1(a_i^\# \otimes v_5 \otimes v_6) = \sum_i \kappa_1(v_1 \otimes v_5 \otimes a_i) \cdot \kappa_1(a_i^\# \otimes v_6 \otimes v_2).$
\end{itemize}

\begin{theorem}\label{thm:SmallRingStr}
Suppose that $\cV = \bQ\{1\} \oplus \cV_{>n}$. In a range of degrees tending to infinity with $g$, the graded-commutative ring $R^\cV$ is generated by the classes (a)--(c), with relations given by linearity in $c$ and each $v_i$ and the relations ($\alpha$)--($\epsilon$).
\end{theorem}
\begin{proof}
As in the proof of Theorem \ref{thm:BigRingStr}, let $\smash{\bar{R}^\cV_\mr{pres}}$ be the graded commutative ring given by the presentation in the statement of this theorem. Let $\bar{\cC}(S, \cV)$ denote the vector space of graphs analogous to ${\cC}(S, \cV)$, but starting with the subspace $\bar{\cC}^\mr{or, pre}(S, \cV) \subset {\cC}^\mr{or, pre}(S, \cV)$ spanned by those graphs which 
\begin{itemize}
\item[(a)] may have nilvalent vertices,
\item[(b)] may have univalent vertices,
\item[(c)] may have trivalent vertices labelled by 1,
\end{itemize}
but have no higher-valent vertices. Let $\bar{\cG}(S, \cV)$ denote the quotient of $\bar{\cC}(S, \cV)$ by the subspace spanned by differences $[\Gamma]-[\Gamma'']$ where $\Gamma''$ is obtained from $\Gamma$ by one of the local moves shown in Figure \ref{fig:graphrels}.

\begin{figure}[h]
	\begin{subfigure}{\textwidth}
	\centering
	\begin{tikzpicture}
	\begin{scope}[scale=0.6]
	\draw (0,0) circle (1.8cm);
	\foreach \i in {1,...,3}
	{
		\draw [thick,Mahogany] (0,0) -- ({360/3*\i}:1.8);
		\node at ({360/3*\i+25}:1.25) [Mahogany] {\small $\sigma(\i)$};
	}
	\node at (0,0) [Mahogany] {$\bullet$};
	\node at (.3,-.3) [Mahogany] {$x$};
	
	\node at (2.5,0) {=};
	
	\node at (4.2,0) {$\mr{sign}(\sigma)^n$};
	\begin{scope}[xshift=7.5cm]
	\draw (0,0) circle (1.8cm);
	\foreach \i in {1,...,3}
	{
		\draw [thick,Mahogany] (0,0) -- ({360/3*\i}:1.8);
		\node at ({360/3*\i+15}:1.3) [Mahogany] {$\i$};
	}
	\node at (0,0) [Mahogany] {$\bullet$};
	\node at (.3,-.3) [Mahogany] {$x$};
	\end{scope}
	\end{scope}
	\end{tikzpicture}
	\caption*{Relation ($\alpha$).}\label{fig:Move0}
	\end{subfigure}
	\vspace{1cm}
	
	\begin{subfigure}{.4\textwidth}
	\begin{tikzpicture}
	\begin{scope}[scale=0.6]
	\draw (0,0) circle (1.8cm);
	\node at (0.7,0) [Mahogany] {$\bullet$};
	\draw [thick,Mahogany] (-0.7,0) -- (0.7,0);
	\draw [->] [thick,Mahogany] (-0.7,0) -- (0,0);
	\node at (-0.7,0) [Mahogany] {$\bullet$};
	
	\node at (-1.2,0) [Mahogany] {$x$};
	\node at (1.2,0) [Mahogany] {$y$};
	
	\node at (2.5,0) {=};
	
	\begin{scope}[xshift=5cm]
	\draw (0,0) circle (1.8cm);

	\node at (0,0) [Mahogany] {$\bullet$};
		\node at (0.7,0) [Mahogany] {$x \cdot y$};

	\end{scope}
	\end{scope}
	\end{tikzpicture}
	\caption*{Relation ($\beta$).}\label{fig:Move1}
	\end{subfigure}
	\hspace{1cm}
	\begin{subfigure}{.4\textwidth}
	\begin{tikzpicture}
	\begin{scope}[scale=0.6]
	\draw (0,0) circle (1.8cm);
	\draw [->-=.5,thick,Mahogany] (0,-0.5) -- (-1,0.3);
	\draw [->-=.5,thick,Mahogany] (0,-0.5) -- (1,0.3);
	\node at (1,0.3) [Mahogany] {$\bullet$};
	\node at (-1,0.3) [Mahogany] {$\bullet$};
	\node at (0,-0.5) [Mahogany] {$\bullet$};
	\draw [thick,Mahogany] (0,-1.8) -- (0,-0.5);

	\node at (-1,0.3) [above,Mahogany] {$x$};
	\node at (1,0.3) [above,Mahogany] {$y$};
	
		\node at (0.2,-1) [Mahogany] {1};
	\node at (0.8,0) [Mahogany,below] {3};
	\node at (-0.8,0) [Mahogany,below] {2};

	\node at (2.5,0) {=};
	
	\begin{scope}[xshift=5cm]
	\draw (0,0) circle (1.8cm);
	\node at (0,-0.5) [Mahogany] {$\bullet$};
	\draw [thick,Mahogany] (0,-1.8) -- (0,-0.5);
	\node at (0,-0.5) [Mahogany,above] {$x \cdot y$};
	\end{scope}
	\end{scope}
	\end{tikzpicture}
	\caption*{Relation ($\gamma$).}
	\end{subfigure}
	\vspace{1cm}
	
	\begin{subfigure}{.4\textwidth}
	\begin{tikzpicture}
	\begin{scope}[scale=0.6]
	\draw (0,0) circle (1.8cm);
	\draw [-<-=.25,thick,Mahogany]  (0,.3) circle (0.8cm);
	\node at (0,-0.5) [Mahogany] {$\bullet$};
	\draw [thick,Mahogany] (0,-1.8) -- (0,-0.5);
	
	\node at (0.2,-1) [Mahogany] {1};
	\node at (0.8,-0.5) [Mahogany] {3};
	\node at (-0.8,-0.5) [Mahogany] {2};
	
	\node at (2.5,0) {=};
	
	\begin{scope}[xshift=5cm]
	\draw (0,0) circle (1.8cm);
	\node at (0,-0.5) [Mahogany] {$\bullet$};
	\draw [thick,Mahogany] (0,-1.8) -- (0,-0.5);
	\node at (0,-.5) [Mahogany,above] {$e$};
	\end{scope}
	\end{scope}
	\end{tikzpicture}
	\caption*{Relation ($\delta$)}
	\end{subfigure}
	\hspace{1cm}
	\begin{subfigure}{.4\textwidth}
	\begin{tikzpicture}
	\begin{scope}[scale=0.6]
	\draw (0,0) circle (1.8cm);
	\draw [thick,Mahogany] (-1,1.5) -- (0,0.7);
	\draw [thick,Mahogany] (0,0.7) -- (1,1.5);
	\node at (0,0.7) [Mahogany] {$\bullet$};
	\draw [->-=.6,thick,Mahogany] (0,0.7) -- (0,-0.7);
	\node at (0,-.7) [Mahogany] {$\bullet$};
	\draw [thick,Mahogany] (-1,-1.5) -- (0,-0.7);
	\draw [thick,Mahogany] (0,-0.7) -- (1,-1.5);
	\node at (-0.7,.8) [Mahogany] {1};
	\node at (0.7,.8) [Mahogany] {2};
	\node at (0.3,0.4) [Mahogany] {3};
	\node at (0.3,-0.4) [Mahogany] {4};
	\node at (0.7,-0.8) [Mahogany] {5};
	\node at (-0.7,-0.8) [Mahogany] {6};
	\node at (0,0.7) [Mahogany,above] {i};
	\node at (0,-0.7) [Mahogany,below] {ii};
	
	\node at (2.5,0) {=};
	
	\begin{scope}[xshift=5cm]
	\draw (0,0) circle (1.8cm);
	\draw [thick,Mahogany] (-1,1.5) -- (-0.7,0);
	\draw [thick,Mahogany] (0.7,0) -- (1,1.5);
	\node at (0.7,0) [Mahogany] {$\bullet$};
	\draw [->-=.6,thick,Mahogany] (-0.7,0) -- (0.7,0);
	\node at (-0.7,0) [Mahogany] {$\bullet$};
	\draw [thick,Mahogany] (-1,-1.5) -- (-0.7,0);
	\draw [thick,Mahogany] (0.7,0) -- (1,-1.5);
	\node at (-1.2,0.8) [Mahogany] {1};
	\node at (1.2,0.8) [Mahogany] {2};
	\node at (-0.4,0.3) [Mahogany] {3};
	\node at (0.4,0.3) [Mahogany] {4};
	\node at (1.2,-0.8) [Mahogany] {5};
	\node at (-1.2,-0.8) [Mahogany] {6};
	\node at (-1.05,0) [Mahogany] {i};
	\node at (1.1,0) [Mahogany] {ii};
	\end{scope}
	\end{scope}
	\end{tikzpicture}
	\caption*{Relation ($\epsilon$)}\label{fig:Move4}

\end{subfigure}
\caption{The relations ($\alpha$)-($\epsilon$) in graphical form.}
\label{fig:graphrels}
\end{figure}
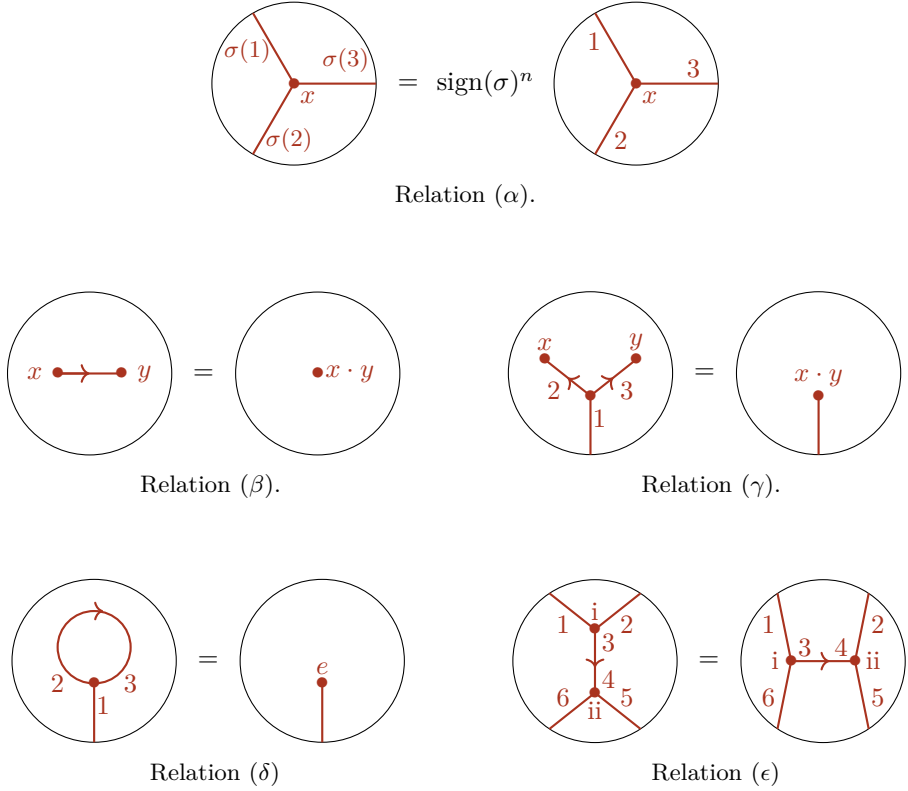

As in the proof of Theorem \ref{thm:BigRingStr} there is a map
\[\bar{\alpha} \colon \bar{\cG}(S, \cV) \lra [\bar{R}^\cV_\mr{pres} \otimes H(g)^{ \otimes S}]^{G'_g},\]
and it is again an epimorphism. Using the commutative diagram
\begin{equation*}
\begin{tikzcd}
\bar{\cG}(S, \cV) \rar[two heads]{\bar{\alpha}} \dar{\beta} & \left[\bar{R}^\cV_\mr{pres} \otimes H(g)^{\otimes S}\right]^{G'_g} \dar\\
{\cG}(S, \cV) \rar{\sim}[swap]{\bar{\alpha}} & \left[{R}^\cV_\mr{pres} \otimes H(g)^{\otimes S}\right]^{G'_g},
\end{tikzcd}
\end{equation*}
to finish the argument we must show that $\beta$ is an isomorphism, and to do so we may use the identification ${\cG}(S, \cV) \smash{\overset{\sim}\lra} \cP(S, \cV)_{\geq 0} \otimes (\det \bQ^S)^{\otimes n}$. We have already explained at the beginning of Section \ref{sec:SmallPres} why $\beta$ is an epimorphism, using the assumption $\cV = \bQ\{1\} \oplus \cV_{>n}$.

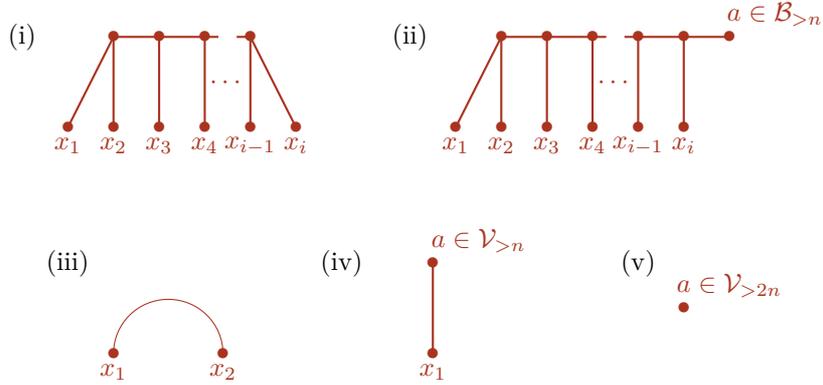
\begin{figure}[h]
	\begin{tikzpicture}
	\begin{scope}[scale=0.6]

	\node at (-3,-2) [Mahogany] {$\bullet$};
	\node at (-2,-2) [Mahogany] {$\bullet$};
	\node at (-1,-2) [Mahogany] {$\bullet$};
	\node at (0,-2) [Mahogany] {$\bullet$};
	\node at (1,-2) [Mahogany] {$\bullet$};
	\node at (2,-2) [Mahogany] {$\bullet$};

	\node at (-2,0) [Mahogany] {$\bullet$};
	\node at (-1,0) [Mahogany] {$\bullet$};
	\node at (0,0) [Mahogany] {$\bullet$};
	\node at (1,0) [Mahogany] {$\bullet$};

	\draw [thick,Mahogany] (-3,-2) -- (-2,0);
	\draw [thick,Mahogany] (-2,-2) -- (-2,0);
	\draw [thick,Mahogany] (-2,0) -- (-1,0);
	\draw [thick,Mahogany] (-1,-2) -- (-1,0);
	\draw [thick,Mahogany] (-1,0) -- (0,0);
	\draw [thick,Mahogany] (0,-2) -- (0,0);
	\draw [thick,Mahogany] (0,0) -- (0.3,0);
	\draw [thick,Mahogany] (0.7,0) -- (1,0);
	\draw [thick,Mahogany] (1,-2) -- (1,0);
	\draw [thick,Mahogany] (1,0) -- (2,-2);

	\node at (-3,-2.4) [Mahogany] {$x_1$};
	\node at (-2,-2.4) [Mahogany] {$x_2$};
	\node at (-1,-2.4) [Mahogany] {$x_3$};
	\node at (0,-2.4) [Mahogany] {$x_4$};
	\node at (1,-2.4) [Mahogany] {$x_{i-1}$};
	\node at (2,-2.4) [Mahogany] {$x_i$};
	\node at (0.5,-1) [Mahogany] {$\cdots$};

	\node at (-4,0) {(i)};

	\begin{scope}[xshift=8.5cm]
	\node at (-3,-2) [Mahogany] {$\bullet$};
	\node at (-2,-2) [Mahogany] {$\bullet$};
	\node at (-1,-2) [Mahogany] {$\bullet$};
	\node at (0,-2) [Mahogany] {$\bullet$};
	\node at (1,-2) [Mahogany] {$\bullet$};
	\node at (2,-2) [Mahogany] {$\bullet$};

	\node at (-2,0) [Mahogany] {$\bullet$};
	\node at (-1,0) [Mahogany] {$\bullet$};
	\node at (0,0) [Mahogany] {$\bullet$};
	\node at (1,0) [Mahogany] {$\bullet$};
	\node at (2,0) [Mahogany] {$\bullet$};
	\node at (3,0) [Mahogany] {$\bullet$};

	\draw [thick,Mahogany] (-3,-2) -- (-2,0);
	\draw [thick,Mahogany] (-2,-2) -- (-2,0);
	\draw [thick,Mahogany] (-2,0) -- (-1,0);
	\draw [thick,Mahogany] (-1,-2) -- (-1,0);
	\draw [thick,Mahogany] (-1,0) -- (0,0);
	\draw [thick,Mahogany] (0,-2) -- (0,0);
	\draw [thick,Mahogany] (0,0) -- (0.3,0);
	\draw [thick,Mahogany] (0.7,0) -- (1,0);
	\draw [thick,Mahogany] (1,-2) -- (1,0);
	\draw [thick,Mahogany] (1,0) -- (2,0);
	\draw [thick,Mahogany] (2,0) -- (3,0);
	\draw [thick,Mahogany] (2,0) -- (2,-2);

	\node at (4,0.5) [Mahogany] {$a \in \cB_{>n}$};	
	\node at (-3,-2.4) [Mahogany] {$x_1$};
	\node at (-2,-2.4) [Mahogany] {$x_2$};
	\node at (-1,-2.4) [Mahogany] {$x_3$};
	\node at (0,-2.4) [Mahogany] {$x_4$};
	\node at (1,-2.4) [Mahogany] {$x_{i-1}$};
	\node at (2,-2.4) [Mahogany] {$x_i$};
	\node at (0.5,-1) [Mahogany] {$\cdots$};

	\node at (-4,0) {(ii)};
	\end{scope}

	\begin{scope}[yshift=-5cm,xshift=-2cm]
	\node at (0,-2) [Mahogany] {$\bullet$};
	\node at (2.4,-2) [Mahogany] {$\bullet$};

	\draw [Mahogany] (0,-2) arc (180:0:1.2) ;

	\node at (0,-2.4) [Mahogany] {$x_1$};
	\node at (2.4,-2.4) [Mahogany] {$x_2$};

	\node at (-1,0) {(iii)};
	\end{scope}
	
	\begin{scope}[yshift=-5cm,xshift=4cm]
	\node at (1,-2) [Mahogany] {$\bullet$};
	\node at (1,0) [Mahogany] {$\bullet$};

	\draw [thick,Mahogany] (1,-2) -- (1,0);

	\node at (2,0.5) [Mahogany] {$a \in \cV_{>n}$};	
	\node at (1,-2.4) [Mahogany] {$x_1$};

	\node at (-1,0) {(iv)};
	\end{scope}

	\begin{scope}[xshift=10.5cm,yshift=-5cm]
	\node at (0,-1) [Mahogany] {$\bullet$};

	\node at (1,-0.5) [Mahogany] {$a \in \cV_{>2n}$};	

	\node at (-1,0) {(v)};
	\end{scope}
	\end{scope}
	\end{tikzpicture}
	\caption{Standard labelled trees for (i) $i \geq 3$, (ii) $i \geq 2$, (iii) $i=2$, (iv) $i=1$, (v) $i=0$.}\label{fig:StdTrees}
\end{figure}

To finish the argument, as in the proof of Theorem \ref{thm:BigRingStr} we may suppose that $\cV$ has finite type; then it is enough to show that in each homological degree the dimension of $\bar{\cG}(S, \cV)$ is \emph{at most} the dimension of ${\cG}(S, \cV)$.

First observe that any labelled graph in $\bar{\cG}(S, \cV)$ is equivalent to a labelled forest, as follows. If a connected graph has a cycle then it has an embedded cycle, in which case the relation ($\epsilon$) can be used to shorten the length of this embedded cycle, and this can be done until the graph has an embedded cycle of length 1: but then the relation ($\delta$) can be used to replace this loop with a leaf labelled by $e$. This reduces the first Betti number of the graph. Continuing in this way, we can eliminate all cycles. Furthermore, by applying relations ($\epsilon$), ($\gamma$), and ($\beta$) each labelled forest is equivalent to a disjoint union of labelled trees of the forms shown in Figure \ref{fig:StdTrees} (i)--(v). This means that in each homological degree the dimension of $\bar{\cG}(S, \cV)$ is at most that of $\mathcal{P}(S, \cV)_{\geq 0}$, which completes the argument. The case of general $\cV$ follows by taking colimits.
\end{proof}

\begin{remark}
We may phrase relation ($\epsilon$) as saying that the compositions
\[\begin{tikzcd}H(g)^{\otimes\{1,2,5,6\}} \ar[shift left=.5ex]{r}{f_I} \ar[shift left=-.5ex]{r}[swap]{f_H} & H(g)^{\otimes 3} \otimes H(g)^{\otimes 3} \rar{\kappa_1 \otimes \kappa_1} & R^\cV\end{tikzcd}\]
are equal, where
\begin{align*}
f_I(v_1 \otimes v_2 \otimes v_5 \otimes v_6) &= \sum_i (v_1 \otimes v_2 \otimes a_i) \otimes (a_i^\# \otimes v_5 \otimes v_6),\\
f_H(v_1 \otimes v_2 \otimes v_5 \otimes v_6) &= \sum_i (v_1 \otimes v_5 \otimes a_i) \otimes (a_i^\# \otimes v_6 \otimes v_2).
\end{align*}
Graphically this corresponds to ``$I=H$": it is somewhat complicated because we are trying to express the fact that edges may be contracted, while only allowing ourselves to consider trivalent graphs.

The class $\kappa_1(v_1 \otimes v_2 \otimes v_3)$ has degree $n$, so the map $\kappa_1 \otimes \kappa_1 \colon H(g)^{\otimes 3} \otimes H(g)^{\otimes 3} \to R^\cV$ factors through $\Lambda^2(H(g)^{\otimes 3})$ if $n$ is odd and through $\mr{Sym}^2(H(g)^{\otimes 3})$ is $n$ is even. Furthermore, by relation \eqref{eq:rel0} the map $\kappa_1 \colon H(g)^{\otimes 3} \to R^\cV$ factors through $\Lambda^3(H(g))$ if $n$ is odd and $\mr{Sym}^3(H(g))$ if $n$ is even. In total it factors through $\Lambda^2(\Lambda^3(H(g)))$ or $\mr{Sym}^2(\mr{Sym}^3(H(g)))$. The following lemma describes the image of the composition
\[H(g)^{\otimes\{1,2,5,6\}} \xrightarrow{f_I - f_H} H(g)^{\otimes 3} \otimes H(g)^{\otimes 3} \xrightarrow{\kappa_1 \otimes \kappa_1} \begin{cases}
\Lambda^2(\Lambda^3(H(g))),\\
\mr{Sym}^2(\mr{Sym}^3(H(g))),
\end{cases}\]
as a $\GG(\bQ)$-representation; the first case is due to Garoufalidis--Nakamura \cite{GN}, and the second case can be proved by the same method.

\begin{lemma}\label{lem:RepThy}
If $n$ is odd then $\Lambda^2(\Lambda^3(H(g)))= V_{1^6} + 2 V_{1^4} +3 V_{1^2} + 2 V_0 + V_{2^2,1^2} + V_{2^2} + V_{2,1^2}$ and
\[\mr{Im}\left[f_I - f_H \colon H(g)^{\otimes 4} \to \Lambda^2(\Lambda^3(H(g)))\right] \cong V_{2^2} + V_{1^2} + V_0.\]

\noindent If $n$ is even then $\mr{Sym}^2(\mr{Sym}^3(H(g))) = 3 V_0 + 4 V_2 + 2 V_{2^2} + V_{2^3} + V_{3,1} + 2 V_4 + V_{4,2} + V_6$ and 
\[\mr{Im}\left[f_I - f_H \colon H(g)^{\otimes 4} \to \mr{Sym}^2(\mr{Sym}^3(H(g)))\right] \cong V_2 + V_{3,1}.\makeatletter\displaymath@qed\]
\end{lemma}
\end{remark}

\subsection{Example: calculation in low degrees}\label{sec:exam-low-degrees} Above we gave a presentation of $R^\cV$; this ring is related to Torelli spaces by a ring homomorphism
\[\frac{R^\cV}{(\kappa_{\cL_i} \mid 4i-2n > 0)} \overset{\cong}{\lra} H^*(B\mr{Tor}(W_{g}, D^{2n});\bQ),\]
which is an isomorphism in a range of degrees tending to infinity with $g$, for $\cV = H^*(B\mr{SO}(2n) \langle n \rangle;\bQ)$. In this section we use the arguments of Theorem \ref{thm:BigRingStr} to compute $R^\cV$ explicitly in degrees $*<2n$ and relate it to pseudoisotopy theory and surgery theory. As usual we give $\cV$ its basis $\cB$ of monomials in Euler and Pontrjagin classes.

Let us define a graded vector space $\cP \coloneqq (\pi_*(BO) \otimes \bQ)^\vee_{>n}$, with basis $\cP\cB$ given by all Pontrjagin classes. If $c \in \cP\cB \cap \cB_{>n}$, we have defined earlier elements $\kappa_c(v) \in R^\cV$ of degree $|c|-n$. We extend this to $c \in \cP\cB \setminus\cB_{>n}$ by declaring $\kappa_c(v) = 0$ if $|c| > 4n$ (these classes will play no role in this computation, as their degree exceeds the range $*<2n$). Together with the classes $\kappa_1(v_1 \otimes v_2 \otimes v_3)$ in $R^\cV$ of degree $n$, these provide a homomorphism
\begin{equation}\psi \colon (\bQ \oplus Y[n]) \otimes S^*(H(g)[-n] \otimes \cP) \lra R^\cV \lra \frac{R^\cV}{(\kappa_{\cL_i} \mid 4i-2n > 0)},\end{equation}
where 
\begin{equation*}
Y[n] \coloneqq \begin{cases}
\Lambda^3(H(g))[n] & \text{$n$ odd} \\ 
\mr{Sym}^3(H(g))[n]  & \text{$n$ even}. 
\end{cases}
\end{equation*}
The following extends the computation of the cohomology of Torelli spaces in the range $*<n-1$ by the second author and Ebert using pseudoisotopy theory \cite{oscarjohannestorelli}.

\begin{proposition}\label{prop:low-degree-graphical} For $*<2n$ and $g$ sufficiently large, $\psi$ is an isomorphism.
\end{proposition}

\begin{proof}That $\psi$ is an isomorphism for $*<2n$ and $g$ sufficiently large can be detected by tensoring with $H(g)^{\otimes S}$ and taking $G'_g$-invariants, for all finite sets $S$:
	\[ \left[(\bQ \oplus Y[n]) \otimes S^*(H(g)[-n] \otimes P) \otimes H(g)^{\otimes S} \right]^{G'_g} \lra \left[\frac{R^\cV}{(\kappa_{\cL_i} \mid 4i-2n > 0)} \otimes H(g)^{\otimes S} \right]^{G'_g}.\]
	
Let $\bar{\cG}(S,\cP)^{(1)}$ denote the vector space of graphs analogous to $\cG(S,\cV)$ or $\bar{\cG}(S,\cV)$ given as follows: we start with the subspace $\bar{\cC}^\mr{or}(S,\cP)^{(1)} \subset \cC^\mr{or}(S,\cP)$ spanned by those graphs which (a) may have univalent vertices, (b) may have a single trivalent vertex labeled by 1, but (c) have no other vertices. Then $\bar{\cG}(S,\cP)^{(1)}$ is the quotient of $\bar{\cC}^\mr{or}(S,\cP)^{(1)}$ by the differences $[\Gamma]-[\Gamma'']$ where $\Gamma''$ differs from $\Gamma$ only by the local move ($\alpha$) of Figure \ref{fig:graphrels}. The only connected components that occur in such graphs are as in Figure \ref{fig:low-deg-graphs}:
\begin{enumerate}[\indent (i)]
	\item a single edge with vertices labelled by $S$ or $\cP$,
	\item a trivalent vertex and univalent vertices labelled by $S$ or $\cP$,
	\item a ``lollipop'' with univalent vertex labelled by $S$ or $\cP$.
\end{enumerate}

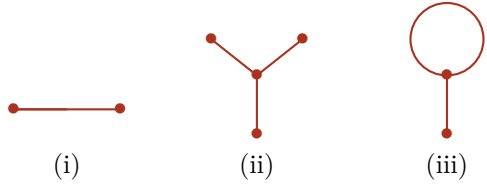
\begin{figure}[h]
	\begin{tikzpicture}
	\node at (0.7,0) [Mahogany] {$\bullet$};
	\draw [thick,Mahogany] (-0.7,0) -- (0.7,0);
	\draw [thick,Mahogany] (-0.7,0) -- (0,0);
	\node at (-0.7,0) [Mahogany] {$\bullet$};
	\node at (0,-.75) {(i)};
	
	\begin{scope}[xshift=2.5cm]
	\begin{scope}[scale=.6,yshift=1.25cm]
	\draw [thick,Mahogany] (0,-0.5) -- (-1,0.3);
	\draw [thick,Mahogany] (0,-0.5) -- (1,0.3);
	\node at (1,0.3) [Mahogany] {$\bullet$};
	\node at (-1,0.3) [Mahogany] {$\bullet$};
	\node at (0,-0.5) [Mahogany] {$\bullet$};
	\draw [thick,Mahogany] (0,-0.5) -- (0,-1.8);
	\node at (0,-1.8) [Mahogany] {$\bullet$};
	\end{scope}
	\node at (0,-.75) {(ii)};
	\end{scope}
	
	\begin{scope}[xshift=5cm]
	\begin{scope}[scale=.6,yshift=1.25cm]
	\draw [thick,Mahogany]  (0,.3) circle (0.8cm);
	\node at (0,-0.5) [Mahogany] {$\bullet$};
	\node at (0,-1.8) [Mahogany] {$\bullet$};
	\draw [thick,Mahogany] (0,-0.5) -- (0,-1.8);
	\end{scope}
	\node at (0,-.75) {(iii)};
	\end{scope}
	\end{tikzpicture}
	\caption{The graphs that appear in the proof of Proposition \ref{prop:low-degree-graphical}, suppressing the labels on univalent vertices.}	
	\label{fig:low-deg-graphs}
\end{figure}

 As in the proof of Theorem \ref{thm:BigRingStr}, there is a map
\[\bar{\cG}(S,\cP)^{(1)} \lra \left[(\bQ \oplus Y[n]) \otimes S^*(H(g)[-n] \otimes \cP) \otimes H(g)^{\otimes S} \right]^{G'_g}\]
which is an isomorphism in a range of degree increasing with $g$. 

Sending $p_i$ for $|p_i|>4n$ to $0$ gives a map $\cP \to \cV$, which induces the left vertical map in the commutative diagram
\[\begin{tikzcd} \bar{\cG}(S,\cP)^{(1)} \rar \dar & \left[(\bQ \oplus Y[n]) \otimes S^*(H(g)[-n] \otimes \cP) \otimes H(g)^{\otimes S} \right]^{G'_g} \dar \\
\cG(S,\cV) \rar & \left[R^\cV \otimes H(g)^{\otimes S}\right]^{G'_g}.\end{tikzcd}\]
In the proof of Theorem \ref{thm:BigRingStr}, we identified the the left-bottom corner with the vector space $\cP(S,\cV)_{\geq 0} \otimes (\det \bQ^S)^{\otimes n}$ of partitions of $S$ with parts labelled by elements of $\cV$, subject to certain conditions on the degrees of allowed labels. In the range $*<2n$, any labelled partition $(\{P_i\},\{c_i\})$ of degree $\sum n(|P_i|-2)+|c_i|$ is a disjoint union of the following indiscrete labelled partitions:
	\begin{enumerate}[\indent (i')]
		\item parts of size $0$ with label of degree $>2n$, 
		\item parts of size $1$ with label of degree $\geq n$,
		\item parts of size $2$ with label of degree $\geq 0$,
		\item parts of size $3$ with label of degree $\geq 0$.
	\end{enumerate}
The passage to the further quotient
\[\left[R^\cV \otimes H(g)^{\otimes S}\right]^{G'_g} \lra \left[\frac{R^\cV}{(\kappa_{\cL_i} \mid 4i-2n > 0)}\otimes H(g)^{\otimes S}\right]^{G'_g}\]
imposes the relation that a part of size $0$ with label $\cL_i$ is $0$.

	The map $ \bar{\cG}(S,\cP)^{(1)}  \to \cG(S,\cV)  \overset{\sim}{\lra} \cP(S,\cV)_{\geq 0} \otimes (\det \bQ^S)^{\otimes n}$ sends a graph to the partition of $S$ induced by the connected components of the graph, each with label given by the product of the labels in $\cV$ of its legs. In the range $*<2n$ and for $g$ sufficiently large. this map provides a bijective correspondence between connected components and indiscrete partitions as long as we set parts of size $0$ with label $\cL_i$ to $0$:
	\begin{itemize}
		\item The parts of type (i') arise as follows: those with label $p_ip_j$ come from graphs of type (i) with labels $p_i,p_j \in \cP$, those with label $p_ip_jp_k$ come from graphs of type (ii) with labels $p_i,p_j,p_k \in \cP$, and those with label $ep_i$ come graphs of type (iii) with label $p_i \in \cP$. Because for $2n<4i<4n$ the monomial $p_i$ has non-zero coefficient in $\cL_i$, these are all non-zero parts of type (i') in the range $*<2n$.
		\item A part of type (ii') comes from a graph of type (i) if its label is $p_i$, from a graph of type (ii) if its label is $p_ip_j$, and from a graph of type (iii) if its label is $e$.
		\item A part of type (iii') comes from either from a graph of type (i) with both labels in $S$, or a graph of type (ii) with two labels in $S$ and one in $\cP$.
		\item A part of type (iv') comes from a graph of type (iii) with all labels in $S$.
	\end{itemize}
	
	In degrees $*<2n$, a graph can contain at most a single connected component of type (ii) or (iii) and a partition can contain at most one part corresponding to such a connected component. Hence this bijective correspondence between connected components and indiscrete partitions gives rise to one between graphs and partitions.
\end{proof}

\begin{remark}
This computation is related to work of Berglund and Madsen on block diffeomorphisms \cite{berglundmadsen2}. Let $\smash{\widetilde{\mr{Diff}}}(W_{g}, D^{2n})$ denote the simplicial group of block diffeomorphisms of $W_{g}$ fixing $D^{2n} \subset W_g$ pointwise, which can be identified with block diffeomorphisms of $W_{g,1} \coloneqq W_g \setminus \mathrm{int}(D^{2n})$ fixing $\partial W_{g,1}$ pointwise. This has a map to the path components of the homotopy automorphisms of $W_{g,1}$ fixing $\partial W_{g,1}$ pointwise, whose kernel we shall denote by $\smash{\widetilde{\mr{Diff}}}_J(W_{g}, D^{2n})$.

The action of a homotopy automorphism of $W_{g,1}$ on $H(g)$ preserves both the intersection form and its quadratic refinement, so there is a further map $\pi_0(\mr{haut}_\partial(W_{g,1})) \to G'_g$. Berglund and Madsen prove that the action of $\pi_0(\mr{haut}_\partial(W_{g,1}))$ on $H^*(B\smash{\widetilde{\mr{Diff}}}_J(W_{g}, D^{2n});\bQ)$ factors over $G'_g$. Since the map $\pi_0(\mr{haut}_\partial(W_{g,1})) \to G'_g$ has finite kernel, it follows from a Serre spectral argument that the inclusion induces an isomorphism of $G'_g$-representations 
	\[H^*(B\widetilde{\mr{Tor}}(W_{g}, D^{2n});\bQ) \overset{\cong}\lra H^*(B\widetilde{\mr{Diff}}_J(W_{g}, D^{2n});\bQ)\]
	if we let $\widetilde{\mr{Tor}}(W_{g}, D^{2n})$ be the subgroup of the block diffeomorphisms of those components that map to the identity in $G'_g$. Furthermore, Berglund and Madsen prove there is an isomorphism of $G'_g$-representations
	\[H^*(B\widetilde{\mr{Diff}}_J(W_{g}, D^{2n});\bQ) \cong H^*_\mr{CE}(\mathfrak{g}_g) \otimes S^*(H(g) \otimes P),\]
	where $H^*_\mr{CE}(\mathfrak{g}_g)$ denotes Chevalley--Eilenberg cohomology of a certain graded Lie algebra $\mathfrak{g}_g$ and $P$ is $(\pi_*(G/O) \otimes \bQ)^\vee[-n]_{>0}$, which can be identified with $(\pi_*(BO) \otimes \bQ)^\vee[-n]_{>0}$ using the rational homotopy equivalence $G/O \to BO$. As $H^*_\mr{CE}(\mathfrak{g}_g) \otimes S^*(H(g) \otimes P)$ is an algebraic representation of $G'_g$, the map
	\[H^*(B\widetilde{\mr{Tor}}(W_{g}, D^{2n});\bQ) \lra H^*(B\mr{Tor}(W_{g}, D^{2n});\bQ)\]
	factors over $H^*(B\mr{Tor}(W_{g}, D^{2n});\bQ)^\mr{alg}$.
	
	Since we can define twisted Miller--Morita--Mumford classes on block bundles, cf.~Remark \ref{rem:bundle-types}, the homomorphism
	\begin{equation}\label{eqn:block-diff-alg} H^*(B\widetilde{\mr{Tor}}(W_{g}, D^{2n});\bQ) \lra H^*(B\mr{Tor}(W_{g}, D^{2n});\bQ)^\mr{alg}\end{equation}
	is surjective for $g$ sufficiently large. In degrees $*<2n$, the groups $H^*_\mr{CE}(\mathfrak{g}_g)$ are concentrated in total degrees $0$ and $n$, and given by $\bQ$ and $Y[n]$ respectively. Using Proposition \ref{prop:low-degree-graphical}, we see that in this range the map \eqref{eqn:block-diff-alg} is a surjection between vector spaces of the same dimension and hence an isomorphism.\end{remark}

\section{Additive structure}\label{sec:Char}

Given Theorem \ref{thm:MainCalc} it is reasonable to ask for an explicit description of the multiplicities in $H^*(B\mr{Tor}(W_{g}, D^{2n}); \bQ)$ of the various irreducible algebraic $G'_g$-representations, which by Theorem \ref{thm:RepsOfSpAndO2} are the $V_\lambda(H(g))$'s. This can be reduced to a manipulation of Schur functions: by Theorem \ref{thm:MainCalc} and the final part of Proposition \ref{prop:Recognition}, in the stable range the multiplicity of $V_\lambda(H(g))$ in $H^*(B\mr{Tor}(W_{g}, D^{2n}); \bQ)$ is the same as the multiplicity of $S^\lambda$ in the $\Sigma_q$-representation
\[\bQ \otimes_{\bQ[\kappa_{\cL_i} \, \mid \, 4i-2n > 0]} \mathcal{P}([q], \cV)_{\geq 0}' \otimes (\det \bQ^{[q]})^{\otimes n},\]
and this can be analysed quite effectively with the theory of symmetric functions.

\subsection{Recollection on symmetric functions}
We follow the exposition of Garoufalidis--Getzler \cite[Section 2]{GG}. Let $\Lambda$ denote the ring of symmetric functions, the inverse limit $\lim_k \bZ[x_1, \ldots, x_k]^{\Sigma_k}$ formed in the category of graded rings where the $x_i$ are placed in grading 1. Write $\Lambda_q$ for the piece of grading $q$. Let $\hat{\Lambda} = \prod_q \Lambda_q$ denote the completion of $\Lambda$ with respect to the filtration induced by this grading. As usual denote by $e_k$ the $k$th elementary symmetric function, by $h_k$ the $k$th complete symmetric function, and by $p_k$ the $k$th power sum function. For example, $e_2 = \sum_{i<j} x_i x_j$, $h_2 = \sum_{i \leq j} x_i x_j$ and $p_2 = \sum_i x_i^2$. Both the $e_k$ and the $h_k$ provide a set of polynomial generators for $\Lambda$, and the $p_k$ form a set of polynomial generators for $\Lambda \otimes \bQ$.

\subsubsection{Symmetric groups}

For a group $G$, let $R(G)$ denote the group-completion of the monoid of isomorphism classes of finite-dimensional $G$-representations under direct sum. Similarly, let $R(\cat{FB})$ denote the group-completion of the monoid of isomorphism classes of objects of $(\bQ\text{-mod}^d)^\cat{FB}$, i.e.~representations of the category $\cat{FB}$ into finite-dimensional (which is the same as dualisable) vector spaces, under objectwise direct-sum. This has the structure of a commutative ring given by Day convolution of functors. There are restriction maps $R(\cat{FB}) \to R(\Sigma_q)$ for each $q$, and taking them all together gives an isomorphism
\[R(\cat{FB}) \lra \prod_{q \geq 0} R(\Sigma_q).\]
The preimage of $\bigoplus_{q \geq 0} R(\Sigma_q)$ under this map consists of (differences of) finite length representations of $\cat{FB}$. As a Day convolution of finite length functors again has finite length, there is an induced multiplication.

There are homomorphisms of abelian groups
\begin{align*}\mr{ch}_q \colon R(\Sigma_q) &\lra \Lambda_q \\
V &\longmapsto \mr{ch}_q(V)\coloneqq\sum_{|\lambda|=q} \chi_V(\cO_\lambda) \frac{p_{\lambda_1} \cdots p_{\lambda_\ell}}{1^{\lambda_1} \lambda_1! 2^{\lambda_2} \lambda_2! \cdots \ell^{\lambda_\ell} \lambda_\ell!},\end{align*}
where $\chi_V(\cO_\lambda)$ is the value of the character of $V$ on the conjugacy class $\cO_\lambda$ of cycle type $\lambda$. For example, given a partition $\lambda$ of $q$ the irreducible representation $S^\lambda$ of $\Sigma_q$ is sent by $\mr{ch}_q$ to the Schur function $s_\lambda$. In particular, the trivial representation is sent to $h_q$ and the sign representation $\det$ to $e_q$.

These homomorphisms are in fact isomorphisms and combine to give a ring isomorphism \[\mr{ch} \colon \bigoplus_q R(\Sigma_q) \lra \Lambda\] when the domain is given the product by Day convolution. Similarly, they combine to give a ring isomorphism \[\mr{ch} \colon R(\cat{FB}) =\prod_{q \geq 0} R(\Sigma_q) \lra \hat{\Lambda}.\]
As the $S^\lambda$ give a $\bZ$-basis for $\bigoplus_q R(\Sigma_q)$, the $s_\lambda$ give a $\bZ$-basis for $\Lambda$.

More generally, if $gR(\cat{FB})$ denotes the group-completion of the monoid of isomorphism classes of objects of $\cat{Gr}(\bQ\text{-mod}^d)^\cat{FB}$, i.e.~representations of the category $\cat{FB}$ into non-negatively graded vector spaces which are finite-dimensional in each degree, then we have a ring isomorphism $gR(\cat{FB}) \cong R(\cat{FB})[[t]]$ by extracting homogeneous pieces. This gives an isomorphism $\mr{ch} \colon gR(\cat{FB}) \overset{\sim}\to \hat{\Lambda}[[t]]$.

The category $\cat{Gr}(\bQ\text{-mod}^d)^\cat{FB}$ has another monoidal structure, the composition product $\circ$, given by
\[(F \circ G)(q) = \bigoplus_{n=0}^\infty F(n) \otimes_{\Sigma_n} \left(\bigoplus_{k_1, \ldots, k_n \geq 0,\sum k_i = q} \mathrm{Ind}^{\Sigma_q}_{\Sigma_{k_1} \times \cdots \times \Sigma_{k_n}} G(k_1) \otimes \ldots \otimes G(k_n)\right).\]
This construction is formed in the symmetric monoidal category $\cat{Gr}(\bQ\text{-mod}^d)$, whose symmetry includes a sign given by the Koszul sign rule. Under the isomorphism above, this induces an associative product $\circ$ on $\hat{\Lambda}[[t]]$. On $\hat{\Lambda}$ this is given by plethysm of symmetric functions, and its extension to $\smash{\hat{\Lambda}}[[t]]$ is characterised by $p_k \circ x = x^k$ for all $x$, and $t$-linearity in the first variable.

\subsubsection{An involution}

There is an involution $\omega \colon \Lambda \to \Lambda$ given by $\omega(e_k) = h_k$. It is easy to see that this satisfies $\omega(p_k) = (-1)^{k-1} p_k$ for all $k$, and hence that under the isomorphisms $\mr{ch}_q$ it corresponds to tensoring with the sign representation of $\Sigma_q$.

\subsubsection{Representations of $G'_g$} 
Recall that $H(g)$ denotes $2g$-dimensional rational vector space equipped with an $\epsilon$-symmetric form $\lambda \colon H(g) \otimes H(g) \to \bQ$, for $\epsilon \in \{-1,1\}$, and $\mr{O}_\epsilon(H(g)) \subset \mr{GL}(H(g))$ denotes the subgroup of those linear isomorphisms which preserve $\lambda$. In the branching rule for $\mr{O}_\epsilon(H(g)) \subset \mr{GL}(H(g))$, the irreducible $\mr{GL}(H(g))$-representation $S_\lambda(H(g))$ restricted to $\mr{O}_\epsilon(H(g))$ decomposes as
\[\mr{Res}^{\mr{GL}(H(g))}_{\mr{O}_\epsilon(H(g))}(S_\lambda(H(g))) \cong V_\lambda(H(g)) \oplus \bigoplus_{\substack{\mu \\ |\mu| < |\lambda|}} a_{\lambda, \mu} V_\lambda(H(g))\]
for certain multiplicities $a_{\lambda, \mu}$ (and which may be given in terms of Littlewood--Richardson coefficients). We may recursively define elements $s_{\langle \lambda \rangle}$ of $\Lambda$ by
\[s_{\langle \lambda \rangle} \coloneqq s_{\lambda } - \sum_{\substack{\mu \\ |\mu| < |\lambda|}} a_{\lambda, \mu} s_{\langle \mu \rangle}.\]
By the upper-triangularity of this definition, the $s_{\langle \lambda \rangle}$ also form a $\bZ$-basis for $\Lambda$, and there is therefore an automorphism of abelian groups 
\begin{align*}
D \colon \Lambda &\lra \Lambda\\
s_\lambda & \longmapsto s_{\langle \lambda \rangle}.
\end{align*}

There are ring homomorphisms $\Lambda \to R(G'_g)$ given by sending $e_k$ to $\Lambda^k(H(g))$, which therefore send $s_{\langle\lambda\rangle}$ to $V_\lambda(H(g))$.

\subsection{Evaluating the character}\label{sec:EvChar}

We can obtain the Poincar{\'e} series in $\Lambda[[t]]$ (and hence in $R(G'_g)[[t]]$) of the graded $G'_g$-representation $H^*(B\mr{Tor}(W_{g}, D^{2n}); \bQ)^\mr{alg}$ as $g \to \infty$ as follows. By interpreting the calculation in Theorem \ref{thm:MainCalc} using the last part of Proposition \ref{prop:Recognition} we see that this Poincar{\'e} series is given by applying $D$ to the character of
\[\bigoplus_{q \geq 0} \left(\bQ \otimes_{\bQ[\kappa_{\cL_i} \, \mid \, 4i-2n >0]} (\mathcal{P}([q], \cV)_{\geq 0}' \otimes (\det \bQ^q)^{\otimes n}) \right) \in \bigoplus_{q \geq 0} R(\Sigma_q).\]
Using the fact that $\mr{ch}$ is a ring homomorphism and sends the operation of tensoring with $\det$ to the involution $\omega$, we see that the Poincar{\'e} series is obtained by applying $D$ to 
\begin{equation}\label{eq:ev-char}
\omega^n \left(\sum_{q \geq 0} \mr{ch}_q\left(\bQ \otimes_{\bQ[\kappa_{\cL_i} \, \mid \, 4i-2n >0]}\mathcal{P}([q], \cV)_{\geq 0}' \right)\right) \in \Lambda[[t]].\end{equation}

To evaluate this, note that $\mathcal{P}([q], \cV)_{\geq 0}'$ is a free $\cP(\varnothing, \cV)_{\geq 0} = \bQ[\kappa_c \, \mid \, c \in \cB_{>2n}]$-module, and so by the proof of Theorem \ref{thm:MainCalc} it is a free $\bQ[\kappa_{\cL_i} \, \mid \, 4i-2n > 0]$-module (we already used this observation in the proof of Proposition \ref{prop:LsAreRegular}). Thus we have
\[\sum_{q \geq 0}\mr{ch}_q\left(\bQ \otimes_{\bQ[\kappa_{\cL_i} \, \mid \, 4i-2n >0]}\mathcal{P}([q], \cV)_{\geq 0}' \right) = \left(\prod_{4i>2n} (1-t^{4i-2n}) \right)\cdot\sum_{q \geq 0}\mr{ch}_q\left(\mathcal{P}([q], \cV)_{\geq 0}' \right).\]

To understand the second factor, we observe that the graded representation
\[\bigoplus_{q\geq 0} \mathcal{P}([q], \cV)_{\geq 0}' \in \left(\bigoplus_{q \geq0} R(\Sigma_q)\right)[[t]]\]
may be expressed, in the larger ring $R(\cat{FB})[[t]]$, in terms of a composition product $\underline{\bQ} \circ B$, where $\underline{\bQ}$ and $B$ are as follows:
\begin{enumerate}[\indent (i)]
\item $\underline{\bQ}$ denotes the graded representation whose $q$th component of is the trivial 1-dimensional representation (in degree 0) for all $q$.

\item $B$ denotes the graded representation whose $q$th component is the trivial $\Sigma_q$-representation with basis the set of allowed labels in $\cB$ for parts of size $q$, where a label $c$ is given degree $|c| + n(q-2)$. A labelling of a partition of the finite set $[q]$ by elements of $\cB$ is allowed here if each part of size 0 has label of degree $>2n$, each part of size $1$ has label of degree $\geq n$, and no parts of size $2$ are labelled by $1 \in \cB$. That is, $B(q)$ is the graded vector space with basis $\cB$ if $q>2$ and with smaller basis according to the aforementioned conditions for $q = 0, 1,2$.
\end{enumerate}

Recall that $\mr{ch}_q$ of the trivial representation is $h_q$, so we get that
\[\mr{ch}(\underline{\bQ}) = \sum_{q=0}^\infty h_q \in \hat{\Lambda}[[t]]\]
and (writing $P(V) \in \bZ[[t]]$ for the Poincar{\'e} series of a graded vector space $V$)
\[\mr{ch}(B) = h_0 P(\cV_{>2n}) t^{-2n} + h_1 P(\cV_{\geq n})t^{-n} + h_2 P(\cV_{>0}) + \sum_{q=3}^\infty h_q P(\cV) t^{n(q-2)} \in t\Lambda[[t]].\]
We may easily analyse these Poincar{\'e} series, as $\cV=\bQ[e, p_{\lceil \tfrac{n+1}{4}\rceil}, \ldots, p_{n-1}]$ so we have
\[P(\cV) = \frac{1}{1-t^{2n}} \cdot \prod_{i=\lceil \tfrac{n+1}{4}\rceil}^{n-1} \frac{1}{1-t^{4i}}\]
so we can write 
\begin{align*}
P(\cV_{>2n}) &= P(\cV) - (1+t^{2n} + t^{4\lceil \tfrac{n+1}{4}\rceil} + \cdots + t^{4\lfloor \tfrac{n}{2}\rfloor}),\\
P(\cV_{\geq n}) &= P(\cV)-1 \quad \text{ and } \quad P(\cV_{>0}) = P(\cV)-1.
\end{align*}
Thus we may write $\mr{ch}(B)$ as
\begin{equation*}
\mr{ch}(B) = \frac{1}{t^{2n}}\left(P(\cV)\left(\sum_{q=0}^\infty h_q t^{nq}\right) - (h_0 (1+t^{2n} + t^{4\lceil \tfrac{n+1}{4}\rceil} + \cdots + t^{4\lfloor \tfrac{n}{2}\rfloor})+ h_1 t^{n} + h_2 t^{2n})\right).
\end{equation*}
As the composition product is sent to plethysm by $\mr{ch}$, we have
\[\sum_{q \geq 0} \mr{ch}_q(\mathcal{P}([q], \cV)_{\geq 0}' ) = \left(\sum_{q=0}^\infty h_q\right) \circ \mr{ch}(B).\]
(Note that as $h_q \circ -$ sends $t\Lambda[[t]]$ into $t^q\Lambda[[t]]$, and $\mr{ch}(B) \in t\Lambda[[t]]$, this plethysm does actually land in $\Lambda[[t]]$.)

\subsection{Example: dimension 6} \label{sec:6-dim-example}As an example consider the case $2n=6$, and compute the character of $H^*(B\mr{Tor}(W_{g}, D^6);\bQ)^\mr{alg}$ for $g \gg 0$ by evaluating \eqref{eq:ev-char} and applying $D$. In this case, we have 
	\[\cV = H^*(B\mr{SO}(6)\langle 3 \rangle;\bQ) = \bQ[p_1, p_2, e],\]
and so $P(\cV) = \frac{1}{1-t^6} \cdot \frac{1}{1-t^4} \cdot \frac{1}{1-t^8} = 1+t^4+t^6+2t^8+t^{10}+3t^{12}+O(t^{14})$. Thus we have
\begin{align*}
\mr{ch}(B) &= \frac{1}{t^6}\left(\frac{1}{1-t^6} \cdot \frac{1}{1-t^4} \cdot \frac{1}{1-t^8}\left(\sum_{q=0}^\infty h_q t^{3q}\right) - (h_0(1+t^6+t^4) + h_1 t^3 + h_2 t^6)\right) \\
&= h_1t+2h_0t^2+(h_3+h_1)t^3+(h_2+h_0)t^4 + O(t^5).
\end{align*}
(The following calculations were performed in {\tt Sage} \cite{Sage}.) Applying $\sum_{q=0}^\infty h_q \circ -$ to this, then expressing the answer in terms of $s_\lambda$'s gives
\[1+s_1t + (2+s_2)t^2 + (3s_1+2 s_3)t^3 + (4+s_{1^2}+4s_2+s_{3,1}+2 s_4)t^4 + O(t^5).\]
Applying $\omega$ and then $D$ sends $s_\lambda$ to $s_{\langle \lambda' \rangle}$, so transforms this to
\[1+s_{\langle 1 \rangle} t + (2+s_{\langle 1^2 \rangle})t^2 + (3s_{\langle 1 \rangle}+2 s_{\langle 1^3 \rangle})t^3 + (4+s_{\langle 2 \rangle}+4s_{\langle 1^2 \rangle}+s_{\langle 2, 1^2 \rangle}+2 s_{\langle 1^4 \rangle})t^4 + O(t^5).\]
Multiplying by $\prod_{4i>6} (1-t^{4i-2n}) = 1 - t^2 + O(t^5)$ gives the result,
\[1+s_{\langle 1 \rangle} t + (1+s_{\langle 1^2 \rangle})t^2 + (2s_{\langle 1 \rangle}+2 s_{\langle 1^3 \rangle})t^3 + (2+s_{\langle 2 \rangle}+3s_{\langle 1^2 \rangle}+s_{\langle 2, 1^2 \rangle}+2 s_{\langle 1^4 \rangle})t^4 + O(t^5),\]
so for $2n=6$ and large enough $g$ we can read off
\begin{align*}
H^1(B\mr{Tor}(W_{g}, D^6); \bQ)^\mr{alg} &\cong V_1,\\
H^2(B\mr{Tor}(W_{g}, D^6); \bQ)^\mr{alg} &\cong V_{1^2} + V_0,\\
H^3(B\mr{Tor}(W_{g}, D^6); \bQ)^\mr{alg} &\cong 2 V_{1^3} + 2 V_1,\\
H^4(B\mr{Tor}(W_{g}, D^6); \bQ)^\mr{alg} &\cong 2 V_{1^4} + V_{2, 1^2} + 3 V_{1^2} + V_2 + 2 V_0.
\end{align*}

\section{Variants}\label{sec:variants}

There are two close variants of $\diff(W_{g}, D^{2n})$, namely the group $\diff^+(W_g, *)$ of those orientation-preserving diffeomorphisms of $W_g$ which preserve a point $* \in W_g$, and the group $\diff^+(W_g)$ of all orientation-preserving diffeomorphisms. Each of these has its associated Torelli subgroup, denoted in the evident way, and we will briefly explain how the cohomology of $B\mr{Tor}^+(W_g, *)$ and $B\mr{Tor}^+(W_g)$ may be deduced from our previous calculations.

Firstly, there is a fibration sequence
\begin{equation}\label{eq:Fib1}
B\mr{Tor}(W_{g}, D^{2n}) \lra B\mr{Tor}^+(W_g, *) \lra B \mathrm{GL}_{2n}(\bR) \simeq B\mr{SO}(2n)
\end{equation}
where the right-hand map is given by taking the derivative at the marked point. This is a fibration of spaces with $G'_g$-action, giving an induced action on rational cohomology. The statement of the following result is best understood by consulting its proof.

\begin{lemma}\label{lem:Fib1}
The fibration \eqref{eq:Fib1} satisfies the Leray--Hirsch property on maximal algebraic subrepresentations in the stable range.
\end{lemma}
\begin{proof}
Consider the Serre spectral sequence $\{E_r^{p,q}\}$ for the fibration sequence
\[B\mr{Diff}(W_{g}, D^{2n}) \lra B\mr{Diff}^+(W_g, *) \lra B \mathrm{GL}_{2n}(\bR) \simeq B\mr{SO}(2n)\]
with $\cH(g)_\bQ^{\otimes S}$-coefficients. By Theorem \ref{thm:Iso} $H^*(B\mr{Diff}(W_{g}, D^{2n}); \cH(g)_\bQ^{\otimes S})$ is generated by twisted Miller--Morita--Mumford classes, and by construction these are defined in $H^*(B\mr{Diff}^+(W_{g}, *); \cH(g)_\bQ^{\otimes S})$, so this spectral sequence satisfies the Leray--Hirsch property and collapses at $E_2$. This gives an isomorphism
\begin{equation}\label{eq:RelPoint}
H^*(B\mr{Diff}^+(W_{g}, *); \cH(g)_\bQ^{\otimes S}) \cong H^*(B\mr{SO}(2n);\bQ) \otimes \cP(S, \cV)_{\geq 0} \otimes (\det \bQ^S)^{\otimes n}
\end{equation}
of $H^*(B\mr{SO}(2n);\bQ)$-modules.

If the cohomology of $B\mr{Tor}(W_{g}, D^{2n})$ is finite-dimensional in degrees $* < N$ then by the Serre spectral sequence for \eqref{eq:Fib1} that of $B\mr{Tor}^+(W_{g}, *)$ is too. Repeating the argument of Theorem \ref{thm:MainCalc} with the input \eqref{eq:RelPoint} shows that there is a map
\[H^*(B\mr{SO}(2n);\bQ) \otimes \frac{i^*(K^\vee) \otimes^{\mathsf{d(s)Br}} \left(\mathcal{P}(-, \cV)_{\geq 0}'  \otimes {\det}^{\otimes n}\right)}{(\kappa_{\cL_i} \, \mid \, 4i-2n > 0)} \lra H^*(B\mr{Tor}^+(W_{g}, *);\bQ)^\mr{alg}\]
of $H^*(B\mr{SO}(2n);\bQ)$-modules which is an isomorphism in degrees $* \leq N$ and a monomorphism in degree $N+1$.
\end{proof}

Secondly, there is a fibration sequence
\begin{equation}\label{eq:Fib2}
W_g \lra B\mr{Tor}^+(W_g, *) \overset{\pi}\lra B\mr{Tor}^+(W_g),
\end{equation}
which may be identified with the universal $W_g$-bundle over $B\mr{Tor}^+(W_g)$.

\begin{lemma}\label{lem:Fib2}
The fibration \eqref{eq:Fib2} satisfies the Leray--Hirsch property, as long as $n$ is even or $g \neq 1$.
\end{lemma}
\begin{proof}
This spectral sequence has three rows, the 0th, $n$th, and $2n$th. The fundamental group of $B\mr{Tor}^+(W_g)$ acts trivially on the cohomology of the fibre $W_g$, by definition of the Torelli group, so this spectral sequence has a product structure. To show that the Leray--Hirsch property is satisfied we must show that it collapses at the $E_2$-page. The Euler class $e(T_\pi)$ of the vertical tangent bundle of this fibre bundle restricts to a non-zero class in $H^{2n}(W_g;\bQ)$ under the stated conditions, meaning that there can be no differentials out of the $2n$th row. On the other hand, we have
\[\pi_!(e(T_\pi) \cdot \pi^*(x)) = \chi(W_g) \cdot x\]
showing that $\pi^*$ is injective under the stated conditions, meaning that there can be no differentials into the 0th row.
\end{proof}

Combining these two results with the method described in Section \ref{sec:EvChar}, one can extract the Poincar{\'e} series in $\Lambda[[t]]$ of $H^*(B\mr{Tor}^+(W_g, *);\bQ)^\mr{alg}$ or $H^*(B\mr{Tor}^+(W_g);\bQ)^\mr{alg}$, in the stable range. Describing these as rings seems to be an interesting problem: the ring $H^*(B\mr{Tor}^+(W_g, *);\bQ)^\mr{alg}$ can be addressed with the methods of this paper, but the ring $H^*(B\mr{Tor}^+(W_g);\bQ)^\mr{alg}$ seems to be difficult to describe well.

\section{Discussion of the case $2n=2$}\label{sec:Dim2} 

Above we gave two techniques to make our computation of the cohomology of Torelli spaces more explicit: Section \ref{sec:GraphicalCalc} gives a presentation of cohomology ring and Section \ref{sec:Char} tells us how to compute the characters of the cohomology groups. We shall now apply both to the case $2n=2$.

\subsection{Additive structure}
Johnson has shown \cite{JohnsonAb} that $H^1(B\mr{Tor}_\partial(W_{g}, D^2); \bQ) \cong \Lambda^3(H(g))$ for $g \geq 3$, which is finite dimensional, so Theorem \ref{thm:MainCalc} gives an isomorphism in degrees $* \leq 2$ and a monomorphism in degree $*=3$. We may therefore use this result to calculate $H^2(B\mr{Tor}(W_{g}, D^2); \bQ)^\mr{alg}$ as a $\mathrm{Sp}_{2g}(\bZ)$-representation for $g \gg 0$, and to estimate $H^3(B\mr{Tor}(W_{g}, D^2); \bQ)^\mr{alg}$ from below.

\begin{theorem}\label{thm:Calc-2d}
For $g \gg 0$ we have
\begin{equation*}
H^2(B\mr{Tor}(W_{g}, D^2); \bQ)^\mr{alg} \cong  2 V_{1^2} + V_{2,1^2}+ 2V_{1^4}  + V_{2^2,1^2} + V_{1^6}.
\end{equation*}
For $g \gg 0$ we have
\begin{align*}
H^3(B\mr{Tor}(W_{g}, D^2); \bQ)^\mr{alg} &\geq  V_1 + V_{2,1} + 3V_{1^3} + 2V_{2^2, 1}+3V_{2,1^3}+V_{3,2,1^2} \\
&\quad \quad + 2V_{2^3, 1}+ V_{3, 2^3} + 4V_{1^5} + 2V_{2^2,1^3}+V_{3^2, 1^3}\\
& \quad\quad\quad   +2V_{2,1^5}+V_{2^3, 1^3}+ 2V_{1^7} + V_{2^2,1^5} + V_{1^9},
\end{align*}
with equality if $H^2(B\mr{Tor}(W_{g}, D^2); \bQ)$ is finite-dimensional for $g \gg 0$.
\end{theorem}
\begin{proof}
We use the method described in Section \ref{sec:EvChar}, in the case $2n=2$, and rely on the notation from that section. In this case we have
\[\cV = H^*(B\mr{SO}(2n);\bQ) = \bQ[e]\]
and so $P(\cV) = \frac{1}{1-t^2} = 1 + t^2  + O(t^4)$. Thus we have
\begin{align*}
\mr{ch}(B) &= \frac{1}{t^2}\left(\frac{1}{1-t^2}(\sum_{q=0}^\infty h_q t^q) - (h_0(1+t^2) + h_1 t + h_2 t^2)\right) \\
&= (h_3+h_1) t+(h_2+h_4+1)t^2+(h_5+ h_3+h_1) t^3 + O(t^4).
\end{align*}
(The following calculations were again performed in {\tt Sage} \cite{Sage}.) Applying $\sum_{q=0}^\infty h_q \circ -$ to this, then expressing the answer in terms of $s_\lambda$'s gives
\begin{align*}
&1 + (s_3+s_1)t + (2s_4+s_{4,2}+s_{3,1}+1+2s_2+s_6) t^2\\
&\quad + (s_{9}+3s_{4,1}+s_{4,2,1}+s_{7,2}+2s_{3,2}+4s_3+2s_7+2s_1+2s_{5,2}\\
&\quad\quad +s_{6,3}+4s_5+s_{5,2,2}+2s_{6,1}+s_{2,1}+2s_{4,3}+s_{4,4,1})t^3 + O(t^4).
\end{align*}
Applying $\omega$ and then $D$ sends $s_\lambda$ to $s_{\langle \lambda' \rangle}$ so transforms this to
\begin{align*}
&1 + (s_{\langle 1^3 \rangle}+s_{\langle 1\rangle})t + (2s_{\langle 1^4 \rangle}+s_{\langle 2^2, 1^2\rangle}+s_{\langle 2, 1^2 \rangle}+1+2 s_{\langle 1^2 \rangle}+s_{\langle 1^6\rangle}) t^2\\
&\quad + (s_{\langle 1^9\rangle}+3s_{\langle 2, 1^3\rangle}+s_{\langle 3,2,1^2\rangle}+s_{\langle 2^2, 1^5\rangle}+2s_{\langle 2^2, 1\rangle}+4s_{\langle 1^3\rangle}+2s_{\langle 1^7\rangle}+2s_{\langle 1\rangle}+2s_{\langle 2^2, 1^3\rangle}\\
&\quad\quad +s_{\langle 2^3, 1^3\rangle} +4s_{\langle 1^5 \rangle}+s_{\langle 3^2, 1^3 \rangle}+2s_{\langle 2, 1^5\rangle}+s_{\langle 2,1\rangle}+2s_{\langle 2^3, 1\rangle}+s_{\langle 3, 2^3\rangle})t^3 + O(t^4).
\end{align*}
Multiplying by $\prod_{4i>2}(1-t^{4i-2n}) = 1-t^2 + O(t^4)$ gives the result,
\begin{align*}
&1+ (s_{\langle 1^3\rangle}+s_{\langle 1\rangle})t +(2s_{\langle 1^4\rangle}+s_{\langle 2^2, 1^2\rangle}+s_{\langle 2,1^2\rangle}+2s_{\langle 1^2\rangle}+s_{\langle 1^6\rangle})t^2\\
&\quad + (s_{\langle 1^9\rangle}+3s_{\langle 2,1^3\rangle}+s_{\langle 3,2,1^2\rangle}+s_{\langle 2^2, 1^5\rangle}+2s_{\langle 2^2, 1\rangle}+3s_{\langle 1^3\rangle}+2s_{\langle 1^7\rangle}+s_{\langle 1\rangle}+2s_{\langle 2^2, 1^3\rangle}\\
&\quad\quad +s_{\langle 2^3, 1^3\rangle}+4s_{\langle 1^5\rangle}+s_{\langle 3^2, 1^3\rangle}+2s_{\langle 2, 1^5\rangle}+s_{\langle 2,1\rangle}+2s_{\langle 2^3, 1\rangle}+s_{\langle 3, 2^3\rangle})t^3 + O(t^4).
\end{align*}

Extracting the coefficient of $t$ we obtain
\[H^1(B\mr{Tor}_\partial(W_{g}, D^2); \bQ)^\mr{alg} \cong V_1 + V_{1^3} \cong \Lambda^3(H(g)),\]
compatible with Johnson's theorem. Extracting the coefficients of $t^2$ and $t^3$ gives the two claimed calculations.
\end{proof}

By Lemma \ref{lem:Fib1}, in the stable range the Poincar{\'e} series for $B\mathrm{Tor}^+(W_g, *)$ is obtained by multiplying that for $B\mathrm{Tor}(W_{g},D^2)$ by the Poincar{\'e} series for $B\mr{SO}(2)$, namely $\frac{1}{1-t^2} = 1 + t^2 + O(t^4)$, so it is
\begin{align*}
&1+(s_{\langle 1^3 \rangle}+s_{\langle 1 \rangle})t+(1+s_{\langle 1^6 \rangle}+s_{\langle 2^2, 1^2 \rangle}+2s_{\langle 1^4 \rangle}+s_{\langle 2, 1^2 \rangle}+2s_{\langle 1^2 \rangle})t^2\\
&\,\,\, +(s_{\langle 1^9 \rangle}+s_{\langle 2^2, 1^5 \rangle}+2s_{\langle 1^7 \rangle}+2s_{\langle 2^3, 1^3 \rangle}+2s_{\langle 2, 1^5 \rangle}+s_{\langle 3^2, 1^3 \rangle}+2s_{\langle 2^2, 1^3 \rangle}+4s_{\langle 1^5 \rangle}\\
&\,\,\,\,\,\,\,\,+s_{\langle 3, 2^3 \rangle}+2s_{\langle 2^3, 1 \rangle}+s_{\langle 3,2,1^2 \rangle}+3s_{\langle 2, 1^3 \rangle}+2s_{\langle 2^2, 1 \rangle}+4s_{\langle 1^3 \rangle}+s_{\langle 2,1 \rangle}+2s_{\langle 1 \rangle})t^3 + O(t^4).
\end{align*}
Considering the proof of Lemma \ref{lem:Fib1} carefully, it follows that $H^3(B\mathrm{Tor}^+(W_g, *);\bQ)^\mr{alg}$ \emph{contains} the indicated $\mr{Sp}_{2g}(\bZ)$-representation.

By Lemma \ref{lem:Fib2}, in the stable range the Poincar{\'e} series for $B\mathrm{Tor}^+(W_g)$ is obtained by dividing that for $B\mathrm{Tor}^+(W_{g}, *)$ by the Poincar{\'e} series for $W_g$, namely $1+ s_{\langle 1 \rangle}t + t^2$. The inverse of this series is $1-s_{\langle 1\rangle}t+(s_{\langle 1^2\rangle} + s_{\langle 2 \rangle})t^2- (s_{\langle 1 \rangle}+s_{\langle 1^3 \rangle} + 2s_{\langle 2, 1 \rangle} + s_{\langle 3 \rangle}) t^3 + O(t^4)$, so in the stable range the Poincar\'e series for  $B\mathrm{Tor}^+(W_g)$ is 
\begin{align*}
&1 + s_{\langle 1^3 \rangle}t + (s_{\langle 1^2 \rangle}+s_{\langle 1^4 \rangle}+s_{\langle 1^6 \rangle}+s_{\langle 2^2, 1^2 \rangle})t^2\\
&\,\,\, + (s_{\langle 1^3\rangle} + s_{\langle 1\rangle} + 2s_{\langle 1^5 \rangle} + s_{\langle 1^7 \rangle} + s_{\langle 1^9\rangle} + s_{\langle 2,1^3 \rangle} + s_{\langle 2, 1^5\rangle} + s_{\langle 2^2, 1\rangle}\\
&\,\,\,\,\,\,\,\, + s_{\langle 2^2, 1^3\rangle} + s_{\langle 2^2, 1^5\rangle} + s_{\langle 2^3, 1\rangle} + 2 s_{\langle 2^3, 1^3\rangle} + s_{\langle 3, 2^3\rangle} + s_{\langle 3^2, 1^3\rangle} )t^3 + O(t^4).
\end{align*}
Considering the proof of Lemma \ref{lem:Fib2}, it follows that $H^3(B\mathrm{Tor}^+(W_g);\bQ)^\mr{alg}$ \emph{contains} the indicated $\mr{Sp}_{2g}(\bZ)$-representation.

It is interesting to compare these results with the literature. The work of Johnson \cite{JohnsonAb} (or our theory) provides a $\mathrm{Sp}_{2g}(\bZ)$-equivariant isomorphism 
\[\tau \colon H_1(B\mathrm{Tor}^+(W_g);\bQ) \lra V_{1^3},\] 
the \emph{Johnson homomorphism}. This provides a $\mathrm{Sp}_{2g}(\bZ)$-equivariant ring homomorphism $\tau^* \colon \Lambda^* V_{1^3} \to H^*(B\mathrm{Tor}^+(W_g);\bQ)$. Hain has shown in \cite{HainTorelli} that its image in degree 2 is precisely $V_{1^2}+V_{1^4}+V_{1^6}+V_{2^2, 1^2}$, and this may be recovered from our calculation above along with the discussion of the ring structure in the following section. Sakasai has shown in \cite{SakasaiThirdTorelli} that its image in degree 3 is either \[V_{1^3} + 2V_{1^5} + V_{1^7} + V_{1^9} + V_{2,1^3} + V_{2, 1^5} + V_{2^2, 1} + V_{2^2, 1^3} + V_{2^2, 1^5} + V_{2^3, 1} + V_{ 2^3, 1^3} + V_{3, 2^3} + V_{3^2, 1^3}\] or the same with $V_1$ added on. Furthermore, he shows that the $V_1$-term is present if and only if
\begin{equation}\label{eq:Sakasai}
\kappa_{e^3} - (2-2g)e^2 \neq 0 \in H^4(B\mathrm{Tor}^+(W_g, *);\bQ).
\end{equation}

\begin{remark}
Sakasai's espression has one fewer copies of $V_{2^3, 1^3}$ than our expression, and in fact the decomposition of $\Lambda^3 (V_{1^3})$ into irreducibles contains a single $V_{2^3, 1^3}$. There is however no contradiction: this simply expresses the fact that the ring $H^*(B\mathrm{Tor}^+(W_g);\bQ)^\mr{alg}$ is not generated by the image of the Johnson homomorphism. 
\end{remark}

Using our results we are able to resolve the ambiguity in Sakasai's result, and hence show that the inequation \eqref{eq:Sakasai} holds. By our graphical interpretation, the image of the composition
\[\Lambda^3(V_{1^3}) \overset{\tau^*}\lra H^3(B\mathrm{Tor}^+(W_g);\bQ) \lra H^3(B\mathrm{Tor}(W_g, D^2);\bQ)\]
after applying $[- \otimes V_1]^{\mr{Sp}_{2g}(\bZ)}$ is the subspace of those elements which can be represented by trivalent graphs with one leg, three internal vertices, and no loops. There is such a graph, displayed in Figure \ref{fig:sakasai-graph}, which gives a map $V_1 \to \Lambda^3(V_{1^3})$ which when composed with the above and contracting all internal edges and loops is seen to be $\pm\kappa_{e^2} \colon V_1 \to  H^3(B\mathrm{Tor}(W_g, D^2);\bQ)$ which we have shown to be non-zero. It follows that $\tau^*(\Lambda^3(V_{1^3}))$ does indeed contain a copy of $V_1$, resolving the ambiguity in Sakasai's result.

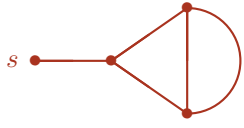
\begin{figure}[h]
	\begin{tikzpicture}
	\node at (0, 0.7) [Mahogany] {$\bullet$};
	\draw [thick,Mahogany] (0, 0.7) -- (0, -0.7);
	\draw [thick,Mahogany] (0, 0.7) -- (0,0);
	\node at (0, -0.7) [Mahogany] {$\bullet$};
	\node at (-1, 0) [Mahogany] {$\bullet$};
	\node at (-2, 0) [Mahogany] {$\bullet$};
	\node at (-2.3, 0) [Mahogany] {$s$};
	\draw [thick,Mahogany] (-2,0) -- (-1,0);
	\draw  [thick,Mahogany] (-2,0) -- (-1.5,0);
	
	\draw [thick,Mahogany] (-1,0) -- (0,0.7);
	\draw  [thick,Mahogany] (-1,0) -- (-0.5, 0.35);
	
	\draw [thick,Mahogany] (-1,0) -- (0,-0.7);
	\draw  [thick,Mahogany] (-1,0) -- (-0.5, -0.35);
	
	\draw [thick,Mahogany] (0,-0.7) arc (-90:90:0.7) ;
	\end{tikzpicture}
	\caption{A trivalent graph with one leg, three internal vertices, and no loops.}
	\label{fig:sakasai-graph}
\end{figure}

One of the referees has pointed out a further conclusion implicit in the above argument: 

\begin{corollary}
On the universal $W_g$-bundle $\pi \colon E \to B\mathrm{Tor}(W_g, D^2)$ the class $e(T_\pi E)^2$ is non-zero.
\end{corollary}

\begin{proof}
For a non-zero $v \in V_1$ the class $\kappa_{e^2}(v) \in H^3(B\mathrm{Tor}(W_g, D^2);\bQ)$ is non-zero: by construction this is given by applying the Gysin map to the class $e(T_\pi E)^2 \cdot \iota(v)$, so in particular $e(T_\pi E)^2 \neq 0$.
\end{proof}

\subsection{Ring structure} 

Let us now use the results of Section \ref{sec:GraphicalCalc} to compute the algebraic part of $H^*(B\mr{Tor}(W_{g}, D^2); \bQ)$ in the stable range, \emph{assuming the conjecture that these cohomology groups are finite dimensional in a range of degrees tending to infinity with $g$}. 

In this case $H^*(B;\bQ) = H^*(B\mr{SO}(2);\bQ)= \bQ[e]$. Combining Theorem \ref{thm:MainCalc} and Theorem \ref{thm:SmallRingStr} we see that the ring $H^*(B\mr{Tor}(W_{g}, D^2); \bQ)^\mr{alg}$ is generated by $\kappa_{e^i}$ for $i \geq 2$, $\kappa_{e^i}(v_1)$ for $i \geq 1$, and $\kappa_1(v_1 \otimes v_2 \otimes v_3)$. By relation ($\beta$), $\kappa_{e^i}$ is decomposable for $i \geq 2$ so can be eliminated from the generators. By relation ($\gamma$), $\kappa_{e^i}(v_1)$ is decomposable for $i \geq 2$, so can be eliminated from the generators.  By relation ($\delta$), $\kappa_e(v_1) = \kappa_1(v_1 \otimes \omega)$ so this can also be eliminated from the generators. This leaves just the classes $\kappa_1(v_1 \otimes v_2 \otimes v_3)$ as generators. By relation ($\alpha$) these provide a copy of the graded representation $\Lambda^3V_1[1]$, so there is a surjection
\[\Lambda^*[\Lambda^3V_1[1]] \lra H^*(B\mr{Tor}(W_{g}, D^2); \bQ)^\mr{alg}.\]
The relations ($\epsilon$) span a certain subspace
\[V_{2^2} + V_{1^2} + V_0 \leq \Lambda^2(\Lambda^3V_1) = V_{1^6} + 2 V_{1^4} +3 V_{1^2} + 2 V_0 + V_{2^2,1^2} + V_{2^2} + V_{2,1^2}\]
described in Lemma \ref{lem:RepThy}. The induced map
\[\frac{\Lambda^*[\Lambda^3V_1[1]]}{(\text{the relations ($\epsilon$)}, \kappa_{\cL_1}, \kappa_{\cL_2}, \ldots)} \lra H^*(B\mr{Tor}(W_{g}, D^2); \bQ)^\mr{alg}\]
is an isomorphism, which may be seen as follows:

As in the proof of Proposition \ref{prop:low-degree-graphical}, injectivity may be checked by tensoring with $V_1^{\otimes S}$ and taking $\mr{Sp}_{2g}(\bZ)$-invariants for all finite sets $S$. Note that $\kappa_{\cL_i} = c_i \kappa_{e^{2i}}$ for a non-zero scalar $c_i$. Using the graphical formalism of the proof of Theorems \ref{thm:BigRingStr} and \ref{thm:SmallRingStr} the left-hand side is given by the space of trivalent graphs with orientation data, and legs in bijection with $S$, modulo ($\epsilon$) and the graphs containing a connected component with no legs and even first Betti number (these are the $\kappa_{e^{2i}}$'s). The right-hand side is given by partitions of $S$ with parts labelled by $e^i$'s, where parts of size zero cannot be labelled by a $e^{i}$ for $i$ even or 1, and parts of size 1 cannot be labelled by $e^0$. The map is given by sending a graph to the induced partition of $S$ given by connected components of the graph, and a part is given label $e^i$ if the first Betti number of the connected graph corresponding to that part is $i$. Given these descriptions it is easy to see the map is injective as in the profs of Theorems \ref{thm:BigRingStr} and \ref{thm:SmallRingStr}.

\begin{remark}
Adding $\kappa_{\cL_1}$ to the relations ($\epsilon$) gives
a certain subspace $V_{2^2} + V_{1^2} + 2V_0 \leq \Lambda^2(\Lambda^3V_1)$ where all summands apart from $V_{1^2}$ are unambiguous, and under the decomposition $\Lambda^3V_1 = V_{1^3} \oplus V_1$ the copy of $V_{1^2}$ is such that it has nontrivial projection to both $\Lambda^2(V_{1^3})$ and $\Lambda^2 (V_1)$. The quadratic (graded) commutative algebra 
\begin{equation}\label{eq:quad-dual}\frac{\Lambda^*[\Lambda^3V_1[1]]}{(2V_0 + V_{1^2} + V_{2^2})}\end{equation}
is precisely the quadratic dual of the quadratic presentation obtained by Hain \cite{HainTorelli} (see Habegger--Sorger \cite{HabeggerSorger} for this case) of the Mal'cev Lie algebra $\mathfrak{t}_{g,1}$ associated to the group $T_{g,1} \coloneqq \pi_0(\mr{Tor}(W_{g}, D^2))$. If the Lie algebra $\mathfrak{t}_{g,1}$ is Koszul, its continuous Lie algebra cohomology is given by \eqref{eq:quad-dual}, and this is also the cohomology of the Mal'cev completion $\smash{\widehat{T}}_{g,1}$. Thus the natural map
\[H^*(\widehat{T}_{g,1};\bQ) \lra H^*(T_{g,1};\bQ)^\mr{alg}\]
would be surjective with kernel the ideal $(\kappa_{\cL_2}, \kappa_{\cL_3}, \ldots)$, in a stable range.
\end{remark}

\section{Explicit ranges} \label{sec:ranges} 

The ranges of cohomological degrees in which our results apply come from three places:
\begin{enumerate}[\indent (i)]
	\item homological stability results for the $B\diff(W_{g}, D^{2n})$,
	\item the Borel vanishing theorem, i.e.\ Theorem \ref{thm.borelvanishingweak},
	\item the stability range for the invariant theory of $\mr{Sp}_{2g}(\bC)$ and $\mr{O}_{2g}(\bC)$.
\end{enumerate}
We expect that the currently known ranges for (i) are likely not optimal, so we have preferred not to state a particular range in our results. In this section we explain the ranges which can be deduced from the current state of the art.

\subsection{Ranges for Theorem \ref{thm:Iso}} 
In the proof of Theorem \ref{thm:Iso}, homological stability results and stable homology computations for $B\mr{Diff}^{\theta \times Y}(W_{g}, D^{2n} ; \ell_{D^{2n}})_{\ell_g}$ are used. In particular, we used that there is a map
\[\alpha \colon B\mr{Diff}^{\theta \times Y}(W_{g}, D^{2n} ; \ell_{D^{2n}})_{\ell_g} \lra \Omega^\infty_0(\mathrm{MT}\theta \wedge Y_+)\]
which is an isomorphism on cohomology in range of degrees tending to infinity with $g$, which can be found in \cite{grwstab1,grwstab2} for $2n \geq 6$ and \cite{boldsen,oscarresolutions} for $2n=2$. The case that is used in the remainder of the paper is that of $\theta \colon B\mr{SO}(2n)\langle n \rangle \to B\mr{SO}(2n)$. In this case, the known ranges are $* \leq \frac{g-3}{2}$ when $2n \geq 6$ and $* \leq \frac{2g-2}{3}$ when $2n=2$; these will also be the ranges for Theorem \ref{thm:Iso}.

\subsection{Ranges for Theorem \ref{thm:MainCalc}} 
The Borel vanishing theorem and its consequences are used in the proof of Theorem \ref{thm:MainCalc}, which also relies on Theorem \ref{thm:Iso}. Explicit ranges for Theorem \ref{thm.borelvanishingweak} already appeared in Borel's original papers, but in Theorem \ref{thm.borelvanishingweak} we have given an improved version which was stated in \cite{HainTorelli} without proofs, and proven in \cite{tshishikuBorel} (this is likely optimal \cite{tshishiku}). This range is linear of slope 1 in $g$, which is larger than the range for Theorem \ref{thm:Iso}. Thus the range in Theorem \ref{thm:MainCalc} is $* \leq \frac{g-3}{2}$ for $2n \geq 6$ (and $* \leq \frac{2g-2}{3}$ for $2n=2$ in the range in which it applies).

\subsection{Ranges for Theorem \ref{thm:Calc-2d}} 
The first part of Theorem \ref{thm:Calc-2d} relies on Theorem \ref{thm:MainCalc}, and also uses Johnson's computation of $H^1$ of the Torelli group as input. To get the maximal algebraic subrepresentation of $H^2$, in addition to requiring $g \geq 3$ for Johnson's result we also need $g \geq 3$ for Borel vanishing in degrees $\leq 2$, and $g \geq 4$ to get the input from Theorem \ref{thm:Iso}. The conclusion is that the first part of Theorem \ref{thm:Calc-2d} holds for $g \geq 4$.

\subsection{Ranges for Proposition \ref{prop:TransfOfTransf} and Corollary \ref{cor:TransfDetectsZero}}\label{sec:PropTransfOfTransf}
The range in Proposition \ref{prop:TransfOfTransf} has a slightly different source, namely the range given by the first and second fundamental theorems of invariant theory for $\OO_{g,g}$ and $\Sp_{2g}$ as we have encoded in Theorem \ref{thm:FTInvariantTheory}. In fact, the range in Theorem \ref{thm:FTInvariantTheory} for $\Sp_{2g}$ is $2g \geq |S|$ as we have stated, but can be improved to $4g \geq |S|$ for $\OO_{g,g}$. This may be deduced from a careful reading of Section 11.6.3 of \cite{Procesi}. Thus when $n$ is even the conclusion of Proposition \ref{prop:TransfOfTransf} can be relaxed to ``when evaluated on sets $S$ with $|S| \leq 4g-N+1$". There is a similar modest improvement to Corollary \ref{cor:TransfDetectsZero}.

\subsection{Ranges for Theorem \ref{thm:BigRingStr}}\label{sec:SupportEstimate} The source of the map $\phi \colon R^\cV_\mr{pres} \to R^\cV$ is generated by the classes $\kappa_c(v_1 \otimes \cdots \otimes v_r)$ of degree $|c| + n(r-2)>0$. Such classes are detected by applying $[- \otimes H(g)^{\otimes r}]^{G'_g}$: let us say they have \emph{weight} $r$. In the cases of interest the space of labels has the form $\cV = \bQ\{1\} \oplus \cV_{>n}$. The smallest homological degree for such classes of weight $\leq 1$ is therefore 1, for weight 2 is $n$, and for weight $r \geq 3$ is $n(r-2)$. As $n \geq 3$, it follows that in homological degree $d$ the kernel $\mr{Ker}(\phi)$ has weight $\leq d$, i.e.~it vanishes if and only if in this degree $[\phi \otimes H(g)^{\otimes S}]^{G'_g}$ is injective for all sets $S$ with $|S| \leq d$.

The proof of Theorem \ref{thm:BigRingStr} uses Proposition \ref{prop:TransfOfTransf} to determine the range of $S$'s and homological degrees in which
\[\mathcal{P}(S, \cV)_{\geq 0} \otimes  {\det(\bQ^S)}^{\otimes  n} \lra [R^\cV \otimes H(g)^{\otimes S}]^{G'_g},\]
is an isomorphism. By a similar count to the above, in degree $d$ the functor $\mathcal{P}(-, \cV)_{\geq 0}'$ vanishes on sets $T$ with $|T| \geq d+1$, and so by Proposition \ref{prop:TransfOfTransf} this map is an isomorphism in degree $d$ as long as $|S| \leq 2g-d$. By the discussion in the previous paragraph we only need it to be an isomorphism for $|S| \leq d$, so in total need $d \leq g$.

Thus for Theorem \ref{thm:BigRingStr} to hold in degrees $* \leq d$ it is enough that $g \geq d$. Combining it with the discussion in Section \ref{sec:PropTransfOfTransf}, if $n$ is even it is enough that $2g \geq d$.

This discussion also makes explicit the range in Proposition \ref{prop:LsAreRegular}.

\bibliographystyle{amsalpha}
\bibliography{cell}

\vspace{.5cm}

\end{document}